\documentclass[10pt]{amsart}

\usepackage{amssymb,amsmath,amsthm}
\usepackage[all]{xy}
\usepackage{eucal}

\theoremstyle{plain}
\newtheorem{thm}[subsection]{Theorem}
\newtheorem{cor}[subsection]{Corollary}

\newtheorem{prop}[subsection]{Proposition}
\newtheorem{assumption}[subsection]{Basic Assumption}
\newtheorem{homotopical_assumption}[subsection]{Homotopical Assumption}
\newtheorem{cofibrancy}[subsection]{Cofibrancy Condition}

\theoremstyle{definition}
\newtheorem{defn}[subsection]{Definition}

\theoremstyle{remark}
\newtheorem{rem}[subsection]{Remark}

\newtheorem{calculation}[subsection]{Calculation}

\makeatletter 
\let\c@equation\c@subsection 
 
\makeatother

\newcommand{\CC}{{ \mathsf{C} }}
\newcommand{\DD}{{ \mathsf{D} }}

\newcommand{\ZZ}{{ \mathbb{Z} }}

\newcommand{\Ho}{{ \mathsf{Ho} }}

\newcommand{\sSet}{{ \mathsf{S} }}
\newcommand{\ssSet}{{ \mathsf{sS} }}
\newcommand{\Chaincx}{{ \mathsf{Ch} }}

\newcommand{\ModR}{{ \mathsf{Mod}_\capR }}
\newcommand{\sModR}{{ \mathsf{sMod}_\capR }}

\newcommand{\Spectra}{{ \mathsf{Sp}^\Sigma }}
\newcommand{\Operad}{{ \mathsf{Op} }}

\newcommand{\M}{{ \mathsf{M} }}

\newcommand{\SymArray}{{ \mathsf{SymArray} }}
\newcommand{\SymSeq}{{ \mathsf{SymSeq} }}

\newcommand{\sSymSeq}{{ \mathsf{sSymSeq} }}

\newcommand{\Set}{{ \mathsf{Set} }}
\newcommand{\Rt}{{ \mathsf{Rt} }}
\newcommand{\Lt}{{ \mathsf{Lt} }}
\newcommand{\Alg}{{ \mathsf{Alg} }}

\newcommand{\sLt}{{ \mathsf{sLt} }}
\newcommand{\sAlg}{{ \mathsf{sAlg} }}
\newcommand{\LL}{{ \mathsf{L} }}
\newcommand{\RR}{{ \mathsf{R} }}
\newcommand{\unit}{{ {\mathcal{K}} }}

\newcommand{\Bi}{{ \mathsf{Bi} }}
\newcommand{\Cube}{{ \mathsf{Cube} }}
\newcommand{\pCube}{{ \mathsf{pCube} }}
\newcommand{\QQ}{{ \mathsf{Q} }}
\newcommand{\TQ}{{ \mathsf{TQ} }}
\newcommand{\K}{{ \mathsf{K} }}

\newcommand{\AlgO}{{ \Alg_\capO }}
\newcommand{\LtO}{{ \Lt_\capO }}
\newcommand{\RtO}{{ \Rt_\capO }}
\newcommand{\sAlgO}{{ \sAlg_\capO }}
\newcommand{\sLtO}{{ \sLt_\capO }}

\newcommand{\capO}{{ \mathcal{O} }}
\newcommand{\capR}{{ \mathcal{R} }}
\newcommand{\capA}{{ \mathcal{A} }}

\newcommand{\Ev}{{ \mathrm{Ev} }}
\newcommand{\ev}{{ \mathrm{ev} }}
\newcommand{\inmap}{{ \mathrm{in} }}
\newcommand{\id}{{ \mathrm{id} }}
\newcommand{\op}{{ \mathrm{op} }}
\newcommand{\pr}{{ \mathrm{pr} }}

\newcommand{\HH}{{ \mathsf{h} }}
\newcommand{\hwedge}{{ \HH\wedge }}
\newcommand{\Smash}{{ \,\wedge\, }}

\newcommand{\ol}[1]{{ \overline{{#1}} }}
\newcommand{\tensor}{{ \otimes }}

\newcommand{\tensorcheck}{{ \check{\tensor} }}
\newcommand{\tensortilde}{{ \tilde{\tensor} }}
\newcommand{\tensordot}{{ \dot{\tensor} }}

\newcommand{\wequiv}{{ \ \simeq \ }}

\newcommand{\Iso}{{  \ \cong \ }}
\newcommand{\Equal}{{ \ = \ }}
\newcommand{\rarrow}{{ \longrightarrow }}
\newcommand{\larrow}{{ \longleftarrow }}

\newcommand{\circtilde}{{ \,\tilde{\circ}\, }}

\newcommand{\functor}[3]{{ {#1}\colon\thinspace{#2}\rarrow{#3} }}
\newcommand{\function}[3]{{ {#1}\colon\thinspace{#2}\rarrow{#3} }}
\newcommand{\functionsub}[3]{{ {#1}\colon{#2}\rightarrow{#3} }}

\newcommand{\subsetof}{{ \ \subset\ }}

\DeclareMathOperator{\hocolim}{hocolim}

\DeclareMathOperator{\colim}{colim}
\DeclareMathOperator{\holim}{holim}

\DeclareMathOperator{\Map}{Map}
\DeclareMathOperator{\End}{End}
\DeclareMathOperator{\BAR}{Bar}

\DeclareMathOperator{\U}{U}

\DeclareMathOperator{\Tor}{Tor}
\DeclareMathOperator{\Hombold}{\mathbf{Hom}}

\title[Homotopy completion and topological {Q}uillen homology]{Homotopy completion and topological {Q}uillen homology of structured ring spectra}

\author{John E. Harper}
\author{Kathryn Hess}

\address{Department of Mathematics, University of Western Ontario, London, Ontario, N6A 5B7, Canada}
\address{MATHGEOM, \'Ecole Polytechnique F\'ed\'erale de Lausanne, CH-1015 Lausanne, Switzerland}
\email{john.edward.harper@gmail.com}

\address{MATHGEOM, \'Ecole Polytechnique F\'ed\'erale de Lausanne, CH-1015 Lausanne, Switzerland}
\email{kathryn.hess@epfl.ch}

\begin{document}

\begin{abstract}  
Working in the context of symmetric spectra, we describe and study a homotopy completion tower for algebras and left modules over operads in the category of modules over a commutative ring spectrum (e.g., structured ring spectra). We prove a strong convergence theorem that for $0$-connected algebras and modules over a $(-1)$-connected operad, the homotopy completion tower interpolates (in a strong sense) between topological Quillen homology and the identity functor. 

By systematically exploiting strong convergence, we prove several  theorems concerning the topological Quillen homology of algebras and modules over operads. These include a theorem relating finiteness properties of topological Quillen homology groups and homotopy groups that can be thought of as a spectral algebra analog of Serre's finiteness theorem for spaces and H.R. Miller's boundedness result for simplicial commutative rings (but in reverse form). We also prove absolute and relative Hurewicz theorems and a corresponding Whitehead theorem for topological Quillen homology. Furthermore, we prove a rigidification theorem, which we use to describe completion with respect to topological Quillen homology (or $\TQ$-completion). The $\TQ$-completion construction can be thought of as a spectral algebra analog of Sullivan's localization and completion of
spaces, Bousfield-Kan's completion of spaces with respect to homology,
and Carlsson's and Arone-Kankaanrinta's completion and localization of spaces with respect to stable
homotopy. We prove analogous results for algebras and left modules over operads in unbounded chain complexes.
\end{abstract}

\maketitle

\section{Introduction}

Associated to each non-unital commutative ring $X$ is the completion tower arising in commutative ring theory
\begin{align}
\label{eq:NUCA_completion_tower}
  X/X^2 \leftarrow X/X^3 \leftarrow \cdots \leftarrow
  X/X^{n} \leftarrow X/X^{n+1} \leftarrow \cdots
\end{align}
of non-unital commutative rings. The limit of the tower \eqref{eq:NUCA_completion_tower} is the completion $X^\wedge$ of $X$, which is sometimes also called the $X$-adic completion of $X$. Here, $X/X^n$ denotes the quotient of $X$ in the underlying category by the image of the multiplication map $X^{\otimes n}\rarrow X$. In algebraic topology, algebraic $K$-theory, and derived algebraic geometry, it is common to encounter objects that are naturally equipped with algebraic structures more general than, for example, commutative rings, but that share certain formal similarities with these classical algebraic structures. A particularly useful and interesting class of such generalized algebraic structures are those that can be described as algebras and modules over operads; see Fresse \cite{Fresse_lie_theory}, Goerss-Hopkins \cite{Goerss_Hopkins_moduli_spaces}, Kriz-May \cite{Kriz_May}, Mandell \cite{Mandell}, and McClure-Smith \cite{McClure_Smith_conjecture}. 

These categories of (generalized) algebraic structures can often be equipped with an associated homotopy theory, or Quillen model category structure, which allows one to construct and calculate derived functors on the associated homotopy category. In \cite[II.5]{Quillen}, Quillen defines ``homology'' in the general context of a model category---now called Quillen homology---to be the left derived functor of abelianization, if it exists. Quillen homology often behaves very much like the ordinary homology of topological spaces, which it recovers as a special case. Quillen \cite{Quillen_rings} and Andr\'e \cite{Andre} originally developed and studied a particular case of Quillen's notion of homology for the special context of commutative rings, now called Andr\'e-Quillen homology. A useful introduction to Quillen homology is given in Goerss-Schemmerhorn \cite{Goerss_Schemmerhorn}; see also Goerss \cite{Goerss_f2_algebras} and H.R. Miller \cite{Miller} for a useful development (from a homotopy viewpoint) in the case of augmented commutative algebras.

In this paper we are primarily interested in the topological analog of Quillen homology, called topological Quillen homology, for (generalized) algebraic structures on spectra. The topological analog for commutative ring spectra, called topological Andr\'e-Quillen homology, was originally studied by Basterra \cite{Basterra}; see also Baker-Gilmour-Reinhard \cite{Baker_Gilmour_Reinhard}, Baker-Richter \cite{Baker_Richter}, Basterra-Mandell \cite{Basterra_Mandell, Basterra_Mandell_thh}, Goerss-Hopkins \cite{Goerss_Hopkins}, Lazarev \cite{Lazarev}, Mandell \cite{Mandell_TAQ}, Richter \cite{Richter}, Rognes \cite{Rognes_topological_Galois, Rognes_logarithmic} and Schwede \cite{Schwede_cotangent, Schwede_algebraic}.

\begin{assumption}
\label{assumption:commutative_ring_spectrum}
From now on in this paper, we assume that $\capR$ is any commutative ring spectrum; i.e., we assume that $\capR$ is any commutative monoid object in the category $(\Spectra,\tensor_S,S)$ of symmetric spectra \cite{Hovey_Shipley_Smith, Schwede_book_project}. Here, the tensor product $\tensor_S$ denotes the usual smash product \cite[2.2.3]{Hovey_Shipley_Smith} of symmetric spectra (Remark \ref{rem:smash_product_and_tensor_product}).
\end{assumption}

\begin{rem}
Among \emph{structured ring spectra} we include many different types of algebraic structures on spectra (resp. $\capR$-modules) including (i) associative ring spectra, which we simply call ring spectra, (ii) commutative ring spectra, (iii) all of the $E_n$ ring spectra for $1\leq n\leq \infty$ that interpolate  between these two extremes of non-commutativity and commutativity, together with (iv) any generalized algebra spectra (resp. generalized $\capR$-algebras)  that can be described as algebras over operads in spectra (resp. $\capR$-modules). It is important to note that the generalized class of algebraic structures in (iv) includes as special cases all of the others (i)--(iii). The area of stable homotopy theory that focuses on problems arising from constructions involving different types of structured ring spectra, their modules, and their homotopy invariants, is sometimes called \emph{brave new algebra} or \emph{spectral algebra}.
\end{rem}

In this paper we describe and study a (homotopy invariant) spectral algebra analog of the completion tower \eqref{eq:NUCA_completion_tower} arising in commutative ring theory. The tower construction is conceptual and provides a  sequence of refinements of the Hurewicz map for topological Quillen homology. More precisely, if $\capO$ is an operad in $\capR$-modules such that $\capO[\mathbf{0}]$ is trivial (i.e., $\capO$-algebras are \emph{non-unital}), we associate to $\capO$ itself a tower
\begin{align*}
  \tau_1\capO \leftarrow 
  \tau_2\capO \leftarrow \cdots \leftarrow
  \tau_{k-1}\capO \leftarrow 
  \tau_{k}\capO \leftarrow \cdots
\end{align*}
of $(\capO,\capO)$-bimodules, which for any $\capO$-algebra $X$ induces the \emph{completion} tower
\begin{align*}
  \tau_1\capO\circ_\capO(X) \leftarrow 
  \tau_2\capO\circ_\capO(X) \leftarrow \cdots \leftarrow
  \tau_{k-1}\capO\circ_\capO(X) \leftarrow 
  \tau_{k}\capO\circ_\capO(X) \leftarrow \cdots
\end{align*}
of $\capO$-algebras whose limit is the \emph{completion} $X^\wedge$ of $X$. There is a homotopy theory of algebras over operads (Theorem \ref{thm:positive_flat_stable_AlgO}) and this construction is homotopy invariant if applied to cofibrant $\capO$-algebras. We sometimes refer to the completion tower of a cofibrant replacement $X^c$ of $X$ as the \emph{homotopy completion tower} of $X$ whose homotopy limit is denoted $X^\hwedge$. By construction, $\tau_1\capO\circ_\capO(X^c)$ is the topological Quillen homology $\TQ(X)$ of $X$. Hence the homotopy completion tower of $X$ interpolates between $\TQ(X)$, which is the bottom term of the tower, and the homotopy completion $X^\hwedge$ of $X$. 

By systematically exploiting the strong convergence properties of this tower (Theorem \ref{MainTheorem} and its proof), we prove a selection of theorems concerning the topological Quillen homology of structured ring spectra. We also prove analogous results for left modules over operads (Definition \ref{defn:algebras_and_modules}). The first main theorem in this paper is the following finiteness theorem for topological Quillen homology. It can be thought of as a structured ring spectra analog of Serre's finiteness theorem for spaces (e.g., for the homotopy groups of spheres) and H.R. Miller's \cite[4.2]{Miller} boundedness result for simplicial commutative rings (but in reverse form); for a related but different type of finiteness result in the algebraic context of augmented commutative algebras over a field of non-zero characteristic, see Turner \cite{Turner}. The $\TQ$ finiteness theorem provides conditions under which topological Quillen homology detects certain finiteness properties.

\begin{rem}
In this paper, we say that a symmetric sequence $X$ of symmetric spectra is $n$-connected if each symmetric spectrum $X[\mathbf{t}]$ is $n$-connected. We say that an algebra (resp. left module) over an operad is $n$-connected if the underlying symmetric spectrum (resp. symmetric sequence of symmetric spectra) is $n$-connected, and similarly for operads.
\end{rem}

\begin{thm}[$\TQ$ finiteness theorem for structured ring spectra]
\label{thm:finiteness}
Let $\capO$ be an operad in $\capR$-modules such that $\capO[\mathbf{0}]$ is trivial. Let $X$ be a $0$-connected $\capO$-algebra (resp. left $\capO$-module) and assume that $\capO,\capR$ are $(-1)$-connected and $\pi_k\capO[\mathbf{r}],\pi_k\capR$ are finitely generated abelian groups for every $k,r$.
\begin{itemize}
\item[(a)] If the topological Quillen homology groups $\pi_k\TQ(X)$ (resp. $\pi_k\TQ(X)[\mathbf{r}]$) are finite for every $k,r$, then the homotopy groups $\pi_k X$ (resp. $\pi_k X[\mathbf{r}]$) are finite for every $k,r$. 
\item[(b)] If the topological Quillen homology groups $\pi_k\TQ(X)$ (resp. $\pi_k\TQ(X)[\mathbf{r}]$) are finitely generated abelian groups for every $k,r$, then the homotopy groups $\pi_k X$ (resp. $\pi_k X[\mathbf{r}]$) are finitely generated abelian groups for every $k,r$. 
\end{itemize}
\end{thm}

Since the sphere spectrum $S$ is $(-1)$-connected and $\pi_kS$ is a finitely generated abelian group for every $k$, we obtain the following immediate corollary.

\begin{cor}[$\TQ$ finiteness theorem for non-unital commutative ring spectra]
\label{cor:finiteness_commutative_ring_spectra}
Let $X$ be a $0$-connected non-unital commutative ring spectrum. If the topological Quillen homology groups $\pi_k\TQ(X)$ are finite (resp. finitely generated abelian groups) for every $k$, then the homotopy groups $\pi_k X$ are finite (resp. finitely generated abelian groups) for every $k$.  
\end{cor}

\begin{rem}
Since all of the theorems in this section apply to the special case of non-unital commutative ring spectra, it follows that each theorem below specializes to a corollary about non-unital commutative ring spectra, similar to the corollary above. To avoid repetition, we usually leave the formulation to the reader.
\end{rem}

We also prove the following Hurewicz theorem for topological Quillen homology. It can be thought of as a structured ring spectra analog of Schwede's \cite[5.3]{Schwede_algebraic} simplicial algebraic theories result, Goerss' \cite[8.3]{Goerss_f2_algebras} algebraic result for augmented commutative $\mathbb{F}_2$-algebras, Livernet's \cite[2.13]{Livernet} rational algebraic result for algebras over operads in non-negative chain complexes over a field of characteristic zero, and Chataur-Rodriguez-Scherer's \cite[2.1]{Chataur_Rodriguez_Scherer} algebraic result for algebras over cofibrant operads in non-negative chain complexes over a commutative ring. The $\TQ$ Hurewicz theorem provides conditions under which topological Quillen homology detects $n$-connected structured ring spectra. It also provides conditions under which the first non-trivial homotopy group agrees via the Hurewicz map with the first non-trivial topological Quillen homology group.

\begin{thm}[$\TQ$ Hurewicz theorem for structured ring spectra]
\label{thm:hurewicz}
Let $\capO$ be an operad in $\capR$-modules such that $\capO[\mathbf{0}]$ is trivial. Let $X$ be a $0$-connected $\capO$-algebra (resp. left $\capO$-module), $n\geq 0$, and assume that $\capO,\capR$ are $(-1)$-connected.
\begin{itemize}
\item[(a)] Topological Quillen homology $\TQ(X)$ is $n$-connected if and only if $X$ is $n$-connected.
\item[(b)] If topological Quillen homology $\TQ(X)$ is $n$-connected, then the natural Hurewicz map
$
  \pi_k X\rarrow\pi_k\TQ(X)
$
is an isomorphism for $k\leq 2n+1$ and a surjection for $k=2n+2$.
\end{itemize}
\end{thm}

Note that one implication of Theorem \ref{thm:hurewicz}(a) follows from Theorem \ref{thm:hurewicz}(b). We also prove the following relative Hurewicz theorem for topological Quillen homology, which we regard as the second main theorem in this paper. It can be thought of as a structured ring spectra analog of the relative Hurewicz theorem for spaces. It provides conditions under which topological Quillen homology detects $n$-connected maps.

\begin{thm}[$\TQ$ relative Hurewicz theorem for structured ring spectra]
\label{thm:relative_hurewicz}
Let $\capO$ be an operad in $\capR$-modules such that $\capO[\mathbf{0}]$ is trivial. Let $\function{f}{X}{Y}$ be a map of $\capO$-algebras (resp. left $\capO$-modules) and $n\geq 0$. Assume that $\capO,\capR$ are $(-1)$-connected.
\begin{itemize}
\item[(a)] If $X,Y$ are $0$-connected, then $f$ is $n$-connected if and only if $f$ induces an $n$-connected map $\TQ(X)\rarrow\TQ(Y)$ on topological Quillen homology.
\item[(b)] If $X,Y$ are $(-1)$-connected and $f$ is $(n-1)$-connected, then $f$ induces an $(n-1)$-connected map $\TQ(X)\rarrow\TQ(Y)$ on topological Quillen homology.
\item[(c)] If $f$ induces an $n$-connected map $\TQ(X)\rarrow\TQ(Y)$ on topological Quillen homology between $(-1)$-connected objects, then $f$ induces an $(n-1)$-connected map $X^\hwedge\rarrow Y^\hwedge$ on homotopy completion.
\item[(d)] If topological Quillen homology $\TQ(X)$ is $(n-1)$-connected, then homotopy completion $X^\hwedge$ is $(n-1)$-connected.
\end{itemize}
Here, $\TQ(X)\rarrow\TQ(Y)$, $X^\hwedge\rarrow Y^\hwedge$ denote the natural induced zigzags in the category of $\capO$-algebras (resp. left $\capO$-modules) with all backward facing maps weak equivalences.
\end{thm}

\begin{rem}
\label{rem:preservation_of_connectivity}
It is important to note Theorem \ref{thm:relative_hurewicz}(b) implies that the conditions in Theorem \ref{thm:relative_hurewicz}(c) are satisfied if $X,Y$ are $(-1)$-connected and $f$ is $n$-connected.
\end{rem}

As a corollary we obtain the following Whitehead theorem for topological Quillen homology. It can be thought of as a structured ring spectra analog of Schwede's \cite[5.4]{Schwede_algebraic} simplicial algebraic theories result, Goerss' \cite[8.1]{Goerss_f2_algebras} algebraic result for augmented commutative $\mathbb{F}_2$-algebras, and Livernet's \cite{Livernet_thesis} rational algebraic result for algebras over Koszul operads in non-negative chain complexes over a field of characteristic zero. As a special case, it recovers Kuhn's \cite{Kuhn} result for non-unital commutative ring spectra, and more generally, Lawson's \cite{Lawson} original structured ring spectra result (which is based on \cite{Harper_Bar}). The $\TQ$ Whitehead theorem provides conditions under which topological Quillen homology detects weak equivalences.

\begin{cor}[$\TQ$ Whitehead theorem for structured ring spectra]
\label{cor:whitehead}
Let $\capO$ be an operad in $\capR$-modules such that $\capO[\mathbf{0}]$ is trivial. Let $\function{f}{X}{Y}$ be a map of $\capO$-algebras (resp. left $\capO$-modules). Assume that $\capO,\capR$ are $(-1)$-connected. If $X,Y$ are $0$-connected, then $f$ is a weak equivalence if and only if $f$ induces a weak equivalence 
$
  \TQ(X)\wequiv\TQ(Y)
$
on topological Quillen homology.
\end{cor}

Associated to the homotopy completion tower is the \emph{homotopy completion spectral sequence}, which goes from topological Quillen homology to homotopy completion (Theorem \ref{MainTheorem}). It can be thought of as a structured ring spectra analog of Quillen's fundamental spectral sequence \cite[6.9]{Quillen_rings} for commutative rings and the corresponding spectral sequence studied by Goerss \cite[6.2]{Goerss_f2_algebras} for augmented commutative $\mathbb{F}_2$-algebras. As a special case, it recovers the spectral sequence in Minasian \cite{Minasian} for non-unital commutative ring spectra. Under the conditions of Theorem \ref{MainTheorem}(b), the homotopy completion spectral sequence is a second quadrant homologically graded spectral sequence and arises from the exact couple of long exact sequences associated to the homotopy completion tower and its homotopy fibers; this is the homotopy spectral sequence of a tower of fibrations \cite{Bousfield_Kan}, reindexed as a homologically graded spectral sequence. For ease of notational purposes, in Theorem \ref{MainTheorem} and Remark \ref{rem:strong_convergence}, we regard such towers $\{A_s\}$ of fibrations  as indexed by the integers such that $A_s=*$ for each $s<0$.

The third main theorem in this paper is the following strong convergence theorem for homotopy completion of structured ring spectra. It can be thought of as a structured ring spectra analog of Johnson-McCarthy's \cite{Johnson_McCarthy} rational algebraic tower results for non-unital commutative differential graded algebras over a field of characteristic zero. As a special case, it recovers Kuhn's \cite{Kuhn} and Minasian's \cite{Minasian} tower results for non-unital commutative ring spectra. For a very restricted class of cofibrant operads in simplicial sets, which they call primitive operads, McCarthy-Minasian \cite{McCarthy_Minasian_preprint} describe a tower that agrees with the completion tower in the special case of non-unital commutative ring spectra, but that is different for most operads.

\begin{thm}[Homotopy completion strong convergence theorem]
\label{MainTheorem}
Let $\capO$ be an operad in $\capR$-modules such that $\capO[\mathbf{0}]$ is trivial. Let $\function{f}{X}{Y}$ be a map of $\capO$-algebras (resp. left $\capO$-modules). 
\begin{itemize}
\item[(a)] If $X$ is $0$-connected and $\capO,\capR$ are $(-1)$-connected, then the natural  coaugmentation $X\wequiv X^\hwedge$ is a weak equivalence.
\item[(b)] If topological Quillen homology $\TQ(X)$ is $0$-connected and $\capO,\capR$ are $(-1)$-connected, then the homotopy completion spectral sequence
\begin{align*}
  E^1_{-s,t} &= \pi_{t-s}\Bigl(i_{s+1}\capO\circ^{\HH}_{\tau_1\capO}\bigl(\TQ(X)\bigr)\Bigr)
  \Longrightarrow
  \pi_{t-s}\bigl(X^{\hwedge}\bigr)\\
  \text{resp.}\quad
  E^1_{-s,t}[\mathbf{r}] &= \pi_{t-s}\Bigl(\bigl(i_{s+1}\capO\circ^{\HH}_{\tau_1\capO}\TQ(X)\bigr)[\mathbf{r}]\Bigr)
  \Longrightarrow
  \pi_{t-s}\bigl(X^{\hwedge}[\mathbf{r}]\bigr),\quad\text{$r\geq 0$},
\end{align*}
converges strongly (Remark \ref{rem:strong_convergence}).
\item[(c)] If $f$ induces a weak equivalence $\TQ(X)\wequiv\TQ(Y)$ on topological Quillen homology, then $f$ induces a weak equivalence $X^\hwedge\wequiv Y^\hwedge$ on homotopy completion. 
\end{itemize}
\end{thm}

\begin{rem}
\label{rem:strong_convergence}
By \emph{strong convergence} of $\{E^r\}$ to $\pi_*(X^\hwedge)$ we mean that (i) for each $(-s,t)$, there exists an $r$ such that $E^r_{-s,t}=E^\infty_{-s,t}$ and (ii) for each $i$, $E^\infty_{-s,s+i}=0$ except for finitely many $s$. Strong convergence implies that for each $i$, $\{E^\infty_{-s,s+i}\}$ is the set of filtration quotients from a finite filtration of $\pi_i(X^\hwedge)$; see, for instance, Bousfield-Kan \cite[IV.5.6, IX.5.3, IX.5.4]{Bousfield_Kan} and Dwyer \cite{Dwyer_strong_convergence}.
\end{rem}

\begin{rem}[Connections with {G}oodwillie's calculus of functors]
Regard the homotopy completion tower as a tower of functors on the category of $\capO$-algebras, and consider the case when $\capO[\mathbf{1}]=I[\mathbf{1}]$ (Definition \ref{defn:operad}). Then it follows easily that (i) the bottom term (or first stage) $\TQ$ of the tower is $1$-excisive in the sense of \cite{Goodwillie_calc3, Kuhn_survey}, (ii) by Theorem \ref{thm:calculating_fiber_of_induced_map}(c), the $n$-th layer of the tower has the form $\capO[\mathbf{n}]\Smash^\LL_{\Sigma_n}\TQ^{\wedge^\LL n}$, and (iii) by the connectivity estimates in the proof of Theorem \ref{thm:hurewicz}, the identity functor and the $n$-th stage of the tower \emph{agree to order $n$} in the sense of \cite[1.2]{Goodwillie_calc3}; more precisely, they satisfy $O_n(0,1)$ as defined in \cite[1.2]{Goodwillie_calc3}. Here, $\wedge^\LL_{\Sigma_n}$, $\wedge^\LL$ are the total left derived functors of $\wedge_{\Sigma_n}$, $\wedge$, respectively. Properties (i)--(iii) illustrate that the homotopy completion tower is the analog, in the context of $\capO$-algebras, of  Goodwillie's Taylor tower of the identity functor. More precisely, according to \cite[1.6, proof of 1.8]{Goodwillie_calc3} and the results in \cite{Goodwillie_calc2} on cubical diagrams, it follows immediately from (i)--(iii) that there are maps of towers (under the constant tower $\{\id(-)^c\}$) of levelwise weak equivalences of the form 
$\{P_n\id(-)^c\}\rightarrow\{P_n\tau_n\capO\circ_\capO(-)^c\}\leftarrow\{\tau_n\capO\circ_\capO(-)^c\}$ where $(-)^c$ denotes functorial cofibrant replacement (see Definition \ref{defn:homotopy_completion}), and hence the homotopy completion tower is weakly equivalent to the Taylor tower of the identity functor on $\capO$-algebras, provided that the analogs of the appropriate constructions and results in \cite{Goodwillie_calc2, Goodwillie_calc3} remain true in the category of $\capO$-algebras; this is the subject of current work, and will not be further elaborated here (but see \cite{Kuhn_survey}). 

Since in the calculation of the layers in (ii) the operad $\capO$ plays a role analogous to that of the Goodwillie derivatives of the identity functor (see \cite{Goodwillie_calc3, Kuhn_survey}), this sheds some positive light on a conjecture of Arone-Ching \cite{Arone_Ching} that an appropriate model of the Goodwillie derivatives of the identity functor on $\capO$-algebras is weakly equivalent as an operad to $\capO$ itself. 
\end{rem}

The following relatively weak cofibrancy condition is exploited in the proofs of the main theorems above. The statements of these theorems do not require this cofibrancy condition since a comparison theorem (Theorem \ref{thm:comparing_homotopy_completion_towers}, Proposition \ref{prop:replacement_of_operads}) shows that the operad $\capO$ can always be replaced by a weakly equivalent operad $\capO'$ that satisfies this cofibrancy condition and such that the corresponding homotopy completion towers are naturally weakly equivalent.

\begin{cofibrancy}
\label{CofibrancyCondition}
If $\capO$ is an operad in $\capR$-modules, consider the unit map $\function{\eta}{I}{\capO}$ of the operad $\capO$ (Definition \ref{defn:operad}) and assume that  $I[\mathbf{r}]\rarrow\capO[\mathbf{r}]$ is a flat stable cofibration (Subsection \ref{sec:model_structures_on_capR_modules}) between flat stable cofibrant objects in $\ModR$ for each $r\geq 0$.
\end{cofibrancy}

\begin{rem}
This is the same as assuming that $I[\mathbf{1}]\rarrow\capO[\mathbf{1}]$ is a flat stable cofibration in $\ModR$ and $\capO[\mathbf{r}]$ is flat stable cofibrant in $\ModR$ for each $r\geq 0$. It can be thought of as the structured ring spectra analog of the following cofibrancy condition: if $X$ is a pointed space, assume that $X$ is well-pointed; i.e., assume that the unique map $*\rightarrow X$ of pointed spaces is a cofibration.
\end{rem}

Most operads appearing in homotopy theoretic settings in mathematics already satisfy Cofibrancy Condition \ref{CofibrancyCondition} and therefore require no replacement in the proofs of the theorems. For instance, Cofibrancy Condition \ref{CofibrancyCondition} is satisfied by every operad in  simplicial sets that is regarded as an operad in $\capR$-modules via adding a disjoint basepoint and tensoring with $\capR$ (Subsection \ref{sec:simplicial_bar_and_homotopy_completion}). 

In this paper, the homotopy groups $\pi_*Y$ of a symmetric spectrum $Y$  denote the \emph{derived} homotopy groups (or true homotopy groups) \cite{Schwede_book_project, Schwede_homotopy_groups}; i.e., $\pi_*Y$ always denotes the homotopy groups of a stable fibrant replacement of $Y$, and hence of a flat stable fibrant replacement of $Y$. See Schwede \cite{Schwede_homotopy_groups} for several useful properties enjoyed by the true homotopy groups of a symmetric spectrum.

\subsection{Organization of the paper}

In Section \ref{sec:preliminaries} we recall some preliminaries on algebras and modules over operads. The purpose of Section \ref{sec:homotopy_completion} is to describe homotopy completion (Definition \ref{defn:homotopy_completion}) and $\TQ$-completion, or less concisely, completion with respect to topological Quillen homology (Definition \ref{defn:quillen_homology_completion}) and to establish a comparison theorem for homotopy completion towers (Theorem \ref{thm:comparing_homotopy_completion_towers}). In Section \ref{sec:homotopical_analysis_of_the_completion_tower} we prove our main theorems, which involves a homotopical analysis of the completion tower. We establish several necessary technical results on the homotopical properties of the forgetful functors in Section \ref{sec:homotopical_analysis_forgetful_functors}, and on simplicial structures and the homotopical properties of the simplicial bar constructions in Section \ref{sec:homotopical_analysis_bar_constructions}. The results in these two sections lie at the heart of the proofs of the main theorems. The purpose of Section \ref{sec:model_structures} is to improve the main results in \cite{Harper_Spectra, Harper_Bar} on model structures, homotopy colimits and simplicial bar constructions from the context of operads in symmetric spectra to the more general context of operads in $\capR$-modules. This amounts to establishing certain technical propositions for $\capR$-modules sufficient for the proofs of the main results in \cite{Harper_Spectra, Harper_Bar} to remain valid in the more general context of $\capR$-modules; these results play a key role in this paper. In Section \ref{sec:chain_complexes_over_a_commutative_ring} we observe that the analogs of the main theorems stated above remain true in the context of unbounded chain complexes over a commutative ring.

\subsection*{Acknowledgments}

The authors would like to thank Greg Arone, Michael Ching, Bill Dwyer, Emmanuel Farjoun, Rick Jardine, Nick Kuhn, Haynes Miller, and Stefan Schwede for useful suggestions and remarks and Kristine Bauer, Mark Behrens, Bjorn Dundas, Benoit Fresse, Paul Goerss, Tom Goodwillie, Jens Hornbostel, Brenda Johnson, Tyler Lawson, Muriel Livernet, Ib Madsen, Mike Mandell, Randy McCarthy, Jack Morava, and Charles Rezk for helpful comments. The first author is grateful to Jens Hornbostel and Stefan Schwede for a stimulating and enjoyable visit to the Mathematisches Institut der Universit\"at Bonn in summer 2010, and to Mark Behrens and Haynes Miller for a stimulating and enjoyable visit to the Department of Mathematics at the Massachusetts Institute of Technology in summer 2011, and for their invitations which made this possible. The authors would like to thank the anonymous referee for his or her detailed suggestions and comments, which have resulted in a significant improvement.

\section{Preliminaries}
\label{sec:preliminaries}

The purpose of this section is to recall various preliminaries on algebras and modules over operads. In this paper the following two contexts will be of primary interest. Denote by $(\ModR,\Smash,\capR)$ the closed symmetric monoidal category of $\capR$-modules (Basic Assumption \ref{assumption:commutative_ring_spectrum}, Remark \ref{rem:dropping_the_adjective_left}), and by $(\Chaincx_\unit,\tensor,\unit)$ the closed symmetric monoidal category of unbounded chain complexes over $\unit$ \cite{Hovey, MacLane_homology}; here, $\unit$ is any commutative ring. Both categories have all small limits and colimits, and the null object is denoted by $*$. It will be useful in this paper, both for  establishing certain results and for ease of notational purposes, to sometimes work in the following more general context; see \cite[VII]{MacLane_categories} followed by \cite[VII.7]{MacLane_categories}.

\begin{assumption}
From now on in this section we assume that $(\CC,\Smash,S)$ is a closed symmetric monoidal category with all small limits and colimits. In particular, $\CC$ has an initial object $\emptyset$ and a terminal object $*$.
\end{assumption}

By \emph{closed} we mean there exists a functor $\CC^\op\times\CC\rarrow \CC: (Y,Z)\longmapsto \Map(Y,Z)$, which we call the \emph{mapping object}, which fits into
$
   \hom(X\Smash Y,Z)\Iso \hom(X,\Map(Y,Z))
$
isomorphisms natural in $X,Y,Z$, where $\hom$ denotes the set of morphisms in $\CC$. Define the sets $\mathbf{n}:=\{1,\dots,n\}$ for each $n\geq 0$, where $\mathbf{0}:=\emptyset$ denotes the empty set. If $T$ is a finite set, we denote by $|T|$ the number of elements in $T$.

\begin{defn}
\label{defn:symmetric_sequences}
Let $n\geq 0$.
\begin{itemize}
\item $\Sigma$ is the category of finite sets and their bijections. 
\item A \emph{symmetric sequence} in $\CC$ is a functor $\functor{A}{\Sigma^{\op}}{\CC}$. Denote by $\SymSeq$ the category of symmetric sequences in $\CC$ and their natural transformations. 
\item A symmetric sequence $A$ is \emph{concentrated at $n$} if $A[\mathbf{r}]=\emptyset$ for all $r\neq n$.
\end{itemize}
\end{defn}

For a more detailed development of the material that follows, see \cite{Harper_Spectra, Harper_Modules}.

\begin{defn}
Consider symmetric sequences in $\CC$. Let $A_1,\dotsc,A_t\in\SymSeq$. Their \emph{tensor product} $A_1\tensorcheck\dotsb\tensorcheck A_t\in\SymSeq$ is the left Kan extension of objectwise smash along coproduct of sets
\begin{align*}
\xymatrix{
  (\Sigma^{\op})^{\times t}
  \ar[rr]^-{A_1\times\dotsb\times A_t}\ar[d]^{\coprod} & &
  \CC^{\times t}\ar[r]^-{\Smash} & \CC \\
  \Sigma^{\op}\ar[rrr]^{A_1\tensorcheck\dotsb\tensorcheck
  A_t}_{\text{left Kan extension}} & & & \CC
}
\end{align*}
\end{defn}

If $X$ is a finite set and $A$ is an object in $\CC$, we use the usual dot notation $A\cdot X$ (\cite{MacLane_categories}, \cite[2.3]{Harper_Modules}) to denote the copower $A\cdot X$ defined by
$
  A\cdot X := \coprod_X A
$,
the coproduct in $\CC$ of $|X|$ copies of $A$. Recall the following useful calculations for tensor products.

\begin{prop}
Consider symmetric sequences in $\CC$. Let $A_1,\dotsc,A_t\in\SymSeq$ and $R\in\Sigma$, with $r:=|R|$. There are natural isomorphisms
\begin{align}
  \notag
  (A_1\tensorcheck\dotsb\tensorcheck A_t)[R]&\Iso\ 
  \coprod_{\substack{\function{\pi}{R}{\mathbf{t}}\\ \text{in $\Set$}}}
  A_1[\pi^{-1}(1)]\Smash\dotsb\Smash
  A_t[\pi^{-1}(t)],\\
  \label{eq:tensor_check_calc}
  &\Iso
  \coprod_{r_1+\dotsb +r_t=r}A_1[\mathbf{r_1}]\Smash\dotsb\Smash 
  A_t[\mathbf{r_t}]\underset{{\Sigma_{r_1}\times\dotsb\times
  \Sigma_{r_t}}}{\cdot}\Sigma_{r}
\end{align}
\end{prop}

Here, $\Set$ is the category of sets and their maps, and \eqref{eq:tensor_check_calc} displays the tensor product $(A_1\tensorcheck\dotsb\tensorcheck A_t)[R]$ as a coproduct of $\Sigma_{r_1}\times\dotsb\times\Sigma_{r_t}$-orbits. It will be conceptually useful to extend the definition of tensor powers $A^{\tensorcheck t}$ to situations in which the integers $t$ are replaced by a finite set $T$.

\begin{defn}
Consider symmetric sequences in $\CC$. Let $A\in\SymSeq$ and $R,T\in\Sigma$. The \emph{tensor powers} $A^{\tensorcheck T}\in\SymSeq$ are defined objectwise by
\begin{align*}
  (A^{\tensorcheck\emptyset})[R]:= 
  \coprod_{\substack{\function{\pi}{R}{\emptyset}\\ \text{in $\Set$}}}
  S,\quad\quad
  &(A^{\tensorcheck T})[R]:= 
  \coprod_{\substack{\function{\pi}{R}{T}\\ \text{in $\Set$}}}
  \bigwedge_{t\in T} A[\pi^{-1}(t)]\quad(T\neq\emptyset).
\end{align*}
Note that there are no functions $\function{\pi}{R}{\emptyset}$ in $\Set$ unless $R=\emptyset$. We will use the abbreviation $A^{\tensorcheck 0}:=A^{\tensorcheck\emptyset}$.  
\end{defn}

\begin{defn}\label{defn:circle_product}
Consider symmetric sequences in $\CC$. Let $A,B,C\in\SymSeq$, and $r,t\geq 0$. The \emph{circle product} (or composition product) $A\circ B\in\SymSeq$ is defined objectwise by the coend
\begin{align}
  \label{eq:circle_product_calc}
  (A\circ B)[\mathbf{r}] := A\Smash_\Sigma (B^{\tensorcheck-})[\mathbf{r}]
  &\Iso 
  \coprod_{t\geq 0}A[\mathbf{t}]\Smash_{\Sigma_t}
  (B^{\tensorcheck t})[\mathbf{r}].
\end{align}
The \emph{mapping sequence} $\Map^\circ(B,C)\in\SymSeq$ and the \emph{mapping object} $\Map^\tensorcheck(B,C)\in\SymSeq$ are defined objectwise by the ends
\begin{align*}
  \Map^\circ(B,C)[\mathbf{t}] &:= \Map((B^{\tensorcheck \mathbf{t}})[-],C)^\Sigma \Iso 
  \prod_{r\geq 0}\Map((B^{\tensorcheck \mathbf{t}})[\mathbf{r}],
  C[\mathbf{r}])^{\Sigma_r},\\
  \Map^\tensorcheck(B,C)[\mathbf{t}] &:= \Map(B,C[\mathbf{t}\amalg -])^\Sigma \Iso
  \prod_{r\geq 0}\Map(B[\mathbf{r}],C[\mathbf{t}
  \boldsymbol{+}\mathbf{r}])^{\Sigma_r}.
\end{align*}
\end{defn}

These mapping sequences and mapping objects fit into isomorphisms
\begin{align}
  \label{eq:circle_mapping_sequence_adjunction}
  \hom(A\circ B,C)&\Iso\hom(A,\Map^\circ(B,C)),\\
  \label{eq:tensorcheck_mapping_sequence_adjunction}
  \hom(A\tensorcheck B,C)&\Iso\hom(A,\Map^\tensorcheck(B,C)),
\end{align}
natural in symmetric sequences $A,B,C$. Here, the $\hom$ notation denotes the indicated set of morphisms in $\SymSeq$.

\begin{prop}
\label{prop:closed_monoidal_on_symmetric_sequences}
Consider symmetric sequences in $\CC$. 
\begin{itemize}
\item [(a)] $(\SymSeq,\tensorcheck,1)$ has the structure of a closed symmetric monoidal category with all small limits and colimits. The unit for $\tensorcheck$ denoted ``$1$'' is the symmetric sequence concentrated at $0$ with value $S$.
\item [(b)] $(\SymSeq,\circ,I)$ has the structure of a closed monoidal category with all small limits and colimits. The unit for $\circ$ denoted ``$I$'' is the symmetric sequence concentrated at $1$ with value $S$. Circle product is not symmetric.
\end{itemize}
\end{prop}

\begin{defn}
\label{defn:hat_construction_embed_at_zero}
Let $Z\in\CC$. Define $\hat{Z}\in\SymSeq$ to be the symmetric sequence concentrated at $0$ with value $Z$.
\end{defn}

The functor $\function{\hat{-}}{\CC}{\SymSeq}$ fits into the adjunction
$
\xymatrix@1{
  \hat{-}\colon\CC\ar@<0.5ex>[r] & 
  \SymSeq:\Ev_0\ar@<0.5ex>[l]
}
$
with left adjoint on top and $\Ev_0$ the \emph{evaluation} functor defined objectwise by $\Ev_0(B):=B[\mathbf{0}]$. Note that $\hat{-}$ embeds $\CC$ in $\SymSeq$ as the full subcategory of symmetric sequences concentrated at $0$.

\begin{defn}\label{defn:corresponding_functor}
Consider symmetric sequences in $\CC$. Let $\capO$ be a symmetric sequence and $Z\in\CC$. The corresponding functor $\functor{\capO}{\CC}{\CC}$ is defined objectwise by
$
  \capO(Z):=\capO\circ(Z):=\amalg_{t\geq 0}\capO[\mathbf{t}]
  \Smash_{\Sigma_t}Z^{\wedge t}.
$
\end{defn}

\begin{prop}
Consider symmetric sequences in $\CC$. Let $\capO,A\in\SymSeq$ and $Z\in\CC$. There are natural isomorphisms
\begin{align}
\label{eq:circ_product_and_evaluate_at_zero}
  \widehat{\capO(Z)}=
  \widehat{\capO\circ(Z)}\Iso\capO\circ\hat{Z},\quad\quad
  \Ev_0(\capO\circ A)\Iso \capO\circ\bigl(\Ev_0(A)\bigr).
\end{align}
\end{prop}

\begin{proof}
This follows from \eqref{eq:circle_product_calc} and \eqref{eq:tensor_check_calc}.
\end{proof}

\begin{defn}
\label{defn:operad}
Consider symmetric sequences in $\CC$. An \emph{operad} in $\CC$ is a monoid object in $(\SymSeq,\circ,I)$ and a \emph{morphism of operads} is a morphism of monoid objects in $(\SymSeq,\circ,I)$.
\end{defn}

\begin{rem} If $\capO$ is an operad, then the associated functor $\capO: \CC\to \CC$ is a monad.
\end{rem}

\begin{defn}
\label{defn:algebras_and_modules}
Let $\capO$ be an operad in $\CC$.
\begin{itemize}
\item A \emph{left $\capO$-module} is an object in $(\SymSeq,\circ,I)$ with a left action of $\capO$ and a \emph{morphism of left $\capO$-modules} is a map that respects the left $\capO$-module structure. Denote by $\LtO$ the category of left $\capO$-modules and their morphisms.
\item A \emph{right $\capO$-module} is an object in $(\SymSeq,\circ,I)$ with a right action of $\capO$ and a \emph{morphism of right $\capO$-modules} is a map  that respects the right $\capO$-module structure.  Denote by $\RtO$ the category of right $\capO$-modules and their morphisms.
\item An \emph{$(\capO,\capO)$-bimodule} is an object in $(\SymSeq,\circ,I)$  with compatible left $\capO$-module and right $\capO$-module structures and a \emph{morphism of $(\capO,\capO)$-bimodules} is a map that respects the $(\capO,\capO)$-bimodule structure. Denote by $\Bi_{(\capO,\capO)}$ the category of $(\capO,\capO)$-bimodules and their morphisms.
\item An \emph{$\capO$-algebra} is an algebra for the monad $\functor{\capO}{\CC}{\CC}$ and a \emph{morphism of $\capO$-algebras} is a map in $\CC$ that respects the $\capO$-algebra structure. Denote by $\AlgO$ the category of $\capO$-algebras and their morphisms.
\end{itemize}
\end{defn}

It follows easily from \eqref{eq:circ_product_and_evaluate_at_zero} that an $\capO$-algebra is the same as an object $Z$ in $\CC$ with a left $\capO$-module structure on $\hat{Z}$, and if $Z$ and $Z'$ are $\capO$-algebras, then a morphism of $\capO$-algebras is the same as a map $\function{f}{Z}{Z'}$ in $\CC$ such that $\function{\hat{f}}{\hat{Z}}{\hat{Z'}}$ is a morphism of left $\capO$-modules. In other words, an algebra over an operad $\capO$ is the same as a left $\capO$-module that is concentrated at $0$, and $\AlgO$ embeds in $\LtO$ as the full subcategory of left $\capO$-modules concentrated at $0$, via the functor $\function{\hat{-}}{\AlgO}{\LtO}$, $Z\longmapsto \hat{Z}$. Define the \emph{evaluation} functor $\function{\Ev_0}{\LtO}{\AlgO}$ objectwise by $\Ev_0(B):=B[\mathbf{0}]$.

\begin{prop}
\label{prop:basic_properties_LTO}
Let $\capO$ be an operad in $\CC$. There are adjunctions
\begin{align}
\label{eq:free_forgetful_adjunction}
\xymatrix{
  \CC\ar@<0.5ex>[r]^-{\capO\circ(-)} & \AlgO,\ar@<0.5ex>[l]^-{U}
}\quad\quad
\xymatrix{
  \SymSeq\ar@<0.5ex>[r]^-{\capO\circ-} & \LtO,\ar@<0.5ex>[l]^-{U}
}\quad\quad
\xymatrix{
  \AlgO\ar@<0.5ex>[r]^-{\hat{-}} & \LtO,\ar@<0.5ex>[l]^-{\Ev_0}
}
\end{align}
with left adjoints on top and $U$ the forgetful functor. All small colimits exist in $\AlgO$ and $\LtO$, and both reflexive coequalizers and filtered colimits are preserved (and created) by the forgetful functors. All small limits exist in $\AlgO$ and $\LtO$, and are preserved (and created) by the forgetful functors.
\end{prop}

\begin{defn}
Consider symmetric sequences in $\CC$. Let $\DD$ be a small category, and let $X,Y\in\SymSeq^\DD$. Denote by $\Map^\circ(X,Y)$ the indicated composition of functors $\DD^\op\times\DD\rarrow\SymSeq$. The \emph{mapping sequence} of $\DD$-shaped diagrams is defined by the end $\Map^\circ(X,Y)^\DD\in\SymSeq$. \end{defn}

 By the universal property of ends, it follows easily that for all $\capO\in\SymSeq$,  there are isomorphisms
\begin{align}
\label{eq:circle_adjunction_for_diagrams}
  \hom_\DD\bigl(\capO\circ X,Y)&\Iso\hom(\capO,\Map^\circ(X,Y)^\DD\bigr)
\end{align}
natural in $\capO,X,Y$ and that $\Map^\circ(X,Y)^\DD$ may be calculated by an equalizer in $\SymSeq$ of the form
\begin{align*}
  \Map^\circ(X,Y)^\DD\Iso
  \lim
  \biggl( 
  \xymatrix{
  \prod_{\alpha\in\DD}\limits
  \Map^\circ(X_\alpha,Y_\alpha)\ar@<2.0ex>[r]\ar@<1.0ex>[r] & 
  \prod_{(\functionsub{\xi}{\alpha}{\alpha'})\in\DD}\limits 
  \Map^\circ(X_\alpha,Y_{\alpha'})
  }
  \biggr).
\end{align*}
Here, $\capO\circ X$ denotes the indicated composition of functors $\DD\rarrow\SymSeq$, the $\hom_\DD$ notation on the left-hand side of \eqref{eq:circle_adjunction_for_diagrams} denotes the indicated set of morphisms in $\SymSeq^\DD$, and the $\hom$ notation on the right-hand side of \eqref{eq:circle_adjunction_for_diagrams} denotes the indicated set of morphisms in $\SymSeq$.

\begin{defn}
Let $\DD$ be a small category and $X\in\CC^\DD$ (resp. $X\in\SymSeq^\DD$) a $\DD$-shaped diagram. The \emph{endomorphism operad} $\End(X)$ of $X$ is defined by 
\begin{align*}
  \End(X):=\Map^\circ(\hat{X},\hat{X})^\DD
\quad\quad
\Bigl(\text{resp.}\quad
  \End(X):=\Map^\circ(X,X)^\DD
\Bigr)
\end{align*}
with its natural operad structure; i.e., such that for each $\alpha\in\DD$, the natural map $\End(X)\rarrow\Map^\circ(\hat{X}_\alpha,\hat{X}_\alpha)$ (resp. $\End(X)\rarrow\Map^\circ(X_\alpha,X_\alpha)$) is a morphism of operads.
\end{defn}

Let $X$ be a $\DD$-shaped diagram in $\CC$ (resp. $\SymSeq$). It follows easily from \eqref{eq:circle_mapping_sequence_adjunction} and \eqref{eq:circle_adjunction_for_diagrams} that giving a map of operads $\function{m}{\capO}{\End(X)}$ is the same as giving $X_\alpha$ an $\capO$-algebra structure (resp. left $\capO$-module structure) for each $\alpha\in\DD$, such that $X$ is a diagram of $\capO$-algebras (resp. left $\capO$-modules). Note that if $\DD$ is the terminal category (with exactly one object and no non-identity morphisms), then $\End(X)\Iso\Map^\circ(\hat{X},\hat{X})$ (resp. $\End(X)\Iso\Map^\circ(X,X)$), which recovers the usual endomorphism operad of an object $X$ in $\CC$ (resp. $\SymSeq$) \cite{Harper_Modules, Kriz_May}.

\section{Homotopy completion and $\TQ$-completion}
\label{sec:homotopy_completion}

The purpose of this section is to describe two notions of completion for structured ring spectra: (i) homotopy completion (Definition \ref{defn:homotopy_completion}) and (ii) $\TQ$-completion, or less concisely, completion with respect to topological Quillen homology (Definition \ref{defn:quillen_homology_completion}). We will also establish a rigidification theorem for derived $\TQ$-resolutions (Theorem \ref{thm:rigidification}), which is required to define $\TQ$-completion, and we will prove Theorem \ref{thm:comparing_homotopy_completion_towers} which compares homotopy completion towers along a map of operads.

Let $\function{f}{\capO}{\capO'}$ be a map of operads in $\capR$-modules. Recall that the change of operads adjunction
\begin{align}
\label{eq:quillen_adjunction_change_of_operads_nice}
\xymatrix{
  \Alg_{\capO}\ar@<0.5ex>[r]^-{f_*} & \Alg_{\capO'}\ar@<0.5ex>[l]^-{f^*}
}
\quad\quad
\Bigl(\text{resp.}\quad
\xymatrix{
  \Lt_{\capO}\ar@<0.5ex>[r]^-{f_*} & \Lt_{\capO'}\ar@<0.5ex>[l]^-{f^*}
}
\Bigr)
\end{align}
is a Quillen adjunction with left adjoint on top and $f^*$ the forgetful functor (more accurately, but less concisely, also called the ``restriction along $f$ of the operad action'') \cite{Harper_Spectra, Harper_Modules}; note that this is a particular instance of the usual change of monoids adjunction.

\begin{rem}
In this paper we always regard $\AlgO$ and $\LtO$ with the positive flat stable model structure (Theorem \ref{thm:positive_flat_stable_AlgO}), unless otherwise specified.
\end{rem}

\begin{defn}
Let $\function{f}{\capO}{\capO'}$ be a map of operads in $\capR$-modules. Let $X$ be an $\capO$-algebra (resp. left $\capO$-module) and define the $\capO$-algebra $\capO'\circ^\HH_\capO (X)$ (resp. left $\capO$-module $\capO'\circ^\HH_\capO X$) by
\begin{align*}
  \capO'\circ^\HH_\capO (X):=\RR f^*(\LL f_* (X)) = 
  \RR f^* \bigl(\capO'\circ^\LL_\capO (X)\bigr)\\
  \Bigl(\text{resp.}\quad
  \capO'\circ^\HH_\capO X:=\RR f^*(\LL f_* (X)) = 
  \RR f^* (\capO'\circ^\LL_\capO X)
  \Bigr).
\end{align*}
Here, $\RR f^*,\LL f_*$ are the total right (resp. left) derived functors of $f^*,f_*$, respectively. 
\end{defn}

\begin{rem}
\label{rem:compute-htpy-orbits}
Note that $\Alg_I=\ModR$ and $\Lt_I=\SymSeq$  (since $I$ is the initial operad) and that for any map of operads $\function{f}{\capO}{\capO'}$, there are weak equivalences
\begin{align*}
  \capO'\circ^\HH_\capO (X)\wequiv \LL f_* (X) = 
  \capO'\circ^\LL_\capO (X)\quad\quad
  \Bigl(\text{resp.}\quad
  \capO'\circ^\HH_\capO X\wequiv \LL f_* (X) = 
  \capO'\circ^\LL_\capO X
  \Bigr)
\end{align*}
in the underlying category $\Alg_I$ (resp. $\SymSeq$), natural in $X$; this follows from the property that the forgetful functor to the underlying category preserves weak equivalences.
\end{rem}

The \emph{truncation} functor $\function{\tau_k}{\SymSeq}{\SymSeq}$ is defined objectwise by  
\begin{align*}
  (\tau_k X)[\mathbf{r}]:=
  \left\{
  \begin{array}{rl}
    X[\mathbf{r}],&\text{for $r\leq k$,}\\
    *,&\text{otherwise},
  \end{array}
  \right.
\end{align*}
for each $k\geq 1$. In other words, $\tau_k X$ is the symmetric sequence obtained by truncating $X$ above level $k$. Let $\capO$ be an operad in $\capR$-modules such that $\capO[\mathbf{0}]=*$. It is easy to verify that the canonical map of operads $\capO\rarrow\tau_1\capO$ factors through each truncation $\tau_k\capO$, and hence gives rise to a commutative diagram of operads
\begin{align}
\label{eq:towers_of_operads}
\xymatrix{
  \{\tau_k\capO\}: &
  \tau_1\capO & \tau_2\capO\ar[l] & \tau_3\capO\ar[l] & 
  \dotsb \ar[l] & & \\
  \{\capO \}:\ar@<0.5ex>[u] &
  \ \ \ \capO\ar@<-1.5ex>[u]\ar@/_0.5pc/[ur]\ar@/_0.5pc/[urr]
  \ar@{}[urrrrr]_(.4){\dotsb}
}
\end{align}
and $(\capO,\capO)$-bimodules. In other words, associated to each such operad $\capO$ is a coaugmented tower $\{\capO\}\rarrow\{\tau_k\capO\}$ of operads and $(\capO,\capO)$-bimodules, where $\{\capO\}$ denotes the constant tower with value $\capO$. This tower underlies the following definition of completion for $\capO$-algebras and left $\capO$-modules, which plays a key role in this paper.

\begin{rem}
Let $\capO$ be an operad in $\capR$-modules such that $\capO[\mathbf{0}]=*$.
\begin{itemize}
\item[(i)] The canonical maps $\tau_1\capO\rarrow\capO\rarrow\tau_1\capO$ of operads factor the identity map.
\item[(ii)] Note that $\capO[\mathbf{0}]=*$ and $\capO[\mathbf{1}]=I[\mathbf{1}]$ if and only if $\tau_1\capO=I$, i.e., if and only if the operad $\capO$ agrees with the initial operad $I$ at levels $0$ and $1$.
\end{itemize}
\end{rem}

\begin{defn}
Let $\capO$ be an operad in $\capR$-modules such that $\capO[\mathbf{0}]=*$. Let $X$ be an $\capO$-algebra (resp. left $\capO$-module). The \emph{completion tower} of $X$ is the coaugmented tower of $\capO$-algebras (resp. left $\capO$-modules)
\begin{align}
\label{eq:towers_of_algebras}
  \{X\}\rarrow\{\tau_k\capO\circ_\capO (X)\}
  \quad\quad
  \Bigl(\text{resp.}\quad
  \{X\}\rarrow\{\tau_k\capO\circ_\capO X\}
  \Bigr)
\end{align}
obtained by applying $-\circ_\capO (X)$ (resp. $-\circ_\capO X$) to the coaugmented tower \eqref{eq:towers_of_operads}. The \emph{completion} $X^\wedge$ of $X$ is the $\capO$-algebra (resp. left $\capO$-module) defined by 
\begin{align}
\label{eq:completion_of_algebras}
  X^\wedge:=\lim\nolimits^\AlgO_k \bigl(\tau_k\capO\circ_\capO (X)\bigr)
  \quad\quad
  \Bigl(\text{resp.}\quad
  X^\wedge:=\lim\nolimits^\LtO_k \bigl(\tau_k\capO\circ_\capO X\bigr)
  \Bigr),
  \end{align}
 i.e., the limit of the completion tower of $X$. Here, $\{X\}$ denotes the constant tower with value $X$. Thus, completion defines a coaugmented functor on $\AlgO$ (resp. $\LtO$).
\end{defn}

\begin{rem}
We often suppress the forgetful functors $\Alg_{\tau_k\capO}\rarrow\AlgO$ and $\Lt_{\tau_k\capO}\rarrow\LtO$ from the notation, as in \eqref{eq:towers_of_algebras}.
\end{rem}

\subsection{Homotopy completion and topological {Q}uillen homology}
The purpose of this subsection is to introduce homotopy completion (Definition \ref{defn:homotopy_completion}) and topological Quillen homology (Definition \ref{defn:quillen_homology}).

In this paper we will primarily be interested in a homotopy invariant version of the completion functor, which involves the following homotopy invariant version of the limit functor on towers.

\begin{defn}
\label{defn:model_structure_on_towers}
Let $\M$ be a model category with all small limits and let $\DD$ be the category $\{0\leftarrow 1\leftarrow 2\leftarrow\cdots\}$ with objects the non-negative integers and a single morphism $i\leftarrow j$ for each $i\leq j$. Consider the category $\M^\DD$ of $\DD$-shaped diagrams (or towers) in $\M$ with the injective model structure \cite[VI.1.1]{Goerss_Jardine}. The \emph{homotopy limit} functor $\function{\holim}{\Ho({\M^\DD})}{\Ho(\M)}$ is the total right derived functor of the limit functor $\function{\lim}{\M^\DD}{\M}$.
\end{defn}

We are now in a good position to define homotopy completion.

\begin{defn}
\label{defn:homotopy_completion}
Let $\capO$ be an operad in $\capR$-modules such that $\capO[\mathbf{0}]=*$. Let $X$ be an $\capO$-algebra (resp. left $\capO$-module). The \emph{homotopy completion} $X^\hwedge$ of $X$ is the $\capO$-algebra (resp. left $\capO$-module) defined by 
\begin{align*}
  X^\hwedge:=\holim^\AlgO_k \bigl(\tau_k\capO\circ_\capO (X^c)\bigr)
  \quad\quad
  \Bigl(\text{resp.}\quad
  X^\hwedge:=\holim^\LtO_k \bigl(\tau_k\capO\circ_\capO X^c\bigr)
  \Bigr),
\end{align*}
the homotopy limit of the completion tower of the functorial cofibrant replacement $X^c$ of $X$ in $\AlgO$ (resp. $\LtO$). 
\end{defn}

\begin{rem}
It is easy to check that if $X$ is a cofibrant $\capO$-algebra (resp. cofibrant left $\capO$-module), then the weak equivalence  $X^c\rarrow X$ induces zigzags of weak equivalences
\begin{align*}
  X^\hwedge&\wequiv\holim^\AlgO_k \bigl(\tau_k\capO\circ_\capO (X)\bigr)\wequiv
  \holim^\AlgO_k \bigl(\tau_k\capO\circ^\HH_\capO (X)\bigr)\\
  \Bigl(\text{resp.}\quad
  X^\hwedge&\wequiv\holim^\LtO_k \bigl(\tau_k\capO\circ_\capO X\bigr)\wequiv
  \holim^\LtO_k \bigl(\tau_k\capO\circ^\HH_\capO X\bigr)
  \Bigr)
\end{align*}
in $\AlgO$ (resp. $\LtO$), natural in $X$. Hence the homotopy completion $X^\hwedge$ of a cofibrant $\capO$-algebra (resp. cofibrant left $\capO$-module) $X$ may be calculated by taking the homotopy limit of the completion tower of $X$.
\end{rem}

In this paper we consider topological Quillen homology of an $\capO$-algebra (resp. left $\capO$-module) as an object in $\AlgO$ (resp. $\LtO$) via the forgetful functor as follows. 

\begin{defn}
\label{defn:quillen_homology} 
If $\capO$ is an operad in $\capR$-modules such that $\capO[\mathbf{0}]=*$, and $X$ is an $\capO$-algebra (resp. left $\capO$-module), then the \emph{topological Quillen homology} $\TQ(X)$ of $X$ is the $\capO$-algebra (resp. left $\capO$-module) $\tau_1\capO\circ^\HH_\capO (X)$ (resp. $\tau_1\capO\circ^\HH_\capO X$).
\end{defn}

In particular, when applied to a cofibrant $\capO$-algebra (resp. cofibrant left $\capO$-module) $X$, the completion tower interpolates between  topological Quillen homology $\TQ(X)$ and homotopy completion $X^\hwedge$.

\subsection{$\TQ$-completion}
\label{subsec:TQ_completion}

The purpose of this subsection is to introduce a second naturally occurring notion of completion for structured ring spectra, called $\TQ$-completion, or less concisely, completion with respect to topological Quillen homology (Definition \ref{defn:quillen_homology_completion}). Defining $\TQ$-completion requires the construction of a rigidification of the derived $\TQ$-resolution \eqref{eq:derived_QH_resolution} from a diagram in the homotopy category to a diagram in the model category. This rigidification problem is solved in Theorem \ref{thm:rigidification}.

The $\TQ$-completion construction is conceptual and can be thought of as a spectral algebra analog of Sullivan's \cite{Sullivan_MIT_notes, Sullivan_genetics} localization and completion of spaces, Bousfield-Kan's \cite[I.4]{Bousfield_Kan} completion of spaces with respect to homology, and Carlsson's \cite[II.4]{Carlsson_equivariant} and Arone-Kankaanrinta's \cite[0.1]{Arone_Kankaanrinta} completion and localization of spaces with respect to stable homotopy. 

Here is the idea behind the construction. We want to define $\TQ$-completion $X_\TQ^\wedge$ of a structured ring spectrum $X$ to be the structured ring spectrum defined by (showing only the coface maps) the homotopy limit of
\begin{align*}
  X^{\wedge}_{\TQ}:=
  \holim\limits_{\Delta}
  \Bigl(
  \xymatrix{
  \TQ(X)\ar@<-0.5ex>[r]\ar@<0.5ex>[r] & 
  (\TQ)^2 (X)
  \ar@<-1.0ex>[r]\ar[r]\ar@<1.0ex>[r] &
  (\TQ)^3 (X)\cdots
  }
  \Bigr)
\end{align*}
the cosimplicial resolution (or Godement resolution) with respect to the monad (or triple) $\TQ$. However, there are technical details that one needs to resolve in order to make sense of this definition for $\TQ$-completion. This is because $\TQ$ naturally arises as a functor on the level of the homotopy categories, and to work with and make sense of the homotopy limit $\holim_\Delta$ we need a point-set level construction of the derived $\TQ$-cosimplicial resolution \eqref{eq:derived_QH_resolution}, or more precisely, a construction on the level of model categories. Successfully resolving this issue is the purpose of the rest of this subsection, and amounts to solving a rigidification problem (Theorem \ref{thm:rigidification}) for the derived cosimplicial resolution with respect to $\TQ$.

Let $\capO$ be an operad in $\capR$-modules such that $\capO[\mathbf{0}]=*$. Then the canonical map of operads $\function{f}{\capO}{\tau_1\capO}$ induces a Quillen adjunction as in \eqref{eq:quillen_adjunction_change_of_operads_nice} and hence induces a corresponding adjunction
\begin{align}
\label{eq:derived_adjunction_quillen_homology}
\xymatrix{
  \Ho(\Alg_{\capO})\ar@<0.5ex>[r]^-{\LL f_*} & \Ho(\Alg_{\tau_1\capO})\ar@<0.5ex>[l]^-{\RR f^*}
}
\quad\quad
\Bigl(\text{resp.}\quad
\xymatrix{
  \Ho(\Lt_{\capO})\ar@<0.5ex>[r]^-{\LL f_*} & \Ho(\Lt_{\tau_1\capO})\ar@<0.5ex>[l]^-{\RR f^*}
}
\Bigr)
\end{align}
on the homotopy categories. Hence topological Quillen homology $\TQ$ is the monad (or triple) on the homotopy category $\Ho(\AlgO)$ (resp. $\Ho(\LtO)$) associated to the derived adjunction \eqref{eq:derived_adjunction_quillen_homology}. Denote by $\K$ the corresponding comonad (or cotriple)
\begin{align*}
  \id\rarrow\TQ&\quad\text{(unit)},\quad\quad\quad
  &\id\larrow\K& \quad\text{(counit)}, \\
  \TQ\TQ\rarrow\TQ&\quad\text{(multiplication)},\quad\quad\quad
  &\K\K\larrow\K& \quad\text{(comultiplication)},
\end{align*}
on $\Ho(\Alg_{\tau_1\capO})$ (resp. $\Ho(\Lt_{\tau_1\capO})$). Then $\TQ=\RR f^*\LL f_*$ and $\K=\LL f_*\RR f^*$, and it follows that for any $\capO$-algebra (resp. left $\capO$-module) X, the adjunction \eqref{eq:derived_adjunction_quillen_homology} determines a cosimplicial resolution of $X$ with respect to topological Quillen homology $\TQ$ of the form
\begin{align}
\label{eq:derived_QH_resolution}
\xymatrix{
  X\ar[r] & 
  \TQ(X)\ar@<-0.5ex>[r]\ar@<0.5ex>[r] & 
  \TQ^2(X)
  \ar@<-1.0ex>[r]\ar[r]\ar@<1.0ex>[r]\ar@<-2.0ex>[l] &
  \TQ^3(X)\cdots\ar@<-2.5ex>[l]\ar@<-3.5ex>[l]
  }
\end{align}
This derived $\TQ$-resolution can be thought of as encoding what it means for $\TQ(X)$ to have the structure of a $\K$-coalgebra. More precisely, the extra structure on $\TQ(X)$ is the $\K$-coalgebra structure on the underlying object $\LL f_*(X)$ of $\TQ(X)$. One difficulty in working with the diagram \eqref{eq:derived_QH_resolution} is that it lives in the homotopy category $\Ho(\AlgO)$ (resp. $\Ho(\LtO)$). The purpose of the rigidification theorem below is to construct a model of \eqref{eq:derived_QH_resolution} that lives in $\AlgO$ (resp. $\LtO$). 

Consider any factorization of the canonical map $\function{f}{\capO}{\tau_1\capO}$ in the category of operads as $\capO\xrightarrow{g}J_1\xrightarrow{h}\tau_1\capO$, a cofibration followed by a weak equivalence (Definition \ref{defn:model-cat-operads}) with respect to the positive flat stable model structure on $\ModR$ (Definition \ref{defn:lets_define_flat_model_structure}); it is easy to verify that such factorizations exist using a small object argument (Proposition \ref{prop:functorial_factorizations_of_maps_of_operads}). The corresponding change of operads adjunctions have the form
\begin{align}
\label{eq:factored_adjunctions}
\xymatrix{
  \AlgO\ar@<0.5ex>[r]^-{g_*} & \Alg_{J_1}\ar@<0.5ex>[l]^-{g^*}\ar@<0.5ex>[r]^-{h_*} & 
  \Alg_{\tau_1\capO}\ar@<0.5ex>[l]^-{h^*}
}\quad
\Bigl(\text{resp.}\quad
\xymatrix{
  \LtO\ar@<0.5ex>[r]^-{g_*} & \Lt_{J_1}\ar@<0.5ex>[l]^-{g^*}\ar@<0.5ex>[r]^-{h_*} & 
  \Lt_{\tau_1\capO}\ar@<0.5ex>[l]^-{h^*}
}
\Bigr)
\end{align}
with left adjoints on top and $g^*,h^*$ the forgetful functors (more accurately, but less concisely, also called the ``restriction along $g,h$, respectively, of the operad action''). These are Quillen adjunctions and since $h$ is a weak equivalence it follows that the $(h_*,h^*)$ adjunction is a Quillen equivalence (Theorem \ref{thm:comparing_homotopy_categories}). We defer the proof of the following rigidification theorem to Section \ref{sec:homotopical_analysis_forgetful_functors} (just after Theorem \ref{thm:cofibration_property_needed_for_homology_completion}).

\begin{thm}[Rigidification theorem for derived $\TQ$-resolutions]
\label{thm:rigidification}
Let $\capO$ be an operad in $\capR$-modules such that $\capO[\mathbf{0}]=*$. Assume that $\capO[\mathbf{r}]$ is flat stable cofibrant in $\ModR$ for each $r\geq 0$. If $X$ is a cofibrant $\capO$--algebra (resp. cofibrant left $\capO$-module) and $n\geq 1$, then there are weak equivalences
$
  (g^*g_*)^n(X)\wequiv\TQ^n(X)
$
natural in such $X$.
\end{thm}

The following description of $\TQ$-completion is closely related to \cite{Carlsson} and \cite{Hess}.

\begin{defn}
\label{defn:quillen_homology_completion}
Let $\capO$ be an operad in $\capR$-modules such that $\capO[\mathbf{0}]=*$. Assume that $\capO[\mathbf{r}]$ is flat stable cofibrant in $\ModR$ for each $r\geq 0$. Let $X$ be an $\capO$-algebra (resp. left $\capO$-module). The \emph{$\TQ$-completion} (or completion with respect to topological Quillen homology) $X_\TQ^\wedge$ of $X$ is the $\capO$--algebra (resp. left $\capO$-module) defined by (showing only the coface maps) the homotopy limit of the cosimplicial resolution
\begin{align}
\label{eq:homology_completion}
  X^{\wedge}_{\TQ}:=
  \holim\limits_{\Delta}
  \Bigl(
  \xymatrix{
  (g^*g_*)(X^c)\ar@<-0.5ex>[r]\ar@<0.5ex>[r] & 
  (g^*g_*)^2(X^c)
  \ar@<-1.0ex>[r]\ar[r]\ar@<1.0ex>[r] &
  (g^*g_*)^3(X^c)\cdots
  }
  \Bigr)
\end{align}
(or Godement resolution) of the functorial cofibrant replacement $X^c$ of $X$ in $\AlgO$ (resp. $\LtO$) with respect to the monad $g^*g_*$. Here, $\holim_\Delta$ is calculated in the category of $\capO$--algebras (resp. left $\capO$-modules).
\end{defn}

\begin{rem}
The $(g^*g_*)$-resolution can be thought of as encoding what it means for $\TQ(X)$ to have the structure of a $\K$-coalgebra. More precisely, the extra structure on $g^*g_*(X^c)\wequiv\TQ(X)$ is the $(g_*g^*)$-coalgebra structure on the underlying object $g_*(X^c)$ of $g^*g_*(X^c)$. In particular, the comonad $(g_*g^*)$ provides a point-set model for the derived comonad $\K$ that coacts on $\TQ(X)$ (up to a Quillen equivalence). This point-set model of $\K$ is conjecturally related to the Koszul dual cooperad associated to $\capO$ (see, for instance, \cite{Ching_duality, Fresse, Ginzburg_Kapranov}).
\end{rem}

It follows that the cosimplicial resolution in \eqref{eq:homology_completion} provides a rigidification of the derived cosimplicial resolution \eqref{eq:derived_QH_resolution}. One of our motivations for introducing the homotopy completion tower was its role as a potentially useful tool in analyzing $\TQ$-completion defined above, but an investigation of these properties and the $\TQ$-completion functor will be the subject of other papers and will not be elaborated here.

\subsection{Comparing homotopy completion towers}

The purpose of this subsection is to prove Theorem \ref{thm:comparing_homotopy_completion_towers}, which compares homotopy completion towers along a map of operads.

Let $\function{g}{\capO'}{\capO}$ be a map of operads in $\capR$-modules, and for each $\capO$-algebra (resp. left $\capO$-module) $X$, consider the corresponding $\capO'$-algebra (resp. left $\capO'$-module) $X$ given by forgetting the left $\capO$-action along the map $g$; here we have dropped the forgetful functor $g^*$ from the notation. Consider the map $\emptyset\rarrow X$ in $\Alg_{\capO'}$ (resp. $\Lt_{\capO'}$) and use functorial factorization in $\Alg_{\capO'}$ (resp. $\Lt_{\capO'}$) to obtain
\begin{align}
\label{eq:functorial_factorization_prime_notation}
  \emptyset\rarrow X'\rarrow X,
\end{align}
a cofibration followed by an acyclic fibration. 

In the next theorem we establish that replacing an operad $\capO$ by a weakly equivalent operad $\capO'$ changes the homotopy completion tower of $X$ only up to natural weak equivalence. In particular, the homotopy completion of $X$ as an $\capO'$-algebra is weakly equivalent to its homotopy completion as an $\capO$-algebra.

\begin{thm}[Comparison theorem for homotopy completion towers]
\label{thm:comparing_homotopy_completion_towers}
Let $\function{g}{\capO'}{\capO}$ be a map of operads in $\capR$-modules such that $\capO'[\mathbf{0}]=*$ and $\capO[\mathbf{0}]=*$. If $X$ is an $\capO$-algebra (resp. left $\capO$-module), then there are maps of towers
\begin{align}
\label{eq:comparing_completion_towers}
\xymatrix{
  \{X'\}\ar[d]\ar@{=}[r] & \{X'\}\ar[d]^{(\sharp)}
  \ar[r] & \{X\}\ar[d]\\
  \{\tau_k\capO'\circ_{\capO'}(X')\}\ar[r]^-{(*)} & 
  \{\tau_k\capO\circ_{\capO'}(X')\}\ar[r]^-{(**)} & 
  \{\tau_k\capO\circ_{\capO}(X)\}
}
\end{align}
\begin{align}
\label{eq:comparing_completion_towers_left_modules}
\text{resp.}\quad
\xymatrix{
  \{X'\}\ar[d]\ar@{=}[r] & \{X'\}\ar[d]^{(\sharp)}
  \ar[r] & \{X\}\ar[d]\\
  \{\tau_k\capO'\circ_{\capO'}X'\}\ar[r]^-{(*)} & 
  \{\tau_k\capO\circ_{\capO'}X'\}\ar[r]^-{(**)} & 
  \{\tau_k\capO\circ_{\capO}X\}
}
\end{align}
of $\capO'$-algebras (resp. left $\capO'$-modules), natural in $X$. If, furthermore, $g$ is a weak equivalence in the underlying category $\SymSeq$, and $X$ is fibrant and cofibrant in $\AlgO$ (resp. $\LtO$), then the maps $(*)$ and $(**)$ are levelwise weak equivalences; here, we are using the notation \eqref{eq:functorial_factorization_prime_notation} to denote functorial cofibrant replacement of $X$ as an $\capO'$-algebra (resp. left $\capO'$-module).
\end{thm}

\begin{proof}
It suffices to consider the case of left $\capO$-modules. The map of operads $\capO'\rarrow\capO$ induces a commutative diagram of towers
\begin{align}
\label{eq:diagram_of_towers_associated_to_map_of_operads}
\xymatrix{
  \{\capO'\}\ar[d]\ar[r] & \{\capO\}\ar[d]\\
  \{\tau_k\capO'\}\ar[r] & \{\tau_k\capO\}
}
\end{align}
of operads and $(\capO',\capO')$-bimodules; here, $\{\capO'\}$ and $\{\capO\}$ denote the constant towers with values $\capO'$ and $\capO$, respectively. 

Consider the map of towers $(*)$. Each map $\tau_k\capO'\circ_{\capO'} X'\rarrow\tau_k\capO\circ_{\capO'}X'$ in $(*)$ is obtained by applying $-\circ_{\capO'}X'$ to the map $\tau_k\capO'\rarrow\tau_k\capO$. By \eqref{eq:diagram_of_towers_associated_to_map_of_operads}, this map is isomorphic to the composite
$
  \tau_k\capO'\circ_{\capO'}X'\xrightarrow{\eta}
  \tau_k\capO\circ_{\tau_k\capO'}\tau_k\capO'\circ_{\capO'}\circ X'
  \Iso
  \tau_k\capO\circ_{\capO'}X'
$
where $\function{\eta}{\id}{\tau_k\capO\circ_{\tau_k\capO'}-}$ is the unit map associated to the change of operads adjunction
$
\xymatrix@1{
  \Lt_{\tau_k\capO'}\ar@<0.5ex>[r] & 
  \Lt_{\tau_k\capO}\ar@<0.5ex>[l]
}
$. If, furthermore, $g$ is a weak equivalence in $\SymSeq$, then the map $\tau_k\capO'\rarrow\tau_k\capO$ is a weak equivalence, and since $X'$ is cofibrant in $\Lt_{\capO'}$ it follows from \ref{thm:comparing_homotopy_categories} and \ref{prop:unit_map_is_weak_equivalence} that $(*)$ is a levelwise weak equivalence. 

Consider the map of towers $(**)$ and the change of operads adjunction 
$
\xymatrix@1{
  \Lt_{\capO'}\ar@<0.5ex>[r] & 
  \Lt_{\capO}\ar@<0.5ex>[l]
}
$. The weak equivalence $X'\rarrow X$ of left $\capO'$-modules in \eqref{eq:functorial_factorization_prime_notation} has corresponding adjoint map
$
\function{\xi}{\capO\circ_{\capO'}X'}{X}
$. Each map $\tau_k\capO\circ_{\capO'}X'\rarrow\tau_k\capO\circ_{\capO}X$ in $(**)$ is obtained by applying $\tau_k\capO\circ_{\capO}-$ to the map $\xi$. If, furthermore, $g$ is a weak equivalence in $\SymSeq$, and $X$ is fibrant and cofibrant in $\LtO$, then by \ref{thm:comparing_homotopy_categories} the map $\xi$ is a weak equivalence between cofibrant objects in $\LtO$, and hence $(**)$ is a levelwise weak equivalence. To finish the proof, it suffices to describe the map of towers $(\sharp)$ in \eqref{eq:comparing_completion_towers_left_modules}. Each map $X'\rarrow\tau_k\capO\circ_{\capO'}X'$ is obtained by applying $-\circ_{\capO'}X'$ to the map $\capO'\rarrow\tau_k\capO$.
\end{proof}

We defer the proof of the following proposition to Section \ref{sec:homotopical_analysis_forgetful_functors}.

\begin{prop}
\label{prop:replacement_of_operads}
Let $\capO$ be an operad in $\capR$-modules such that $\capO[\mathbf{0}]=*$. Then there exists a map of operads $\function{g}{\capO'}{\capO}$ such that $\capO'[\mathbf{0}]=*$, and 
\begin{itemize}
\item[(i)] $g$ is a weak equivalence in the underlying category $\SymSeq$,
\item[(ii)] $\capO'$ satisfies Cofibrancy Condition \ref{CofibrancyCondition}.
\end{itemize}
\end{prop}

Later in this paper, we need the following observation that certain homotopy limits commute with the forgetful functor.

\begin{prop}
\label{prop:forgetful_functor_commutes_with_holim}
Let $\capO$ be an operad in $\capR$-modules. Consider any tower $B_0\leftarrow B_1\leftarrow B_2\leftarrow\cdots$ of $\capO$-algebras (resp. left $\capO$-modules). There are natural zigzags
\begin{align*}
  U\holim^{\AlgO}_k B_k\wequiv \holim_k UB_k
  \quad\quad
\Bigl(\text{resp.}\quad
  U\holim^{\LtO}_k B_k\wequiv \holim_k UB_k
\Bigr)
\end{align*}
of weak equivalences. Here, $U$ is the forgetful functor \eqref{eq:free_forgetful_adjunction}.
\end{prop}

\begin{proof}
This follows from the dual of \cite[proof of 3.15]{Harper_Bar}, together with the observation that the forgetful functor $U$ preserves weak equivalences and that fibrant towers are levelwise fibrant.
\end{proof}

\section{Homotopical analysis of the completion tower}
\label{sec:homotopical_analysis_of_the_completion_tower}

The purpose of this section is to prove the main theorems stated in the introduction (Theorems \ref{thm:finiteness}, \ref{thm:hurewicz}, \ref{thm:relative_hurewicz}, and \ref{MainTheorem}). The unifying approach behind each of these theorems is to systematically exploit induction ``up the homotopy completion tower'' together with explicit calculations of the layers in terms of simplicial bar constructions (Theorem \ref{thm:calculating_fiber_of_induced_map} and Proposition \ref{prop:refined_bar_construction_calculation_for_homotopy_fiber}). An important property of these layer calculations, which we fully exploit in the proofs of the main theorems, is that the simplicial bar constructions are particularly amenable to systematic connectivity and finiteness estimates (Propositions \ref{prop:connectivity_of_simplicial_maps_spectra}, \ref{prop:connectivity}, and \ref{prop:homotopy_spectral_sequence}--\ref{prop:useful_finiteness_properties_dervied_smash}).

The first step to proving the main theorems is to establish conditions under which the homotopy completion tower of $X$ converges strongly to $X$. This is accomplished in Theorem \ref{MainTheorem}, which necessarily is the first of the main theorems to be proved. Establishing strong convergence amounts to verifying that the connectivity of the natural maps from X into each stage of the tower increase as you go up the tower, and verifying this essentially reduces to understanding the implications of the connectivity estimates in Propositions \ref{prop:connectivity_of_simplicial_maps_spectra} and \ref{prop:connectivity} when studied in the context of the calculations in Propositions \ref{prop:fattened_version_of_tower} and \ref{prop:pushout_diagram_for_bar_construction_tower_coaugmented} (see Proposition \ref{prop:connectivity_estimates}).

The upshot of strong convergence is that to calculate $\pi_i X$ for a fixed $i$, one only needs to calculate $\pi_i$ of a (sufficiently high but) finite stage of the tower. Having to only go ``finitely high up the tower'' to calculate $\pi_iX$, together with the explicit layer calculations in Theorem \ref{thm:calculating_fiber_of_induced_map} and Proposition \ref{prop:refined_bar_construction_calculation_for_homotopy_fiber}, are the key technical properties underlying our approach to the main theorems. For instance, our approach to the $\TQ$ finiteness theorem (Theorem \ref{thm:finiteness}) is to (i) start with an assumption about the finiteness properties of $\pi_i$ of $\TQ$-homology (which is the bottom stage of the tower), (ii) to use explicit calculations of the layers of the tower to prove that these same finiteness properties are inherited by $\pi_i$ of the layers, and (iii) to conclude that these finiteness properties are inherited by $\pi_i$ of each stage of the tower. Strong convergence of the homotopy completion tower then finishes the proof of the $\TQ$ finiteness theorem. It is essentially in this manner that we systematically exploit induction ``up the homotopy completion tower'' to prove each of the main theorems stated in the introduction.

\subsection{Simplicial bar constructions and the homotopy completion tower}
\label{sec:simplicial_bar_and_homotopy_completion}

Recall that $\capR$ is any commutative ring spectrum (Basic Assumption \ref{assumption:commutative_ring_spectrum}) and that $(\ModR,\Smash,\capR)$ denotes the closed symmetric monoidal category of $\capR$-modules (Definition \ref{defn:left_R_modules}). Denote by $\sSet$ (resp. $\sSet_*$) the category of simplicial sets (resp. pointed simplicial sets). There are adjunctions
$
\xymatrix@1{
  \sSet\ar@<0.5ex>[r]^-{(-)_+} & 
  \sSet_*\ar@<0.5ex>[l]^-{U}\ar@<0.5ex>[r]^-{\capR\tensor G_0} &
  \ModR,\ar@<0.5ex>[l]
}
$
with left adjoints on top and $U$ the forgetful functor (see Proposition \ref{prop:closed_symmetric_monoidal_structure_on_sym_sequences_pointed_ssets} for the tensor product $\tensor$ notation together with \eqref{eq:new_adjunctions_for_spectra_as_S_modules}). The functor $\capR\tensor G_0$ is left adjoint to ``evaluation at $0$''; the notation agrees with Subsection \ref{sec:model_structures_on_capR_modules} and \cite[after 2.2.5]{Hovey_Shipley_Smith}. Note that if $X\in\ModR$ and $K\in\sSet_*$, then there are natural isomorphisms $X\Smash K\Iso X\Smash (\capR\tensor G_0 K)$ in $\ModR$; in other words, taking the objectwise smash product of $X$ with $K$ (as pointed simplicial sets) is the same as taking the smash product of $X$ with $\capR\tensor G_0 K$ (as $\capR$-modules).

Recall the usual realization functor on simplicial $\capR$-modules and simplicial symmetric sequences; see also \cite[IV.1, VII.1]{Goerss_Jardine}.

\begin{defn} 
\label{defn:realization}
Consider symmetric sequences in $\ModR$. The \emph{realization} functors $|-|$ for simplicial $\capR$-modules and simplicial symmetric sequences are defined objectwise by the coends
\begin{align*}
  \functor{|-|}{\sModR}{\ModR},
  &\quad\quad
  X\longmapsto |X|:=X\Smash_{\Delta}\Delta[-]_+\ ,\\
  \functor{|-|}{\sSymSeq}{\SymSeq},
  &\quad\quad
  X\longmapsto |X|:=X\Smash_{\Delta}\Delta[-]_+\ .
\end{align*}
\end{defn}

\begin{prop}
\label{prop:realzns_fit_into_adjunctions}
The realization functors fit into adjunctions
\begin{align}
\label{eq:realization_mapping_object_adjunction_underlying}
\xymatrix{
  \sModR
  \ar@<0.5ex>[r]^-{|-|} & \ModR,\ar@<0.5ex>[l]
}\quad\quad
\xymatrix{
  \sSymSeq
  \ar@<0.5ex>[r]^-{|-|} & \SymSeq,\ar@<0.5ex>[l]
}
\end{align}
with left adjoints on top.
\end{prop}

\begin{proof}
Consider the case of $\capR$-modules (resp. symmetric sequences). Using the universal property of coends, it is easy to verify that the functor given objectwise by
$
  \Map(\capR\tensor G_0\Delta[-]_+,Y)
$
is a right adjoint of $|-|$.
\end{proof}

The following is closely related to \cite[IV.1.7]{Goerss_Jardine} and \cite[X.2.4]{EKMM}; see also \cite[A]{Dugger_Isaksen} and \cite[Chapter 18]{Hirschhorn}.

\begin{prop}
\label{prop:realization_monomorphisms_weak_equivalences}
Let $\function{f}{X}{Y}$ be a morphism of simplicial $\capR$-modules. 
If $f$ is a monomorphism (resp. objectwise weak equivalence), then $\function{|f|}{|X|}{|Y|}$ is a monomorphism (resp. weak equivalence).
\end{prop}

\begin{proof}
This is verified exactly as in \cite[proof of 4.8, 4.9]{Harper_Bar}, except using $(\ModR,\Smash,\capR)$ instead of $(\Spectra,\tensor_S,S)$.
\end{proof}

The following is closely related to \cite[X.1.3]{EKMM}.

\begin{prop}
\label{prop:realization_respects_smash_tensor_and_circle}
Consider symmetric sequences in $\capR$-modules. 
\begin{itemize}
\item[(a)] If $X,Y$ are simplicial $\capR$-modules, then there is a  natural isomorphism\\ $|X\Smash Y|\Iso|X|\Smash|Y|$.
\item[(b)] If $X,Y$ are simplicial symmetric sequences, then there are natural isomorphisms $|X\tensorcheck Y|\Iso|X|\tensorcheck|Y|$ and $|X\circ Y|\Iso|X|\circ|Y|$.
\item[(c)] If $\capO$ is a symmetric sequence, and $B$ is a simplicial symmetric sequence, then there is a natural isomorphism $|\capO[\mathbf{k}]\Smash_{\Sigma_k}B^{\tensorcheck k}|\Iso \capO[\mathbf{k}]\Smash_{\Sigma_k}|B|^{\tensorcheck k}$ for every $k\geq 2$.
\end{itemize}
Here, smash products, tensor products and circle products of simplicial objects are defined objectwise.
\end{prop}

\begin{rem}
\label{rem:realization_of_simplicial_pointed_spaces}
If $X\in\ssSet_*$, denote by $|X|:=X\Smash_\Delta\Delta[-]_+$ the realization of $X$. There is a natural isomorphism $X\times_\Delta\Delta[-]\Iso |X|$.
\end{rem}

\begin{proof}[Proof of Proposition \ref{prop:realization_respects_smash_tensor_and_circle}]
Consider part (a). Let $X,Y$ be simplicial objects in $\sSet_*$. By Remark \ref{rem:realization_of_simplicial_pointed_spaces}, together with \cite[IV.1.4]{Goerss_Jardine}, there is a natural isomorphism $|X\times Y|\Iso|X|\times|Y|$. Since realization $\functor{|-|}{\ssSet_*}{\sSet_*}$ is a left adjoint it commutes with colimits, and thus there is a natural isomorphism $|X\Smash Y|\Iso|X|\Smash|Y|$. Let $X,Y$ be simplicial $\capR$-modules and recall that $X\Smash Y\Iso X\tensor_\capR Y$. It follows that there are natural isomorphisms
$
  |X\Smash Y|
  \Iso\colim
  \Bigl(
  \xymatrix@1{
  |X|\tensor |Y| & |X|\tensor |\capR|\tensor |Y|
  \ar@<-0.5ex>[l]\ar@<0.5ex>[l]
  }
  \Bigr)
  \Iso|X|\Smash|Y|.
$
Parts (b) and (c) follow from part (a), together with the property that realization $|-|$ is a left adjoint and hence commutes with colimits.
\end{proof}

\begin{rem}
\label{rem:realization_induces_functor_on_algebras}
Let $\capO$ be an operad in $\capR$-modules. It follows easily from Proposition \ref{prop:realization_respects_smash_tensor_and_circle} that if $X$ is a simplicial $\capO$-algebra (resp. simplicial left $\capO$-module), then the realization of its underlying simplicial object $|X|$ has an induced $\capO$-algebra (resp. left $\capO$-module) structure; it follows that realization of the underlying simplicial objects induces functors
$
  \function{|-|}{\sAlgO}{\AlgO}
$ and 
$
  \function{|-|}{\sLtO}{\LtO}.
$
\end{rem}

\begin{rem}
In this paper we use the notation $\BAR$, as in Proposition \ref{prop:natural_map_is_weak_equivalence} below, to denote the simplicial bar construction (with respect to circle product) defined in \cite[5.30]{Harper_Bar}.
\end{rem}

\begin{prop}
\label{prop:natural_map_is_weak_equivalence}
Let $\capO\rarrow\capO'$ be a morphism of operads in $\capR$-modules. Let $X$ be a cofibrant $\capO$-algebra (resp. cofibrant left $\capO$-module). If the simplicial bar construction $\BAR(\capO,\capO,X)$ is objectwise cofibrant in $\AlgO$ (resp. $\LtO$), then the natural map
\begin{align*}
  &|\BAR(\capO',\capO,X)|\xrightarrow{\wequiv}\capO'\circ_\capO(X)
  \quad\quad
  \Bigl(\text{resp.}\quad
  &|\BAR(\capO',\capO,X)|\xrightarrow{\wequiv}\capO'\circ_\capO X
  \Bigr)
\end{align*}
is a weak equivalence.
\end{prop}

\begin{proof}
This follows easily from Theorem \ref{thm:bar_calculates_derived_circle} and its proof.
\end{proof}

The following theorem illustrates some of the good properties of the (positive) flat stable model structures (Section \ref{sec:model_structures}). We defer the proof to Section \ref{sec:homotopical_analysis_forgetful_functors}.

\begin{thm}
\label{thm:homotopical_analysis_of_forgetful_functors}
Let $\capO$ be an operad in $\capR$-modules such that $\capO[\mathbf{r}]$ is flat stable cofibrant in $\ModR$ for each $r\geq 0$.
\begin{itemize}
\item[(a)] If $\function{j}{A}{B}$ is a cofibration between cofibrant objects in $\AlgO$ (resp. $\LtO$), then $j$ is a positive flat stable cofibration in $\ModR$ (resp. $\SymSeq$).
\item[(b)] If $A$ is a cofibrant $\capO$-algebra (resp. cofibrant left $\capO$-module) and $\capO[\mathbf{0}]=*$, then $A$ is positive flat stable cofibrant in $\ModR$ (resp. $\SymSeq$).
\end{itemize}
\end{thm}

If $X$ is an $\capO$-algebra (resp. left $\capO$-module), then under appropriate cofibrancy conditions the coaugmented tower
$
  \{|\BAR(\capO,\capO,X)|\}\rarrow\{|\BAR(\tau_k\capO,\capO,X)|\}
$
obtained by applying $|\BAR(-,\capO, X)|$ to the coaugmented tower \eqref{eq:towers_of_operads}, provides a weakly equivalent ``fattened version'' of the completion tower of $X$.

\begin{cofibrancy}
\label{CofibrancyConditionWeaker}
If $\capO$ is an operad in $\capR$-modules, assume that $\capO[\mathbf{r}]$ is flat stable cofibrant in $\ModR$ for each $r\geq 0$.
\end{cofibrancy}

\begin{prop}
\label{prop:fattened_version_of_tower}
Let $\capO$ be an operad in $\capR$-modules such that $\capO[\mathbf{0}]=*$. Assume that $\capO$ satisfies Cofibrancy Condition \ref{CofibrancyConditionWeaker}. If $X$ is a cofibrant left $\capO$-module, then in the following commutative diagram of towers in $\LtO$
\begin{align*}
\xymatrix{
  \{|\BAR(\capO,\capO,X)|\}\ar[d]^{\wequiv}\ar[r] & 
  \{|\BAR(\tau_k\capO,\capO,X)|\}\ar[d]^{\wequiv}\\
  \{X\}\ar[r] & \{\tau_k\capO\circ_\capO X\},
}
\end{align*}
the vertical maps are levelwise weak equivalences.
\end{prop}

\begin{rem}
It follows from Remark \ref{rem:realization_induces_functor_on_algebras} that this diagram is a diagram of towers of left $\capO$-modules.
\end{rem}

\begin{proof}
Since $X$ is a cofibrant left $\capO$-module, by Theorem \ref{thm:homotopical_analysis_of_forgetful_functors} the simplicial bar construction $\BAR(\capO,\capO,X)$ is objectwise cofibrant in $\LtO$, and Proposition \ref{prop:natural_map_is_weak_equivalence} finishes the proof.
\end{proof}

\subsection{Homotopy fiber sequences and the homotopy completion tower}

The purpose of this subsection is to prove Theorem \ref{MainTheorem}(c). We begin by introducing the following useful notation. For each $k\geq 0$, the functor $\function{i_k}{\SymSeq}{\SymSeq}$ is defined objectwise by
\begin{align*}
  (i_k X)[\mathbf{r}]:=
  \left\{
  \begin{array}{rl}
    X[\mathbf{k}],&\text{for $r=k$,}\\
    *,&\text{otherwise}.
  \end{array}
  \right.
\end{align*}
In other words, $i_k X$ is the symmetric sequence concentrated at $k$ with value $X[\mathbf{k}]$.

\begin{prop}
\label{prop:pushout_diagram_for_bar_construction_tower}
Let $\capO$ be an operad in $\capR$-modules such that $\capO[\mathbf{0}]=*$. Let $X$ be an $\capO$-algebra (resp. left $\capO$-module) and $k\geq 2$. Then the left-hand pushout diagram
\begin{align}
\label{eq:pushout_diagram_bar_constructions}
\xymatrix{
  i_{k}\capO\ar[r]^-{\subsetof}\ar[d] & \tau_{k}\capO\ar[d] \\
  {*}\ar[r] & \tau_{k-1}\capO
}
\quad\quad
\xymatrix{
|\BAR(i_{k}\capO,\capO,X)|\ar[d]\ar[r]^-{(*)} & 
|\BAR(\tau_{k}\capO,\capO,X)|\ar[d]\\
{*}\ar[r] & 
|\BAR(\tau_{k-1}\capO,\capO,X)|
}
\end{align}
in $\RtO$ induces the right-hand pushout diagram in $\Alg_I$ (resp. $\SymSeq$).  The map $(*)$ is a monomorphism, the left-hand diagram is a pullback diagram in $\Bi_{(\capO,\capO)}$, and the right-hand diagram is a pullback diagram in $\AlgO$ (resp. $\LtO$).
\end{prop}

\begin{proof}
It suffices to consider the case of left $\capO$-modules. The right-hand diagram is obtained by applying $|\BAR(-,\capO, X)|$ to the left-hand diagram. Since the forgetful functor $\RtO\rarrow\SymSeq$ preserves colimits, the left-hand diagram is also a pushout diagram in $\SymSeq$. It follows from the adjunction \eqref{eq:circle_mapping_sequence_adjunction} that applying $\BAR(-,\capO, X)$ to the left-hand diagram gives a pushout diagram of simplicial symmetric sequences. Noting that the realization functor $|-|$ is a left adjoint and preserves monomorphisms (\ref{prop:realzns_fit_into_adjunctions}, \ref{prop:realization_monomorphisms_weak_equivalences}), together with the fact that pullbacks in $\Bi_{(\capO,\capO)}$ and $\LtO$ are calculated in the underlying category, finishes the proof.
\end{proof}

\begin{prop}
\label{prop:retract_property_and_derived_circle}
Let $\capO$ be an operad in $\capR$-modules such that $\capO[\mathbf{0}]=*$, and let $k\geq 2$.
\begin{itemize}
\item[(a)] The canonical maps $i_k\capO\rarrow\capO\rarrow i_k\capO$ in $\Rt_{\tau_1\capO}$ factor the identity map.
\item[(b)] The functors
$
  \function{i_k\capO\circ_{\tau_1\capO}(-)}{\Alg_{\tau_1\capO}}{\AlgO}
$ and
$
  \function{i_k\capO\circ_{\tau_1\capO}-}{\Lt_{\tau_1\capO}}{\LtO}
$
preserve weak equivalences between cofibrant objects, and hence the total left derived functors $i_k\capO\circ^\HH_{\tau_1\capO}(-)$ and $i_k\capO\circ^\HH_{\tau_1\capO}-$ exist \cite[9.3, 9.5]{Dwyer_Spalinski}.
\end{itemize}
\end{prop}

\begin{proof}
Part (a) is clear. To prove part (b), it suffices to consider the case of left $\tau_1\capO$-modules. Let $B\rarrow B'$ be a weak equivalence between cofibrant objects in $\Lt_{\tau_1\capO}$. By part (a) there is a retract of maps of the form
\begin{align*}
\xymatrix{
  i_k\capO\circ_{\tau_1\capO}B\ar[d]^{(*)}\ar[r] &
  \capO\circ_{\tau_1\capO}B\ar[d]^{(**)}\ar[r] &
  i_k\capO\circ_{\tau_1\capO}B\ar[d]^{(*)}\\
  i_k\capO\circ_{\tau_1\capO}B'\ar[r] &
  \capO\circ_{\tau_1\capO}B'\ar[r] &
  i_k\capO\circ_{\tau_1\capO}B'
}
\end{align*}
in $\SymSeq$. Since $\function{\capO\circ_{\tau_1\capO}-}{\Lt_{\tau_1\capO}}{\LtO}$ is a left Quillen functor (induced by the canonical map $\tau_1\capO\rarrow\capO$ of operads), we know that $(**)$ is a weak equivalence and hence $(*)$ is a weak equivalence.
\end{proof}

The following theorem illustrates a few more of the good properties of the (positive) flat stable model structures (Section \ref{sec:model_structures}). We defer the proof to Section \ref{sec:homotopical_analysis_bar_constructions}.

\begin{thm}
\label{thm:reedy_cofibrant_for_bar_constructions}
Let $\function{f}{\capO}{\capO'}$ be a morphism of operads in $\capR$-modules such that $\capO[\mathbf{0}]=*$. Assume that $\capO$ satisfies Cofibrancy Condition \ref{CofibrancyCondition}. Let $Y$ be an $\capO$-algebra (resp. left $\capO$-module) and consider the simplicial bar construction $\BAR(\capO',\capO,Y)$.
\begin{itemize}
\item[(a)] If $Y$ is positive flat stable cofibrant in $\ModR$ (resp. $\SymSeq$), then $\BAR(\capO',\capO,Y)$ is Reedy cofibrant in $\sAlg_{\capO'}$ (resp. $\sLt_{\capO'}$).
\item[(b)] If $Y$ is positive flat stable cofibrant in $\ModR$ (resp. $\SymSeq$), then $|\BAR(\capO',\capO,Y)|$ is cofibrant in $\Alg_{\capO'}$ (resp. $\Lt_{\capO'}$).
\end{itemize}
\end{thm}

\begin{prop}
\label{prop:cofibrant_tau_1_O_algebras_and_bar_constructions}
Let $\capO$ be an operad in $\capR$-modules such that $\capO[\mathbf{0}]=*$. Assume that $\capO$ satisfies Cofibrancy Condition \ref{CofibrancyCondition}. If $X$ is a cofibrant $\capO$-algebra (resp. cofibrant left $\capO$-module), then $|\BAR(\tau_1\capO,\capO,X)|$ is cofibrant in $\Alg_{\tau_1\capO}$ (resp. $\Lt_{\tau_1\capO}$).
\end{prop}

\begin{proof}
This follows from Theorems \ref{thm:reedy_cofibrant_for_bar_constructions} and \ref{thm:homotopical_analysis_of_forgetful_functors}.
\end{proof}

Next we explicitly calculate the $k$-th layer of the homotopy completion tower.

\begin{thm}
\label{thm:calculating_fiber_of_induced_map}
Let $\capO$ be an operad in $\capR$-modules such that $\capO[\mathbf{0}]=*$. Assume that $\capO$ satisfies Cofibrancy Condition \ref{CofibrancyCondition}. Let $X$ be an $\capO$-algebra (resp. left $\capO$-module), and let $k\geq 2$. 
\begin{itemize}
\item[(a)] There is a homotopy fiber sequence of the form
\begin{align*}
  &i_k\capO\circ^\HH_{\tau_1\capO}\bigl(\TQ(X)\bigr)\rarrow
  \tau_k\capO\circ^\HH_\capO (X)\rarrow
  \tau_{k-1}\capO\circ^\HH_\capO (X)\\
  \Bigl(\text{resp.}\quad
  &i_k\capO\circ^\HH_{\tau_1\capO}\TQ(X)\rarrow
  \tau_k\capO\circ^\HH_\capO X\rarrow
  \tau_{k-1}\capO\circ^\HH_\capO X
  \Bigr)
\end{align*}
in $\AlgO$ (resp. $\LtO$), natural in $X$.
\item[(b)] If $X$ is cofibrant in $\AlgO$ (resp. $\LtO$), then there are natural weak equivalences
\begin{align*}
  &|\BAR(i_k\capO,\capO,X)|\wequiv
  i_k\capO\circ_{\tau_1\capO}(|\BAR(\tau_1\capO,\capO,X)|)\wequiv
  i_k\capO\circ^\HH_{\tau_1\capO}\bigl(\TQ(X)\bigr)
  \\
  \Bigl(\text{resp.}\quad
  &|\BAR(i_k\capO,\capO,X)|\wequiv
  i_k\capO\circ_{\tau_1\capO}|\BAR(\tau_1\capO,\capO,X)|\wequiv
  i_k\capO\circ^\HH_{\tau_1\capO}\TQ(X)
  \Bigr).
\end{align*}
\item[(c)] If $X$ is cofibrant in $\AlgO$ (resp. $\LtO$) and $\capO[\mathbf{1}]=I[\mathbf{1}]$, then there are natural weak equivalences
\begin{align*}
  &\capO[\mathbf{k}]\Smash_{\Sigma_{k}}
  |\BAR(I,\capO,X)|^{\wedge k}\wequiv
  i_k\capO\circ^\HH_{\tau_1\capO}\bigl(\TQ(X)\bigr)
  \\
  \Bigl(\text{resp.}\quad
  &\capO[\mathbf{k}]\Smash_{\Sigma_{k}}
  |\BAR(I,\capO,X)|^{\tensorcheck k}\wequiv
  i_k\capO\circ^\HH_{\tau_1\capO}\TQ(X)
  \Bigr).
\end{align*}
\end{itemize}
\end{thm}

For useful material related to homotopy fiber sequences, see \cite[II.8, II.8.20]{Goerss_Jardine}.

\begin{proof}
It suffices to consider the case of left $\capO$-modules. Consider part (a). It is enough to treat the special case where $X$ is a cofibrant left $\capO$-module. By Proposition \ref{prop:pushout_diagram_for_bar_construction_tower} there is a homotopy fiber sequence of the form
\begin{align}
\label{eq:homotopy_fiber_sequence_bar_constructions}
  |\BAR(i_{k}\capO,\capO,X)| \rarrow
  |\BAR(\tau_{k}\capO,\capO,X)| \rarrow
  |\BAR(\tau_{k-1}\capO,\capO,X)|
\end{align}
in $\LtO$, natural in $X$. By Proposition \ref{prop:fattened_version_of_tower} we know that \eqref{eq:homotopy_fiber_sequence_bar_constructions} has the form
\begin{align*}
  |\BAR(i_{k}\capO,\capO,X)|\rarrow
  \tau_k\capO\circ^\HH_\capO X\rarrow
  \tau_{k-1}\capO\circ^\HH_\capO X.
\end{align*}
Since the right $\capO$-action map $i_k\capO\circ\capO\rarrow i_k\capO$ factors as $i_k\capO\circ\capO\rarrow i_k\capO\circ\tau_1\capO\rarrow i_k\capO$, there are natural isomorphisms
\begin{align}
\label{eq:factoring_the_trivial_module_in_bar_construction}
  \BAR(i_{k}\capO,\capO,X)\Iso
  i_k\capO\circ_{\tau_1\capO}\BAR(\tau_1\capO,\capO,X)
\end{align}
of simplicial left $\capO$-modules. Applying the realization functor to \eqref{eq:factoring_the_trivial_module_in_bar_construction}, it follows from Proposition \ref{prop:realization_respects_smash_tensor_and_circle}, Proposition \ref{prop:cofibrant_tau_1_O_algebras_and_bar_constructions},  Theorem \ref{thm:homotopical_analysis_of_forgetful_functors}, and Proposition \ref{prop:fattened_version_of_tower}, that there are natural weak equivalences
\begin{align}
\label{eq:derived_functor_interpretation_of_above}
|\BAR(i_k\capO,\capO,X)|\wequiv
  i_k\capO\circ_{\tau_1\capO}|\BAR(\tau_1\capO,\capO,X)|\wequiv
  i_k\capO\circ^\HH_{\tau_1\capO}\TQ(X)
\end{align}
which finishes the proof of part (a). Part (b) follows from the proof of part (a) above. 

Consider part (c). Proceed as in the proof of part (a) above, and assume furthermore that $\capO[\mathbf{1}]=I[\mathbf{1}]$. It follows from \eqref{eq:circle_product_calc} that
\begin{align*}
  i_k\capO\circ |\BAR(I,\capO,X)|\wequiv
  \capO[\mathbf{k}]\Smash_{\Sigma_{k}}
  |\BAR(I,\capO,X)|^{\tensorcheck k}
\end{align*}
from which we can conclude, by applying the second equivalence in \eqref{eq:derived_functor_interpretation_of_above}, since $\tau_1\capO=I$ (Definition \ref{defn:operad}).
\end{proof}

\begin{prop}
\label{prop:induction_argument_cofiber_sequence}
Let $\capO$ be an operad in $\capR$-modules such that $\capO[\mathbf{0}]=*$. Assume that $\capO$ satisfies Cofibrancy Condition \ref{CofibrancyCondition}. Let $\function{f}{X}{Y}$ be a map between cofibrant objects in $\AlgO$ (resp. $\LtO$). If the induced map $|\BAR(\tau_1\capO,\capO,X)|\xrightarrow{\wequiv}|\BAR(\tau_1\capO,\capO,Y)|$ is a weak equivalence, then the induced map
$|\BAR(\tau_k\capO,\capO,X)|\xrightarrow{\wequiv}|\BAR(\tau_k\capO,\capO,Y)|$ is a weak equivalence for each $k\geq 2$.
\end{prop}

\begin{proof}
It suffices to consider the case of left $\capO$-modules. Consider the 
\begin{align}
\label{eq:diagram_of_fibration_sequences_bar}
\xymatrix{
  |\BAR(i_k\capO,\capO,X)|\ar[d]\ar[r] & 
  |\BAR(\tau_k\capO,\capO,X)|\ar[d]\ar[r] & 
  |\BAR(\tau_{k-1}\capO,\capO,X)|\ar[d] \\
  |\BAR(i_k\capO,\capO,Y)|\ar[r] & 
  |\BAR(\tau_k\capO,\capO,Y)|\ar[r] & 
  |\BAR(\tau_{k-1}\capO,\capO,Y)|
}
\end{align}
commutative diagram in $\SymSeq$. It follows from Theorem \ref{thm:calculating_fiber_of_induced_map} that the left-hand vertical map is a weak equivalence for each $k\geq 2$. If $k=2$, then the right-hand vertical map is a weak equivalence by assumption, hence by Proposition \ref{prop:pushout_diagram_for_bar_construction_tower} and induction on $k$, the middle vertical map is a weak equivalence for each $k\geq 2$.
\end{proof}

\begin{proof}[Proof of Theorem \ref{MainTheorem}(c)]
It suffices to consider the case of left $\capO$-modules. By Theorem \ref{thm:comparing_homotopy_completion_towers} and Propositions \ref{prop:replacement_of_operads} and \ref{prop:forgetful_functor_commutes_with_holim}, we can suppose that $\capO$ satisfies Cofibrancy Condition \ref{CofibrancyCondition}. We can restrict to the following special case. Let $\function{f}{X}{Y}$ be a map of left $\capO$-modules between cofibrant objects in $\LtO$ such that the induced map $\tau_1\capO\circ_\capO X\rarrow \tau_1\capO\circ_\capO Y$ is a weak equivalence. We need to verify that the induced map 
$
  \function{f_*}{\tau_k\capO\circ_\capO X}{\tau_k\capO\circ_\capO Y}
$
is a weak equivalence for each $k\geq 2$. We know by Theorem \ref{thm:homotopical_analysis_of_forgetful_functors} that $X,Y$ are positive flat stable cofibrant in $\SymSeq$. If $k=1$, the map $f_*$ is a weak equivalence by assumption, and hence the induced map
$
  |\BAR(\tau_1\capO,\capO,X)|\rarrow|\BAR(\tau_1\capO,\capO,Y)|
$
is a weak equivalence by Proposition \ref{prop:fattened_version_of_tower}. It follows from Propositions \ref{prop:induction_argument_cofiber_sequence} and \ref{prop:fattened_version_of_tower} that $f_*$ is a weak equivalence for each $k\geq 2$, which finishes the proof.
\end{proof}

\subsection{Strong convergence of the homotopy completion tower}

The purpose of this subsection is to prove Theorem \ref{MainTheorem}(a). For each $k\geq 0$, the functor $\function{(-)^{>k}}{\SymSeq}{\SymSeq}$ is defined objectwise by
\begin{align*}
  (X^{>k})[\mathbf{r}]:=
  \left\{
  \begin{array}{rl}
    X[\mathbf{r}],&\text{for $r>k$,}\\
    *,&\text{otherwise}.
  \end{array}
  \right.
\end{align*}

\begin{prop}
\label{prop:pushout_diagram_for_bar_construction_tower_coaugmented}
Let $\capO$ be an operad in $\capR$-modules such that $\capO[\mathbf{0}]=*$. Let $X$ be an $\capO$-algebra (resp. left $\capO$-module) and $k\geq 1$. Then the left-hand pushout diagram
\begin{align}
\label{eq:pushout_diagram_bar_constructions_tower_coaugmented}
\xymatrix{
  \capO^{>k}\ar[d]\ar[r]^-{\subsetof} & \capO\ar[d]\\
  {*}\ar[r] & \tau_k\capO
}\quad\quad
\xymatrix{
  |\BAR(\capO^{>k},\capO,X)|\ar[d]\ar[r]^-{(*)} &
  |\BAR(\capO,\capO,X)|\ar[d]\\
  {*}\ar[r] & |\BAR(\tau_k\capO,\capO,X)|
}
\end{align}
in $\RtO$ induces the right-hand pushout diagram in $\Alg_I$ (resp. $\SymSeq$).  The map $(*)$ is a monomorphism, the left-hand diagram is a pullback diagram in $\Bi_{(\capO,\capO)}$, and the right-hand diagram is a pullback diagram in $\AlgO$ (resp. $\LtO$).
\end{prop}

\begin{proof}
It suffices to consider the case of left $\capO$-modules. The right-hand diagram is obtained by applying $|\BAR(-,\capO, X)|$ to the left-hand diagram, and exactly the same argument used in the proof of Proposition \ref{prop:pushout_diagram_for_bar_construction_tower} allows to conclude.
\end{proof}

The following two propositions are well known in stable homotopy theory. For the convenience of the reader, we have included short homotopical proofs in the context of symmetric spectra; see also \cite[4.3]{Jardine_generalized_etale}. We defer the proof of the second proposition to Section \ref{sec:homotopical_analysis_forgetful_functors}.

\begin{prop}
\label{prop:connectivity_of_simplicial_maps_spectra}
Let $\function{f}{X}{Y}$ be a morphism of simplicial symmetric spectra (resp. simplicial $\capR$-modules). Let $k\in\ZZ$.
\begin{itemize}
\item[(a)] If $Y$ is objectwise $k$-connected, then $|Y|$ is $k$-connected.
\item[(b)] If $f$ is objectwise $k$-connected, then $\function{|f|}{|X|}{|Y|}$ is $k$-connected.
\end{itemize}
\end{prop}

\begin{proof}
Consider  part (b) for the case of symmetric spectra. We need to verify that the realization $\function{|f|}{|X|}{|Y|}$ is $k$-connected. By exactly the same argument as in the proof of \cite[9.21]{Harper_Bar}, it follows from a filtration of degenerate subobjects (see also \cite[4.3]{Jardine_generalized_etale}) that the induced map $\function{Df_n}{DX_n}{DY_n}$ on degenerate subobjects is $k$-connected for each $n\geq 1$. Using exactly the same argument as in the proof of \cite[4.8]{Harper_Bar}, it then follows from the skeletal filtration of realization that $|f|$ is $k$-connected. Part (a) follows from part (b) by considering the map $*\rarrow Y$. The case of $\capR$-modules reduces to the case of symmetric spectra by applying the forgetful functor.
\end{proof}

\begin{rem}
\label{rem:smash_product_and_tensor_product}
It is important to note (Basic Assumption \ref{assumption:commutative_ring_spectrum}), particularly in Proposition \ref{prop:connectivity} below, that the tensor product $\tensor_S$ denotes the usual smash product of symmetric spectra \cite[2.2.3]{Hovey_Shipley_Smith}. For notational convenience, in this paper we use the smash product notation $\wedge$ to denote the smash product of $\capR$-modules (Definition \ref{defn:left_R_modules}), since the entire paper is written in this context. In particular, in the special case when $\capR=S$, the two agree $\Smash=\tensor_S$.
\end{rem}

\begin{prop}
\label{prop:connectivity}
Consider symmetric sequences in $\capR$-modules. Let $m,n\in\ZZ$ and $t\geq 1$. Assume that $\capR$ is $(-1)$-connected.
\begin{itemize}
\item[(a)] If $X,Y$ are symmetric spectra such that $X$ is $m$-connected and $Y$ is $n$-connected, then $X\tensor_S^\LL Y$ is $(m+n+1)$-connected.
\item[(b)] If $X,Y$ are $\capR$-modules such that $X$ is $m$-connected and $Y$ is $n$-connected, then $X\wedge^\LL Y$ is $(m+n+1)$-connected.
\item[(c)] If $X,Y$ are $\capR$-modules with a right (resp. left) $\Sigma_t$-action such that $X$ is $m$-connected and $Y$ is $n$-connected, then $X\wedge^\LL_{\Sigma_t} Y$ is $(m+n+1)$-connected.
\item[(d)] If $X,Y$ are symmetric sequences such that $X$ is $m$-connected and $Y$ is $n$-connected, then $X\tensorcheck^\LL Y$ is $(m+n+1)$-connected.
\item[(e)] If $X,Y$ are symmetric sequences with a right (resp. left) $\Sigma_t$-action such that $X$ is $m$-connected and $Y$ is $n$-connected, then $X\tensorcheck^\LL_{\Sigma_t} Y$ is $(m+n+1)$-connected.
\end{itemize}
Here, $\tensor_S^\LL$, $\wedge^\LL$, $\wedge^\LL_{\Sigma_t}$, $\tensorcheck^\LL$, and $\tensorcheck^\LL_{\Sigma_t}$ are the total left derived functors of $\tensor_S$, $\wedge$, $\wedge_{\Sigma_t}$, $\tensorcheck$, and $\tensorcheck_{\Sigma_t}$ respectively.
\end{prop}

\begin{prop}
\label{prop:connectivity_estimates}
Let $\capO$ be an operad in $\capR$-modules such that $\capO[\mathbf{0}]=*$. Assume that $\capO$ satisfies Cofibrancy Condition \ref{CofibrancyConditionWeaker}. Let $X$ be a cofibrant $\capO$-algebra (resp. cofibrant left $\capO$-module) and $k\geq 1$. If $\capO,\capR$ are $(-1)$-connected and $X$ is $0$-connected, then $|\BAR(\tau_{k}\capO,\capO,X)|$ is $0$-connected and both $|\BAR(\capO^{>k},\capO,X)|$ and $|\BAR(i_{k+1}\capO,\capO,X)|$ are $k$-connected.
\end{prop}

\begin{proof}
This follows from Theorem \ref{thm:homotopical_analysis_of_forgetful_functors} and Propositions \ref{prop:connectivity_of_simplicial_maps_spectra} and \ref{prop:connectivity}.
\end{proof}

The following Milnor type short exact sequences are well known in stable homotopy theory (for a recent reference, see \cite{Dwyer_Greenlees_Iyengar}); they can be established as a consequence of \cite[IX]{Bousfield_Kan}.

\begin{prop}
\label{prop:short_exact_sequence}
Consider any tower $B_0\leftarrow B_1\leftarrow B_2\leftarrow\cdots$ of symmetric spectra (resp. $\capR$-modules). There are natural short exact sequences
\begin{align*}
  &0\rightarrow\lim\nolimits^1_k\pi_{i+1}B_k\rightarrow
  \pi_i\holim_k B_k\rightarrow
  \lim\nolimits_k\pi_i B_k\rightarrow 0.
\end{align*}
\end{prop}

\begin{proof}[Proof of Theorem \ref{MainTheorem}(a)]
It suffices to consider the case of left $\capO$-modules. By Theorem \ref{thm:comparing_homotopy_completion_towers} and Propositions \ref{prop:replacement_of_operads} and \ref{prop:forgetful_functor_commutes_with_holim}, we can restrict to operads $\capO$ satisfying Cofibrancy Condition \ref{CofibrancyCondition}. It is enough to treat the following special case. Let $X$ be a $0$-connected, cofibrant left $\capO$-module. We need to verify that the natural coaugmentation
$
  X\wequiv\holim_k X \rarrow
  \holim_k(\tau_k\capO\circ_\capO X)
$
is a weak equivalence. By Proposition \ref{prop:fattened_version_of_tower} it suffices to verify that 
$
  \holim_k|\BAR(\capO,\capO,X)|\rarrow
  \holim_k|\BAR(\tau_k\capO,\capO,X)|
$
is a weak equivalence. Consider the commutative diagram
\begin{align*}
\xymatrix{
  \pi_i\holim_k|\BAR(\capO,\capO,X)|\ar[d]_{\Iso}\ar[r]^-{(*)} &
  \pi_i\holim_k|\BAR(\tau_k\capO,\capO,X)|\ar[d]^{(*'')}\\
  \lim_k\pi_i|\BAR(\capO,\capO,X)|\ar[r]^-{(*')} &
  \lim_k\pi_i|\BAR(\tau_k\capO,\capO,X)|
}
\end{align*}
for each $i$. Since $\lim^1_k\pi_{i+1}|\BAR(\capO,\capO,X)|=0$, the left-hand vertical map is an isomorphism by Proposition \ref{prop:short_exact_sequence}. We need to show that the map $(*)$ is an isomorphism, hence it suffices to verify that $(*')$ and $(*'')$ are isomorphisms. First note that Propositions \ref{prop:pushout_diagram_for_bar_construction_tower_coaugmented} and \ref{prop:connectivity_estimates} imply that $(*')$ is an isomorphism. Similarly, by Propositions \ref{prop:pushout_diagram_for_bar_construction_tower} and \ref{prop:connectivity_estimates}, it follows that for each $k\geq 1$ the induced map
$
  \pi_i|\BAR(\tau_{k+1}\capO,\capO,X)|\rarrow
  \pi_i|\BAR(\tau_{k}\capO,\capO,X)|
$
is an isomorphism for $i\leq k$ and a surjection for $i={k+1}$; in particular, for each fixed $i$ the tower of abelian groups $\{\pi_i|\BAR(\tau_{k}\capO,\capO,X)|\}$ is eventually constant. Hence  
$
  \lim\nolimits^1_k\pi_{i+1}|\BAR(\tau_k\capO,\capO,X)|=0
$
and by Proposition \ref{prop:short_exact_sequence} the map $(*'')$ is an isomorphism which finishes the proof. By the argument above, note that for each $k\geq 1$ the natural maps
$
  \pi_i X\rarrow
  \pi_i(\tau_{k}\capO\circ_\capO X)
$ and 
$
  \pi_i(\tau_{k+1}\capO\circ_\capO X)\rarrow
  \pi_i(\tau_{k}\capO\circ_\capO X)
$
are isomorphisms for $i\leq k$ and surjections for $i={k+1}$; we sometimes refer to this as the \emph{strong convergence} of the homotopy completion tower.
\end{proof}

\subsection{On $n$-connected maps and the homotopy completion tower}

The purpose of this subsection is to prove Theorems \ref{thm:hurewicz}, \ref{thm:relative_hurewicz}, and \ref{MainTheorem}(b).

\begin{prop}
\label{prop:refined_bar_construction_calculation_for_homotopy_fiber}
Let $\capO$ be an operad in $\capR$-modules such that $\capO[\mathbf{0}]=*$. Assume that $\capO$ satisfies Cofibrancy Condition \ref{CofibrancyCondition}. Let $X$ be a cofibrant $\capO$-algebra (resp. cofibrant left $\capO$-module) and $k\geq 2$. There are natural weak equivalences
\begin{align}
\label{eq:refined_bar_calculation_hello}
  |\BAR(i_k\capO,\capO,X)|\wequiv
  |\BAR(i_k\capO,\tau_1\capO,|\BAR(\tau_1\capO,\capO,X)|)|.
\end{align}
\end{prop}

Below we give a simple conceptual proof of this proposition using derived functors. An anonymous referee has suggested an alternate proof working directly with (bi)simplicial bar constructions, for which the interested reader may jump directly to Remark \ref{rem:suggested_proof_multisimplicial}. The following proposition is an easy exercise in commuting certain left derived functors and homotopy colimits; we defer the proof to Section \ref{sec:homotopical_analysis_forgetful_functors}.

\begin{prop}
\label{prop:commuting_hocolim_of_simplicial_objects}
Let $\capO$ be an operad in $\capR$-modules such that $\capO[\mathbf{0}]=*$. Let $k\geq 2$. If $B$ is a simplicial $\tau_1\capO$-algebra (resp. simplicial left $\tau_1\capO$-module), then there is a zigzag of weak equivalences
\begin{align*}
  i_k\capO\circ^\HH_{\tau_1\capO}\bigl(\hocolim\limits^{\Alg_{\tau_1\capO}}_{\Delta^\op}B\bigr)
  &\wequiv
  \hocolim\limits^{\Alg_{\capO}}_{\Delta^\op}
  i_k\capO\circ^\HH_{\tau_1\capO}(B) \\
  \Bigl(
  \text{resp.}\quad
  i_k\capO\circ^\HH_{\tau_1\capO}\hocolim\limits^{\Lt_{\tau_1\capO}}_{\Delta^\op}B
  &\wequiv
  \hocolim\limits^{\Lt_{\capO}}_{\Delta^\op}
  i_k\capO\circ^\HH_{\tau_1\capO}B 
    \Bigr)
\end{align*}
natural in $B$.
\end{prop}

\begin{proof}[Proof of Proposition \ref{prop:refined_bar_construction_calculation_for_homotopy_fiber}]
It suffices to consider the case of left $\capO$-modules. For notational ease, define $B:=|\BAR(\tau_1\capO,\capO,X)|$. By Theorems \ref{thm:calculating_fiber_of_induced_map} and \ref{thm:fattened_replacement}, Proposition \ref{prop:commuting_hocolim_of_simplicial_objects}, Proposition \ref{prop:cofibrant_tau_1_O_algebras_and_bar_constructions} and Theorem \ref{main_hocolim_theorem}, there are natural weak equivalences
\begin{align*}
  |\BAR(i_k\capO,\capO,X)|\wequiv
  i_k\capO\circ^\HH_{\tau_1\capO}B
  \wequiv
  i_k\capO\circ^\HH_{\tau_1\capO}
  \hocolim\limits^{\Lt_{\tau_1\capO}}_{\Delta^\op}
  \BAR(\tau_1\capO,\tau_1\capO,B)\\
  \wequiv
  \hocolim\limits^{\Lt_{\capO}}_{\Delta^\op}
  i_k\capO\circ^\HH_{\tau_1\capO}
  \BAR(\tau_1\capO,\tau_1\capO,B)
  \wequiv
  \hocolim\limits^{\Lt_{\capO}}_{\Delta^\op}
  i_k\capO\circ_{\tau_1\capO}
  \BAR(\tau_1\capO,\tau_1\capO,B)\\
  \wequiv
  \hocolim\limits^{\Lt_{\capO}}_{\Delta^\op}
  \BAR(i_k\capO,\tau_1\capO,B)
  \wequiv
  |\BAR(i_k\capO,\tau_1\capO,B)|.
\end{align*}
\end{proof}

\begin{rem}
\label{rem:suggested_proof_multisimplicial}
Here is an alternate proof of Proposition \ref{prop:refined_bar_construction_calculation_for_homotopy_fiber} that was suggested by an anonymous referee. It suffices to consider the case of left $\capO$-modules. For notational ease, define $B:=|\BAR(i_k\capO,\tau_1\capO,\tau_1\capO)|$. The right-hand side of \eqref{eq:refined_bar_calculation_hello} is isomorphic to $|\BAR(B,\capO,X)|$ (they are both realizations of a bisimplicial symmetric sequence). Noting that the natural map $B\rarrow i_k\capO$ of right $\tau_1\capO$-modules (and hence of right $\capO$-modules) is a weak equivalence (\cite[8.4, 8.3]{Harper_Bar}), together with Theorem \ref{thm:homotopical_analysis_of_forgetful_functors} and Proposition \ref{prop:realization_monomorphisms_weak_equivalences}, it follows that $|\BAR(B,\capO,X)|\rarrow|\BAR(i_k\capO,\capO,X)|$ is a weak equivalence, which finishes the proof.
\end{rem}

\begin{proof}[Proof of Theorem \ref{thm:hurewicz}]
It suffices to consider the case of left $\capO$-modules. By Theorem \ref{thm:comparing_homotopy_completion_towers} and Propositions \ref{prop:replacement_of_operads} and \ref{prop:forgetful_functor_commutes_with_holim}, we can restrict to operads $\capO$ satisfying Cofibrancy Condition \ref{CofibrancyCondition}. It is enough to treat the special case where $X$ is a cofibrant left $\capO$-module. 

Consider part (a). Assume that $\tau_{1}\capO\circ_\capO X$ is $n$-connected. Then $|\BAR(\tau_{1}\capO,\capO,X)|$ is $n$-connected by \ref{prop:fattened_version_of_tower}, hence by Proposition \ref{prop:refined_bar_construction_calculation_for_homotopy_fiber}, together with Theorem \ref{thm:homotopical_analysis_of_forgetful_functors} and Propositions \ref{prop:connectivity_of_simplicial_maps_spectra} and \ref{prop:connectivity}, it follows that $|\BAR(i_{k+1}\capO,\capO,X)|$ is $((k+1)n+k)$-connected for each $k\geq 1$. Hence it follows from \ref{prop:pushout_diagram_for_bar_construction_tower} and \ref{prop:fattened_version_of_tower} that for each $k\geq 1$ the natural maps $
  \pi_i(\tau_{k+1}\capO\circ_\capO X)\rarrow
  \pi_i(\tau_{k}\capO\circ_\capO X)
$
are isomorphisms for $i\leq (k+1)n+k$ and surjections for $i=(k+1)(n+1)$. In particular, for each $i\leq 2n+1$ the tower $\{\pi_i(\tau_{k}\capO\circ_\capO X)\}$ is a tower of isomorphisms, and since $\tau_{1}\capO\circ_\capO X$ is $n$-connected, it follows that each stage in the tower $\{\tau_{k}\capO\circ_\capO X\}$ is $n$-connected. Since $X$ is $0$-connected by assumption, it follows from strong convergence of the homotopy completion tower (proof of Theorem \ref{MainTheorem}(a)) that the map $\pi_i X\rarrow\pi_i(\tau_{k}\capO\circ_\capO X)$ is an isomorphism for every $i\leq k$. Hence taking $k$ sufficiently large ($k\geq n$) verifies that $X$ is $n$-connected.

Conversely, assume that $X$ is $n$-connected. Then by Theorem \ref{thm:homotopical_analysis_of_forgetful_functors} and Propositions \ref{prop:connectivity_of_simplicial_maps_spectra} and \ref{prop:connectivity}, it follows that $|\BAR(\tau_{k}\capO,\capO,X)|$ is $n$-connected and both $|\BAR(\capO^{>k},\capO,X)|$ and $|\BAR(i_{k+1}\capO,\capO,X)|$ are $((k+1)n+k)$-connected for each $k\geq 1$. It follows from \ref{prop:pushout_diagram_for_bar_construction_tower}, \ref{prop:pushout_diagram_for_bar_construction_tower_coaugmented}, and \ref{prop:fattened_version_of_tower} that for each $k\geq 1$ the natural maps
$
  \pi_i X\rarrow
  \pi_i(\tau_{k}\capO\circ_\capO X)
$ and
$
  \pi_i(\tau_{k+1}\capO\circ_\capO X)\rarrow
  \pi_i(\tau_{k}\capO\circ_\capO X)
$
are isomorphisms for $i\leq (k+1)n+k$ and surjections for $i=(k+1)(n+1)$. Consequently, $\pi_i X\rarrow\pi_i(\tau_1\capO\circ_\capO X)$
is an isomorphism for $i\leq 2n+1$ and a surjection for $i=2n+2$. Since $X$ is $n$-connected, it follows that $\tau_{1}\capO\circ_\capO X$ is $n$-connected.

Consider part (b). Assume that $\tau_{1}\capO\circ_\capO X$ is $n$-connected. Then it follows from the proof of part (a) above that $\pi_i X\rarrow\pi_i(\tau_1\capO\circ_\capO X)$ is an isomorphism for $i\leq 2n+1$ and a surjection for $i=2n+2$.
\end{proof}

\begin{proof}[Proof of Theorem \ref{MainTheorem}(b)]
The homotopy completion spectral sequence is the homotopy spectral sequence \cite{Bousfield_Kan} associated to the tower of fibrations (of fibrant objects) of a fibrant replacement (Definition \ref{defn:model_structure_on_towers}) of the homotopy completion tower, reindexed as a (second quadrant) homologically graded spectral sequence. Strong convergence (Remark \ref{rem:strong_convergence}) follows immediately from the first part of the proof of Theorem \ref{thm:hurewicz} by taking $n=0$.
\end{proof}

We defer the proof of the following to Section \ref{sec:homotopical_analysis_forgetful_functors}.

\begin{prop}
\label{prop:connectivity_of_smash_powers_of_maps}
Consider symmetric sequences in $\capR$-modules. Let $\function{f}{X}{Z}$ be a map between $(-1)$-connected objects in $\ModR$ (resp. $\SymSeq$). Let $m\in\ZZ$, $n\geq -1$, and $t\geq 1$. Assume that $\capR$ is $(-1)$-connected.
\begin{itemize}
\item[(a)] If $X,Z$ are flat stable cofibrant and $f$ is $n$-connected, then $X^{\wedge t}\rarrow Z^{\wedge t}$ (resp. $X^{\tensorcheck t}\rarrow Z^{\tensorcheck t}$) is $n$-connected.
\item[(b)] If $B\in\ModR^{\Sigma_t^\op}$ (resp. $B\in\SymSeq^{\Sigma_t^\op}$) is $m$-connected, $X,Z$ are positive flat stable cofibrant and $f$ is $n$-connected, then $B\Smash_{\Sigma_t}X^{\wedge t}\rarrow B\Smash_{\Sigma_t}Z^{\wedge t}$ (resp. $B\tensorcheck_{\Sigma_t}X^{\tensorcheck t}\rarrow B\tensorcheck_{\Sigma_t}Z^{\tensorcheck t}$) is $(m+n+1)$-connected.
\end{itemize}
\end{prop}

\begin{prop}
\label{prop:holim_and_n_connected_maps}
Let $n\in\ZZ$. If $\{A_k\}\rarrow\{B_k\}$ is a map of towers in symmetric spectra (resp. $\capR$-modules)  that is levelwise $n$-connected, then the induced map $\holim_k A_k\rarrow\holim_k B_k$  is $(n-1)$-connected.
\end{prop}

\begin{proof}
This follows from the short exact sequences in Proposition \ref{prop:short_exact_sequence}.
\end{proof}

\begin{proof}[Proof of Theorem \ref{thm:relative_hurewicz}]
It suffices to consider the case of left $\capO$-modules. By Theorem \ref{thm:comparing_homotopy_completion_towers} and Propositions \ref{prop:replacement_of_operads} and \ref{prop:forgetful_functor_commutes_with_holim}, we can restrict to operads $\capO$ satisfying Cofibrancy Condition \ref{CofibrancyCondition}. 

We first prove part (c), where it is enough to consider the following special case. Let $X\rarrow Y$ be a map of left $\capO$-modules between cofibrant objects in $\LtO$ such that the induced map $\tau_1\capO\circ_\capO X\rarrow \tau_1\capO\circ_\capO Y$ is an $n$-connected map between $(-1)$-connected objects. Consider the corresponding commutative diagram \eqref{eq:diagram_of_fibration_sequences_bar} in $\SymSeq$. If $k=2$, then the right-hand vertical map is $n$-connected  by Proposition \ref{prop:fattened_version_of_tower}. It follows from  Proposition \ref{prop:refined_bar_construction_calculation_for_homotopy_fiber}, Proposition \ref{prop:cofibrant_tau_1_O_algebras_and_bar_constructions}, and Propositions \ref{prop:connectivity}, \ref{prop:connectivity_of_smash_powers_of_maps}, and \ref{prop:connectivity_of_simplicial_maps_spectra} that the left-hand vertical map is $n$-connected for each $k\geq 2$. Hence by Proposition \ref{prop:pushout_diagram_for_bar_construction_tower} and induction on $k$, the middle vertical map is $n$-connected for each $k\geq 2$, and Proposition \ref{prop:holim_and_n_connected_maps} finishes the proof of part (b). 

Consider part (b). It is enough to consider the following special case. Let $X\rarrow Y$ be an $(n-1)$-connected map of left $\capO$-modules between $(-1)$-connected cofibrant objects in $\LtO$. Consider the corresponding commutative diagram \eqref{eq:diagram_of_fibration_sequences_bar} in $\SymSeq$. It follows from Propositions \ref{prop:connectivity}, \ref{prop:connectivity_of_smash_powers_of_maps}, and \ref{prop:connectivity_of_simplicial_maps_spectra} that the right-hand vertical map is $(n-1)$-connected for $k=2$, and hence by Proposition \ref{prop:fattened_version_of_tower} the induced map $\tau_1\capO\circ_\capO X\rarrow \tau_1\capO\circ_\capO Y$ is $(n-1)$-connected.

Consider part (a). Proceeding as above for part (c), we know that for each $k\geq 1$ the induced map
$
  \tau_k\capO\circ_\capO X \rarrow
  \tau_k\capO\circ_\capO Y
$
is $n$-connected, and hence the bottom horizontal map in the 
\begin{align*}
\xymatrix{
  \pi_iX\ar[r]\ar[d] & \pi_iY\ar[d]\\
  \pi_i(\tau_k\capO\circ_\capO X)\ar[r] &
  \pi_i(\tau_k\capO\circ_\capO Y)
}
\end{align*}
commutative diagram is an isomorphism for every $i<n$ and a surjection for $i=n$. Since $X,Y$ are $0$-connected by assumption, it follows from strong convergence of the homotopy completion tower (proof of Theorem \ref{MainTheorem}(a)) that the vertical maps are isomorphisms for $k\geq i$, and hence the top horizontal map is an isomorphism for every $i<n$ and a surjection for $i=n$. Part (b) implies the converse.

Consider part (d). By arguing as in the proof of Theorem \ref{thm:hurewicz}, it follows that the layers of the homotopy completion tower are $(n-1)$-connected. Hence by  Proposition \ref{prop:short_exact_sequence} the homotopy limit of this tower is $(n-1)$-connected, which finishes the proof.
\end{proof}

\subsection{Finiteness and the homotopy completion tower}

The purpose of this subsection is to prove Theorem \ref{thm:finiteness}. The following homotopy spectral sequence for a simplicial symmetric spectrum is well known; for a recent reference, see \cite[X.2.9]{EKMM} and \cite[4.3]{Jardine_generalized_etale}.

\begin{prop}
\label{prop:homotopy_spectral_sequence}
Let $Y$ be a simplicial symmetric spectrum. There is a natural homologically graded spectral sequence in the right-half plane such that
\begin{align*}
  E^2_{p,q} = H_p(\pi_q(Y))\Longrightarrow\pi_{p+q}(|Y|)
\end{align*}
Here, $\pi_q(Y)$ denotes the simplicial abelian group obtained by applying $\pi_q$ levelwise to $Y$.
\end{prop}

The following finiteness properties for realization will be useful.

\begin{prop}
\label{prop:realzn_preserves_finiteness_properties}
Let $Y$ be a simplicial symmetric spectrum. Let $m\in\ZZ$. Assume that $Y$ is levelwise $m$-connected.
\begin{itemize}
\item[(a)] If $\pi_k Y_n$ is finite for every $k,n$, then $\pi_k |Y|$ is finite for every $k$.
\item[(b)] If $\pi_k Y_n$ is a finitely generated abelian group for every $k,n$, then $\pi_k|Y|$ is a finitely generated abelian group for every $k$.
\end{itemize}
\end{prop}

\begin{proof}
This follows from Proposition \ref{prop:homotopy_spectral_sequence}.
\end{proof}

Recall the following Eilenberg-Moore type spectral sequences; for a recent reference, see \cite[IV.4--IV.6]{EKMM}.

\begin{prop}
\label{prop:eilenberg_moore}
Let $t\geq 1$. Let $X,Y$ be $\capR$-modules with a right (resp. left) $\Sigma_t$-action. There is a natural homologically graded spectral sequence in the right-half plane such that
\begin{align*}
  E^2_{p,q} &= \Tor^{\pi_*\capR[\Sigma_t]}_{p,q}(\pi_*X,\pi_*Y)
  \Longrightarrow\pi_{p+q}(X\wedge^\LL_{\Sigma_t} Y).
\end{align*}
Here, $\capR[\Sigma_t]$ is the group algebra spectrum and $\wedge^\LL_{\Sigma_t}$ is the total left derived functor of $\wedge_{\Sigma_t}$.
\end{prop}

The following proposition, which is well known to the experts, will be needed in the proof of Proposition \ref{prop:useful_finiteness_properties_dervied_smash} below; since it is a key ingredient in the proof of Theorem \ref{thm:finiteness}, and since we are unaware of an appropriate reference in literature, we give a concise homotopy theoretic proof in Section \ref{sec:homotopical_analysis_forgetful_functors}.

\begin{prop}
\label{prop:finiteness_derived_tensor_uses_bar_construction}
Let $\capA$ be any monoid object in $(\Chaincx_\ZZ,\tensor,\ZZ)$. Let $M,N$ be unbounded chain complexes over $\ZZ$ with a right (resp. left) action of $\capA$. Let $m\in\ZZ$. Assume that $\capA$ is $(-1)$-connected, $M,N$ are $m$-connected, and $H_kM,H_k\capA$ are finitely generated abelian groups for every $k$. 
\begin{itemize}
\item[(a)] If $H_kN$ is finite for every $k$, then $H_k(M\tensor^\LL_\capA N)$ is finite for every $k$.
\item[(b)] If $H_kN$ is a finitely generated abelian group for every $k$, then $H_k(M\tensor^\LL_\capA N)$ is a finitely generated abelian group for every $k$.
\end{itemize}
Here, $\tensor^\LL_\capA$ is the total left derived functor of $\tensor_\capA$.
\end{prop}

\begin{prop}
\label{prop:useful_finiteness_properties_dervied_smash}
Let $t\geq 1$. Let $X,Y$ be $\capR$-modules with a right (resp. left) $\Sigma_t$-action. Let $m\in\ZZ$. Assume that $\capR$ is $(-1)$-connected, $X,Y$ are $m$-connected, and $\pi_k X,\pi_k\capR$ are finitely generated abelian groups for every $k$.
\begin{itemize}
\item[(a)] If $\pi_k Y$ is finite for every $k$, then $\pi_k(X\wedge^\LL_{\Sigma_t} Y)$ is finite for every $k$.
\item[(b)] If $\pi_k Y$ is a finitely generated abelian group for every $k$, then $\pi_k(X\wedge^\LL_{\Sigma_t} Y)$ is a finitely generated abelian group for every $k$.
\end{itemize}
Here, $\wedge^\LL_{\Sigma_t}$ is the total left derived functor of $\wedge_{\Sigma_t}$.
\end{prop}

\begin{proof}
Part (a) follows from Propositions \ref{prop:eilenberg_moore} and \ref{prop:finiteness_derived_tensor_uses_bar_construction}, and the proof of part (b) is similar.
\end{proof}

\begin{proof}[Proof of Theorem \ref{thm:finiteness}]
It suffices to consider the case of left $\capO$-modules. By Theorem \ref{thm:comparing_homotopy_completion_towers} and Propositions \ref{prop:replacement_of_operads} and \ref{prop:forgetful_functor_commutes_with_holim}, we can restrict to operads $\capO$ satisfying Cofibrancy Condition \ref{CofibrancyCondition}. We first prove part (a), for which it suffices to consider the following special case. Let $X$ be a cofibrant left $\capO$-module such that $\tau_1\capO\circ_\capO X$ is $0$-connected and $\pi_i(\tau_1\capO\circ_\capO X)$ is objectwise finite for every $i$. Consider the cofiber sequences
\begin{align*}
\xymatrix{
  |\BAR(i_k\capO,\capO,X)|\ar[r] & 
  |\BAR(\tau_k\capO,\capO,X)|\ar[r] & 
  |\BAR(\tau_{k-1}\capO,\capO,X)|
}
\end{align*}
in $\SymSeq$. We know by Proposition \ref{prop:fattened_version_of_tower} that $\pi_i|\BAR(\tau_{1}\capO,\capO,X)|$ is objectwise finite for every $i$, hence  by Proposition \ref{prop:refined_bar_construction_calculation_for_homotopy_fiber}, Proposition \ref{prop:cofibrant_tau_1_O_algebras_and_bar_constructions}, and Propositions \ref{prop:realzn_preserves_finiteness_properties} and \ref{prop:useful_finiteness_properties_dervied_smash}, $\pi_i|\BAR(i_k\capO,\capO,X)|$ is objectwise finite for every $i$. By Proposition \ref{prop:pushout_diagram_for_bar_construction_tower} and induction on $k$, it follows that $\pi_i|\BAR(\tau_k\capO,\capO,X)|$ is objectwise finite for every $i$ and $k$. Hence by the first part of the proof of Theorem \ref{thm:hurewicz} (by taking $n=0$) it follows easily that $\pi_i(X^\hwedge)$ is objectwise finite for every $i$. If furthermore $X$ is $0$-connected, then by Theorem \ref{MainTheorem}(a) the natural coaugmentation $X\wequiv X^\hwedge$ is a weak equivalence which finishes the proof of part (a). The proof of part (b) is similar.
\end{proof}

\section{Homotopical analysis of the forgetful functors}
\label{sec:homotopical_analysis_forgetful_functors}

The purpose of this section is to prove Theorem \ref{thm:homotopical_analysis_of_forgetful_functors} together with several closely related technical results on the homotopical properties of the forgetful functors. We will also prove Theorem \ref{thm:rigidification} and Propositions \ref{prop:replacement_of_operads}, \ref{prop:connectivity}, \ref{prop:connectivity_of_smash_powers_of_maps}, and \ref{prop:finiteness_derived_tensor_uses_bar_construction}, each of which uses constructions or results established below in Section \ref{sec:homotopical_analysis_forgetful_functors}. It will be useful to work in the following context.

\begin{assumption}
From now on in this section we assume that $(\CC,\Smash,S)$ is a closed symmetric monoidal category with all small limits and colimits. In particular, $\CC$ has an initial object $\emptyset$ and a terminal object $*$.
\end{assumption}

In some of the propositions that follow involving homotopical properties of $\capO$-algebras and left $\capO$-modules, we will explicitly assume the following.

\begin{homotopical_assumption}
\label{HomotopicalAssumption}
If $\capO$ is an operad in $\CC$, assume that 
\begin{itemize}
\item[(i)] $\CC$ is a cofibrantly generated model category in which the generating cofibrations and acyclic cofibrations have small domains \cite[2.2]{Schwede_Shipley}, and that with respect to this model structure $(\CC,\Smash,S)$ is a monoidal model category \cite[3.1]{Schwede_Shipley}; and
\item[(ii)] the following model structure exists on $\AlgO$ (resp. $\LtO$): the model structure on $\AlgO$ (resp. $\LtO$) has weak equivalences and fibrations created by the forgetful functor $U$ \eqref{eq:free_forgetful_adjunction}; i.e., the weak equivalences are the underlying weak equivalences and the fibrations are the underlying fibrations.
\end{itemize}
\end{homotopical_assumption}

\begin{rem}
The main reason for working in the generality of a monoidal model category $(\CC,\Smash)$ is because when we start off with arguments using the properties of a particular monoidal model category, say, $(\ModR,\Smash)$, we are naturally led to need the corresponding results in the diagram category $(\SymSeq,\tensorcheck)$, and in the diagram category $(\SymArray,\tensortilde)$ (e.g., Proposition \ref{prop:cofibrant_operad_symmetric_array_underlying}). So working in the generality of a monoidal model category allows us to give a single proof that works for several different contexts. For instance, we also use the results in this section in the contexts of both symmetric spectra and unbounded chain complexes, even when proving the main theorems only in the context of symmetric spectra (e.g., in the proof of Proposition \ref{prop:finiteness_derived_tensor_uses_bar_construction}).
\end{rem}

\begin{defn}\label{def:symmetric_array}
Consider symmetric sequences in $\CC$. A \emph{symmetric array} in $\CC$ is a symmetric sequence in $\SymSeq$; i.e., a functor $\functor{A}{\Sigma^\op}{\SymSeq}$. Denote by $\SymArray:=\SymSeq^{\Sigma^\op}$ the category of symmetric arrays in $\CC$ and their natural transformations. 
\end{defn}

Recall from \cite{Harper_Spectra} the following proposition.

\begin{prop}
\label{prop:coproduct_modules}
Let $\capO$ be an operad in $\CC$, $A\in\AlgO$ (resp. $A\in\LtO$), and $Y\in\CC$ (resp. $Y\in\SymSeq$). Consider any coproduct in $\AlgO$ (resp. $\LtO$) of the form $A\amalg\capO\circ(Y)$ (resp. $A\amalg(\capO\circ Y)$). There exists a symmetric sequence $\capO_A$ (resp. symmetric array $\capO_A$) and natural isomorphisms
\begin{align*}
  A\amalg\capO\circ(Y) \Iso 
  \coprod\limits_{q\geq 0}\capO_A[\mathbf{q}]
  \Smash_{\Sigma_q}Y^{\wedge q}\quad
  \Bigl(\text{resp.}\quad
  A\amalg(\capO\circ Y) \Iso 
  \coprod\limits_{q\geq 0}\capO_A[\mathbf{q}]
  \tensorcheck_{\Sigma_q}Y^{\tensorcheck q}
    \Bigr)
\end{align*}
in the underlying category $\CC$ (resp. $\SymSeq$). If $q\geq 0$, then $\capO_A[\mathbf{q}]$ is naturally isomorphic to a colimit of the form
\begin{align*}
  \capO_A[\mathbf{q}]&\Iso
  \colim\biggl(
  \xymatrix{
    \coprod\limits_{p\geq 0}\capO[\mathbf{p}\boldsymbol{+}\mathbf{q}]
    \Smash_{\Sigma_p}A^{\wedge p} & 
    \coprod\limits_{p\geq 0}\capO[\mathbf{p}\boldsymbol{+}\mathbf{q}]
    \Smash_{\Sigma_p}(\capO\circ (A))^{\wedge p}\ar@<-0.5ex>[l]^-{d_1}
    \ar@<-1.5ex>[l]_-{d_0}
  }
  \biggl),\\
  \text{resp.}\quad
  \capO_A[\mathbf{q}]&\Iso
  \colim\biggl(
  \xymatrix{
    \coprod\limits_{p\geq 0}\capO[\mathbf{p}\boldsymbol{+}\mathbf{q}]
    \Smash_{\Sigma_p}A^{\tensorcheck p} & 
    \coprod\limits_{p\geq 0}\capO[\mathbf{p}\boldsymbol{+}\mathbf{q}]
    \Smash_{\Sigma_p}(\capO\circ A)^{\tensorcheck p}\ar@<-0.5ex>[l]^-{d_1}
    \ar@<-1.5ex>[l]_-{d_0}
  }
  \biggl),
\end{align*}
in $\CC^{\Sigma_q^\op}$ (resp. $\SymSeq^{\Sigma_q^\op}$), with $d_0$ induced by operad multiplication and $d_1$ induced by the left $\capO$-action map $\function{m}{\capO\circ (A)}{A}$ (resp. $\function{m}{\capO\circ A}{A}$). 
\end{prop}

\begin{rem}
Other possible notations for $\capO_A$ include $\U_\capO(A)$ or $\U(A)$; these are closer to the notation used in \cite{Elmendorf_Mandell, Mandell} and are not to be confused with the forgetful functors. It is interesting to note---although we will not use it in this paper---that in the context of $\capO$-algebras the symmetric sequence $\capO_A$ has the structure of an operad; it parametrizes $\capO$-algebras under $A$ and is sometimes called the enveloping operad for $A$.
\end{rem}

\begin{prop}
\label{prop:OA_commutes_with_certain_colimits}
Let $\capO$ be an operad in $\CC$ and let $q\geq 0$. Then the functor 
$
  \function{\capO_{(-)}[\mathbf{q}]}{\AlgO}{\CC^{\Sigma_q^\op}}
$ (resp. 
$
  \function{\capO_{(-)}[\mathbf{q}]}{\LtO}{\SymSeq^{\Sigma_q^\op}}
$) preserves reflexive coequalizers and filtered colimits.
\end{prop}

\begin{proof}
This follows from Proposition \ref{prop:basic_properties_LTO} and \cite[5.7]{Harper_Modules}.
\end{proof}

\begin{prop}
\label{prop:relating_the_OA_constructions}
Let $\capO$ be an operad in $\CC$ and $A$ an $\capO$-algebra. For each $q\geq 0$, $\capO_{\hat{A}}[\mathbf{q}]$ is concentrated at $0$ with value $\capO_A[\mathbf{q}]$; i.e., $\capO_{\hat{A}}[\mathbf{q}]\Iso\widehat{\capO_A[\mathbf{q}]}$.
\end{prop}

\begin{proof}
This follows from Proposition \ref{prop:coproduct_modules}, together with \eqref{eq:tensor_check_calc} and \eqref{eq:circ_product_and_evaluate_at_zero}.
\end{proof}

\begin{defn}\label{def:filtration_setup_modules}
Let $\function{i}{X}{Y}$ be a morphism in $\CC$ (resp. $\SymSeq$) and $t\geq 1$. Define $Q_0^t:=X^{\wedge t}$ (resp. $Q_0^t:=X^{\tensorcheck t}$) and $Q_t^t:=Y^{\wedge t}$ (resp. $Q_t^t:=Y^{\tensorcheck t}$). For $0<q<t$ define $Q_q^t$ inductively by the left-hand (resp. right-hand) pushout diagrams
\begin{align*}
\xymatrix{
  \Sigma_t\cdot_{\Sigma_{t-q}\times\Sigma_{q}}X^{\wedge(t-q)}
  \Smash Q_{q-1}^q\ar[d]^{i_*}\ar[r]^-{\pr_*} & Q_{q-1}^t\ar[d]\\
  \Sigma_t\cdot_{\Sigma_{t-q}\times\Sigma_{q}}X^{\wedge(t-q)}
  \Smash Y^{\wedge q}\ar[r] & Q_q^t
}\quad
\xymatrix{
  \Sigma_t\cdot_{\Sigma_{t-q}\times\Sigma_{q}}X^{\tensorcheck(t-q)}
  \tensorcheck Q_{q-1}^q\ar[d]^{i_*}\ar[r]^-{\pr_*} & Q_{q-1}^t\ar[d]\\
  \Sigma_t\cdot_{\Sigma_{t-q}\times\Sigma_{q}}X^{\tensorcheck(t-q)}
  \tensorcheck Y^{\tensorcheck q}\ar[r] & Q_q^t
}
\end{align*}
in $\CC^{\Sigma_t}$ (resp. $\SymSeq^{\Sigma_t}$). We sometimes denote $Q_q^t$ by $Q_q^t(i)$ to emphasize in the notation the map $\function{i}{X}{Y}$. The maps $\pr_*$ and $i_*$ are the obvious maps induced by $i$ and the appropriate projection maps. 
\end{defn}

The following proposition is proved in \cite{Harper_Spectra} and is closely related to a similar construction in \cite{Elmendorf_Mandell}; for other approaches to these types of filtrations compare \cite{Fresse_modules, Schwede_Shipley}.

\begin{prop}\label{prop:small_arg_pushout_modules}
Let $\capO$ be an operad in $\CC$, $A\in\AlgO$ (resp. $A\in\LtO$), and $\function{i}{X}{Y}$ in $\CC$ (resp. $\SymSeq$). Consider any pushout diagram in $\AlgO$ (resp. $\LtO$) of the form
\begin{align}\label{eq:small_arg_pushout_modules}
\xymatrix{
  \capO\circ (X)\ar[r]^-{f}\ar[d]^{\id\circ (i)} & A\ar[d]^{j}\\
  \capO\circ (Y)\ar[r] & B.
}\quad\quad\text{resp.}\quad
\xymatrix{
  \capO\circ X\ar[r]^-{f}\ar[d]^{\id\circ i} & A\ar[d]^{j}\\
  \capO\circ Y\ar[r] & B.
}
\end{align}
The pushout in \eqref{eq:small_arg_pushout_modules} is naturally isomorphic to a filtered colimit of the form
$
  B\Iso 
  \colim\bigl(
  \xymatrix@1{
    A_0\ar[r]^{j_1} & A_1\ar[r]^{j_2} & A_2\ar[r]^{j_3} & \dotsb
  }
  \bigr)
$
in the underlying category $\CC$ (resp. $\SymSeq$), with $A_0:=\capO_A[\mathbf{0}]\Iso A$ and $A_t$ defined inductively by pushout diagrams in $\CC$ (resp. $\SymSeq$) of the form
\begin{align}\label{eq:good_filtration_modules}
\xymatrix{
  \capO_A[\mathbf{t}]\Smash_{\Sigma_t}Q_{t-1}^t\ar[d]^{\id\wedge_{\Sigma_t}i_*}
  \ar[r]^-{f_*} & A_{t-1}\ar[d]^{j_t}\\
  \capO_A[\mathbf{t}]\Smash_{\Sigma_t}Y^{\wedge t}\ar[r]^-{\xi_t} & A_t
}\quad\quad\text{resp.}\quad
\xymatrix{
  \capO_A[\mathbf{t}]\tensorcheck_{\Sigma_t}Q_{t-1}^t\ar[d]^{\id\tensorcheck_{\Sigma_t}i_*}
  \ar[r]^-{f_*} & A_{t-1}\ar[d]^{j_t}\\
  \capO_A[\mathbf{t}]\tensorcheck_{\Sigma_t}Y^{\tensorcheck t}\ar[r]^-{\xi_t} & A_t
}
\end{align}
\end{prop}

We are now in a good position to prove Theorem \ref{thm:homotopical_analysis_of_forgetful_functors}.

\begin{proof}[Proof of Theorem \ref{thm:homotopical_analysis_of_forgetful_functors}]
It suffices to consider the case of left $\capO$-modules. Consider part (a). Let $\function{i}{X}{Y}$ be a generating cofibration in $\SymSeq$ with the positive flat stable model structure, and consider the pushout diagram
\begin{align}
\label{eq:gluing_on_cells_proof_of_forgetful}
\xymatrix{
  \capO\circ X\ar[r]\ar[d] & Z_0\ar[d]^{i_0}\\
  \capO\circ Y\ar[r] & Z_1
}
\end{align}
in $\LtO$. Assume $Z_0$ is cofibrant in $\LtO$; let's verify that $i_0$ is a positive flat stable cofibration in $\SymSeq$. Let $A:=Z_0$. By Proposition \ref{prop:small_arg_pushout_modules}, we know $Z_1$ is naturally isomorphic to a filtered colimit of the form 
$
  Z_1\Iso 
  \colim\bigl(
  \xymatrix@1{
    A_0\ar[r]^{j_1} & A_1\ar[r]^{j_2} & A_2\ar[r]^{j_3} & \dotsb
  }
  \bigr)
$
in the underlying category $\SymSeq$, and hence it suffices to verify each $j_t$ is a positive flat stable cofibration in $\SymSeq$. By the construction of $j_t$ in Proposition \ref{prop:small_arg_pushout_modules}, it is enough to check that each $\id\tensorcheck_{\Sigma_t}i_*$ in \eqref{eq:good_filtration_modules} is a positive flat stable cofibration in $\SymSeq$. The generating cofibrations in $\SymSeq$ with the positive flat stable model structure have cofibrant domains, and by Proposition \ref{prop:generating_cofibration} we know that $i_*$ is a cofibration between cofibrant objects in $\SymSeq^{\Sigma_t}$ with the positive flat stable model structure. We need therefore only show that $\id\tensorcheck_{\Sigma_t}i_*$ is a flat stable cofibration in $\SymSeq$. 

Suppose $\function{p}{C}{D}$ is a flat stable acyclic fibration in $\SymSeq$. We want to verify $\id\tensorcheck_{\Sigma_t}i_*$ has the left lifting property with respect to $p$. Consider any such lifting problem; we want to verify that the corresponding solid commutative diagram
\begin{align}
\label{eq:final_lifting_argument}
\xymatrix{
  Q_{t-1}^t\ar[d]_{i_*}\ar[r] & 
  \Map^\tensorcheck(\capO_A[\mathbf{t}], C)
  \ar[d]^{(*)}\\
  Y^{\tensorcheck t}\ar[r]\ar@{.>}[ur] & 
  \Map^\tensorcheck(\capO_A[\mathbf{t}], D)
}
\end{align}
in $\SymSeq^{\Sigma_t^\op}$ has a lift. We know that $i_*$ is a flat stable cofibration in $\SymSeq^{\Sigma_t^\op}$, hence it is enough to verify that $(*)$ is a flat stable acyclic fibration in $\SymSeq$. By Proposition \ref{prop:analysis_of_OA_symmetric_spectra} below, $\capO_A[\mathbf{t}]$ is flat stable cofibrant in $\SymSeq$, hence we know  that $(*)$ has the desired property by \cite[6.1]{Harper_Modules}, which finishes the argument that $i_0$ is a positive flat stable cofibration in $\SymSeq$. Consider a sequence
$
\xymatrix@1{
  Z_0\ar[r]^{i_0} & Z_1\ar[r]^{i_1} & Z_2\ar[r]^{i_2} & \dotsb
}
$
of pushouts of maps as in \eqref{eq:gluing_on_cells_proof_of_forgetful}, and let $Z_\infty:=\colim_k Z_k$. Consider the naturally occurring map $\function{i_\infty}{Z_0}{Z_\infty}$, and assume $Z_0$ is cofibrant in $\LtO$. By the argument above, we know this is a sequence of positive flat stable cofibrations in $\SymSeq$, hence $i_\infty$ is a positive flat stable cofibration in $\SymSeq$. Since every cofibration $A\rarrow B$ in $\LtO$ is a retract of a (possibly transfinite) composition of pushouts of maps as in \eqref{eq:gluing_on_cells_proof_of_forgetful}, starting with $Z_0=A$, where $A$ is assumed to be cofibrant in $\LtO$, finishes the proof of part (a). Part (b) follows from part (a) by taking $A=\capO\circ\emptyset$, together with the natural isomorphism $\capO\circ\emptyset\Iso\widehat{\capO[\mathbf{0}]}$.
\end{proof}

\subsection{Homotopical analysis of the $\capO_A$ constructions}

The purpose of this subsection is to prove the following proposition, which we used in the proof of Theorem \ref{thm:homotopical_analysis_of_forgetful_functors}. It provides a homotopical analysis of the $\capO_A$ constructions, and a key ingredient in its proof is a filtration of $\capO_A$ (Proposition \ref{prop:filtering_OA}). We will also prove Proposition \ref{prop:analysis_of_OA_monoidal_model_category} and Theorem \ref{thm:homotopical_analysis_of_forgetful_functors_monoidal}, which are analogs of Proposition \ref{prop:analysis_of_OA_symmetric_spectra} and Theorem \ref{thm:homotopical_analysis_of_forgetful_functors}, respectively. These analogous results are applicable to a general class of monoidal model categories, but at the cost of requiring stronger assumptions.

The following proposition is motivated by \cite[13.6]{Mandell}.

\begin{prop}
\label{prop:analysis_of_OA_symmetric_spectra}
Let $\capO$ be an operad in $\capR$-modules such that $\capO[\mathbf{r}]$ is flat stable cofibrant in $\ModR$ for each $r\geq 0$. If $A$ is a cofibrant $\capO$-algebra (resp. cofibrant left $\capO$-module), then $\capO_A[\mathbf{r}]$ is flat stable cofibrant in $\ModR$ (resp. $\SymSeq$) for each $r\geq 0$.
\end{prop}

The following proposition is closely related to \cite[13.6]{Mandell}.

\begin{prop}
\label{prop:analysis_of_OA_monoidal_model_category}
Let $\capO$ be an operad in $\CC$. Suppose that Homotopical Assumption \ref{HomotopicalAssumption} is satisfied, and assume that $\capO[\mathbf{r}]$ is cofibrant in $\CC^{\Sigma_r^\op}$ for each $r\geq 0$.  If $A$ is a cofibrant $\capO$-algebra (resp. cofibrant left $\capO$-module), then $\capO_A[\mathbf{r}]$ is cofibrant in $\CC^{\Sigma_r^\op}$ (resp. $\SymSeq^{\Sigma_r^\op}$) for each $r\geq 0$.
\end{prop}

\begin{thm}
\label{thm:homotopical_analysis_of_forgetful_functors_monoidal}
Let $\capO$ be an operad in $\CC$. Suppose that Homotopical Assumption \ref{HomotopicalAssumption} is satisfied, and assume that $\capO[\mathbf{r}]$ is cofibrant in $\CC^{\Sigma_r^\op}$ for each $r\geq 0$.
\begin{itemize}
\item[(a)] If $\function{j}{A}{B}$ is a cofibration between cofibrant objects in $\AlgO$ (resp. $\LtO$), then $j$ is a cofibration in the underlying category $\CC$ (resp. $\SymSeq$).
\item[(b)] If $A$ is a cofibrant $\capO$-algebra (resp. cofibrant left $\capO$-module), then $A$ is cofibrant in the underlying category $\CC$ (resp. $\SymSeq$).
\end{itemize}
\end{thm}

\begin{proof}
It suffices to consider the case of left $\capO$-modules. Consider part (a). This follows exactly as in the proof of Theorem \ref{thm:homotopical_analysis_of_forgetful_functors}, except using Proposition \ref{prop:analysis_of_OA_monoidal_model_category} instead of Proposition \ref{prop:analysis_of_OA_symmetric_spectra}, and replacing the lifting problem  \eqref{eq:final_lifting_argument} with a lifting problem of the form
\begin{align*}
\xymatrix{
  \emptyset\ar[r]\ar[d] & 
  \Map^\tensorcheck(Y^{\tensorcheck t}, C)\ar[d]^{(*)}\\
  \capO_A[\mathbf{t}]\ar[r]\ar@{.>}[ur] & 
  \Map^\tensorcheck(Q_{t-1}^t,C)
  \times_{\Map^\tensorcheck(Q_{t-1}^t,D)}
  \Map^\tensorcheck(Y^{\tensorcheck t},D)
}
\end{align*}
in $\SymSeq^{\Sigma_t^\op}$. Part (b) follows from part (a) by taking $A=\capO\circ\emptyset$, together with the natural isomorphism $\capO\circ\emptyset\Iso\widehat{\capO[\mathbf{0}]}$, since $\capO[\mathbf{0}]$ is cofibrant in $\CC$.
\end{proof}

When working with certain arguments involving left modules over an operad, we are naturally led to replace $(\CC,\Smash,S)$ with $(\SymSeq,\tensorcheck,1)$ as the underlying closed symmetric monoidal category. In particular, we will consider symmetric sequences in $(\SymSeq,\tensorcheck,1)$, i.e., symmetric arrays (Defintion \ref{def:symmetric_array}), together with the corresponding tensor product and circle product. To avoid notational confusion, we will use $\tensortilde$ to denote the tensor product of symmetric arrays and $\circtilde$ to denote the circle product of symmetric arrays. We summarize their structure and properties in the following propositions.

\begin{prop}
Consider symmetric sequences in $\CC$. Let $A_1,\dotsc,A_t$ and $A,B$ be symmetric arrays in $\CC$. Then the tensor product $A_1\tensortilde\dotsb\tensortilde A_t\in\SymArray$ and the circle product $A\circtilde B\in\SymArray$ satisfy objectwise the natural isomorphisms
\begin{align}
  \label{eq:tensor_tilde_calc}
  (A_1\tensortilde\dotsb\tensortilde A_t)[\mathbf{r}]  
  &\Iso
  \coprod_{r_1+\dotsb +r_t=r}A_1[\mathbf{r_1}]\tensorcheck\dotsb\tensorcheck 
  A_t[\mathbf{r_t}]\underset{{\Sigma_{r_1}\times\dotsb\times
  \Sigma_{r_t}}}{\cdot}\Sigma_{r},\\
  \label{eq:circle_tilde_calc}
  (A\circtilde B)[\mathbf{r}]
  &\Iso 
  \coprod_{t\geq 0}A[\mathbf{t}]\tensorcheck_{\Sigma_t}
  (B^{\tensortilde t})[\mathbf{r}].
\end{align}
\end{prop}

\begin{defn}
Consider symmetric sequences in $\CC$. Let $Z\in\SymSeq$. Define $\tilde{Z}\in\SymArray$ to be the symmetric array such that $\tilde{Z}[\mathbf{t}]\in\SymSeq^{\Sigma_t^\op}$ is concentrated at $0$ with value $Z[\mathbf{t}]$; i.e., $\tilde{Z}[\mathbf{t}]:=\widehat{Z[\mathbf{t}]}$ and hence $\tilde{Z}[\mathbf{t}][\mathbf{0}]=Z[\mathbf{t}]$.
\end{defn}

The adjunction immediately below Definition \ref{defn:hat_construction_embed_at_zero} induces objectwise the adjunction
$
\xymatrix@1{
  \tilde{-}\colon\SymSeq\ar@<0.5ex>[r] & 
  \SymArray:\Ev_0\ar@<0.5ex>[l]
}
$
with left adjoint on top and $\Ev_0$ the functor defined objectwise by $\Ev_0(B)[\mathbf{t}]:=\Ev_0(B[\mathbf{t}])=B[\mathbf{t}][\mathbf{0}]$; i.e., $\Ev_0(B)=B[-][\mathbf{0}]$. Note that $\tilde{-}$ embeds $\SymSeq$ in $\SymArray$ as a full subcategory.

\begin{prop}
Consider symmetric sequences in $\CC$. Let $\capO,A,B\in\SymSeq$ and $X,Y\in\SymArray$. There are natural isomorphisms
\begin{align}
  \label{eq:tilde_respects_monoidal_products}
  &\widetilde{A\tensorcheck B}\Iso\tilde{A}\tensortilde\tilde{B},
  \quad\quad
  \widetilde{A\circ B}\Iso\tilde{A}\circtilde\tilde{B},
  \quad\quad
  \Ev_0(\tilde{\capO}\circtilde Y)\Iso\capO\circ\Ev_0(Y),\\
  \label{eq:evaluate_at_zero_respects_monoidal_products}
  &\Ev_0(X\tensortilde Y)\Iso \Ev_0(X)\tensorcheck\Ev_0(Y),
  \quad\quad
  \Ev_0(X\circtilde Y)\Iso \Ev_0(X)\circ\Ev_0(Y).
\end{align}
\end{prop}

\begin{prop}
Consider symmetric sequences in $\CC$.
\begin{itemize}
\item [(a)] $(\SymArray,\tensortilde,\tilde{1})$ is a closed symmetric monoidal category with all small limits and colimits. The unit for $\tensortilde$, denoted ``\,$\tilde{1}$'', is the symmetric array concentrated at $0$ with value the symmetric sequence $1$.
\item [(b)] $(\SymArray,\circtilde,\tilde{I})$ is a closed monoidal category with all small limits and colimits. The unit for $\circtilde$, denoted ``\,$\tilde{I}$'', is the symmetric array concentrated at $1$ with value the symmetric sequence $1$. Circle product is not symmetric.
\end{itemize}
\end{prop}

Since all of the statements and constructions in earlier sections that were previously described in terms of $(\CC,\Smash,S)$ are equally true for $(\SymSeq,\tensorcheck,1)$, we will cite and use the appropriate statements and constructions without further comment. 

\begin{prop}
Consider symmetric sequences in $\CC$.
\begin{itemize}
\item[(a)] If $\capO$ is an operad in $\CC$, then $\tilde{\capO}$ is an operad in $\SymSeq$.
\item[(b)] If $A$ is a left $\capO$-module, then $\tilde{A}$ is a left $\tilde{\capO}$-module.
\item[(c)] There are adjunctions
\begin{align}
\label{eq:evaluate_at_zero_adjunction}
\xymatrix{
  \SymArray\ar@<0.5ex>[r]^-{\tilde{\capO}\circtilde-} & \Lt_{\tilde{\capO}},\ar@<0.5ex>[l]^-{U}
}\quad
\xymatrix{
  \LtO\ar@<0.5ex>[r]^-{\tilde{-}} & 
  \Lt_{\tilde{\capO}}\ar@<0.5ex>[l]^-{\Ev_0},
}\quad
\xymatrix{
  \Operad(\CC)\ar@<0.5ex>[r]^-{\tilde{-}} & 
  \Operad(\SymSeq)\ar@<0.5ex>[l]^-{\Ev_0}
}
\end{align}
with left adjoints on top, $U$ the forgetful functor and $\Ev_0$ the functor defined objectwise by $\Ev_0(B)[\mathbf{t}]:=\Ev_0(B[\mathbf{t}])=B[\mathbf{t}][\mathbf{0}]$, i.e., $\Ev_0(B)=B[-][\mathbf{0}]$. 
\end{itemize}
Here, $\Operad(\CC)$ denotes the category of operads in $\CC$, and similarly for $\Operad(\SymSeq)$.
\end{prop}

The following two propositions are exercises left to the reader. They will be needed in the proof of Proposition \ref{prop:analysis_of_OA_for_coproducts} below.

\begin{prop}
\label{prop:tilde_commutes_with_OA_constructions}
Let $\capO$ be an operad in $\CC$ and $A$ a left $\capO$-module. For each $q,r\geq 0$, $\tilde{\capO}_{\tilde{A}}[\mathbf{q}][\mathbf{r}]$ is concentrated at $0$ with value $\capO_A[\mathbf{q}][\mathbf{r}]$ (see \eqref{prop:coproduct_modules}); i.e., $\tilde{\capO}_{\tilde{A}}[\mathbf{q}]\Iso\widetilde{\capO_A[\mathbf{q}]}$.
\end{prop}

\begin{prop}
\label{prop:symmetric_sequence_decompositions}
Consider symmetric sequences in $\CC$. Let $B$ be a symmetric sequence (resp. symmetric array) and $r,t\geq 0$. There are natural isomorphisms
\begin{align*}
  B[\mathbf{t}]\Iso
  \Bigl(
  \coprod\limits_{q\geq 0}\widehat{B[\mathbf{q}]}
  \tensorcheck_{\Sigma_q}I^{\tensorcheck q}
  \Bigr)[\mathbf{t}]
  \quad\quad
  \biggl(\text{resp.}\quad
  B[\mathbf{t}][\mathbf{r}]\Iso
  \Bigl(
  \coprod\limits_{q\geq 0}\widetilde{B[\mathbf{q}]}
  \tensortilde_{\Sigma_q}\hat{I}^{\tensortilde q}
  \Bigr)[\mathbf{r}][\mathbf{t}]
  \biggr).
\end{align*}
Here, $\hat{I}$ is the symmetric array concentrated at $0$ with value $I$.
\end{prop}

The following will be needed in the proof of Proposition \ref{prop:filtering_OA} below.

\begin{prop}
\label{prop:analysis_of_OA_for_coproducts}
Let $\capO$ be an operad in $\CC$, $A\in\AlgO$ (resp. $A\in\LtO$), $Y\in\CC$ (resp. $Y\in\SymSeq$) and $q\geq 0$. Consider any coproduct in $\AlgO$ (resp. $\LtO$) of the form $A\amalg\capO\circ (Y)$ (resp. $A\amalg(\capO\circ Y)$). There are natural isomorphisms 
\begin{align*}
  \capO_{A\amalg\capO\circ (Y)}[\mathbf{q}] &\Iso 
  \coprod\limits_{p\geq 0}\capO_A[\mathbf{p}\boldsymbol{+}\mathbf{q}]
  \Smash_{\Sigma_p}Y^{\wedge p},
  \quad
  \capO_{\capO\circ (Y)}[\mathbf{q}] \Iso 
  \coprod\limits_{p\geq 0}\capO[\mathbf{p}\boldsymbol{+}\mathbf{q}]
  \Smash_{\Sigma_p}Y^{\wedge p}\\
  \Bigl(\text{resp.}\quad
  \capO_{A\amalg(\capO\circ Y)}[\mathbf{q}] &\Iso 
  \coprod\limits_{p\geq 0}\capO_A[\mathbf{p}\boldsymbol{+}\mathbf{q}]
  \tensorcheck_{\Sigma_p}Y^{\tensorcheck p},
  \quad
  \capO_{\capO\circ Y}[\mathbf{q}] \Iso 
  \coprod\limits_{p\geq 0}\capO[\mathbf{p}\boldsymbol{+}\mathbf{q}]
  \Smash_{\Sigma_p}Y^{\tensorcheck p}
  \Bigr)
\end{align*}
in $\CC^{\Sigma_q^\op}$ (resp. $\SymSeq^{\Sigma_q^\op}$). In particular, there are natural isomorphisms 
\begin{align}
\label{eq:calculating_OA_for_initial_algebras_and_modules}
  \capO_{\capO\circ (\emptyset)}[\mathbf{q}]\Iso\capO[\mathbf{q}]
  \quad\quad
  \Bigl(\text{resp.}\quad
  \capO_{\capO\circ \emptyset}[\mathbf{q}]\Iso\widehat{\capO[\mathbf{q}]}
  \Bigr)
\end{align}
in $\CC^{\Sigma_q^\op}$ (resp. $\SymSeq^{\Sigma_q^\op}$).
\end{prop}

\begin{proof}
Consider the left-hand natural isomorphisms. Since the case for left $\capO$-modules is more involved, it is useful to consider first the case of $\capO$-algebras. Let $A$ be an $\capO$-algebra and $Y\in\CC$.  Let $Z\in\SymSeq$, and consider the corresponding left $\capO$-module $\hat{A}$ and the corresponding symmetric sequence $\hat{Y}$. It follows easily from Proposition \ref{prop:coproduct_modules} and \cite[proof of 4.7]{Harper_Spectra} that there are natural isomorphisms 
\begin{align}
  \label{eq:first_calculation_proof}
  \hat{A}\amalg(\capO\circ \hat{Y})\amalg(\capO\circ Z)&\Iso
  \coprod\limits_{q\geq 0}\capO_{\hat{A}\amalg(\capO\circ \hat{Y})}[\mathbf{q}]
  \tensorcheck_{\Sigma_q}Z^{\tensorcheck q},\\
  \label{eq:second_calculation_proof}
  \hat{A}\amalg(\capO\circ \hat{Y})\amalg(\capO\circ Z)
  &\Iso \coprod\limits_{q\geq 0}
  \Bigl(
  \coprod\limits_{p\geq 0}\capO_{\hat{A}}[\mathbf{p}\boldsymbol{+}\mathbf{q}]
  \tensorcheck_{\Sigma_p}\hat{Y}^{\tensorcheck p}
  \Bigr)
  \tensorcheck_{\Sigma_q}Z^{\tensorcheck q},
\end{align}
in the underlying category $\SymSeq$. Comparing \eqref{eq:first_calculation_proof} with \eqref{eq:second_calculation_proof} and taking $Z=I$, together with Proposition \ref{prop:relating_the_OA_constructions} and Proposition \ref{prop:symmetric_sequence_decompositions}, gives a natural isomorphism of symmetric sequences of the form
\begin{align*}
  \capO_{A\amalg\capO\circ(Y)}[\mathbf{q}]\Iso
  \coprod_{p\geq 0}\capO_{A}[\mathbf{p}\boldsymbol{+}\mathbf{q}]
  \Smash_{\Sigma_p}Y^{\wedge p},\quad\quad q\geq 0,
\end{align*}
which finishes the proof of the left-hand natural isomorphisms for the case of $\capO$-algebras. 

Consider the case of left $\capO$-modules. Let $A$ be a left $\capO$-module and $Y\in\SymSeq$.  Let $Z\in\SymArray$ and consider the corresponding operad $\tilde{\capO}$ in $\SymSeq$, the corresponding left $\tilde{\capO}$-module $\tilde{A}$ and the corresponding symmetric array $\tilde{Y}$. Arguing as above, by Proposition \ref{prop:coproduct_modules} there is a natural isomorphism 
\begin{align}
  \label{eq:calculation_proof_modules}
  \coprod\limits_{q\geq 0}\tilde{\capO}_{\tilde{A}\amalg(\tilde{\capO}\circtilde \tilde{Y})}[\mathbf{q}]
  \tensortilde_{\Sigma_q}Z^{\tensortilde q}
  \Iso \coprod\limits_{q\geq 0}
  \Bigl(
  \coprod\limits_{p\geq 0}\tilde{\capO}_{\tilde{A}}[\mathbf{p}\boldsymbol{+}\mathbf{q}]
  \tensortilde_{\Sigma_p}\tilde{Y}^{\tensortilde p}
  \Bigr)
  \tensortilde_{\Sigma_q}Z^{\tensortilde q},
\end{align}
in the underlying category $\SymArray$. By \eqref{eq:calculation_proof_modules} and taking $Z=\hat{I}$, together with Proposition \ref{prop:tilde_commutes_with_OA_constructions} and Proposition \ref{prop:symmetric_sequence_decompositions}, gives a natural isomorphism of symmetric arrays of the form
\begin{align*}
  \Bigl(
  \capO_{A\amalg\capO\circ Y}[\mathbf{q}]
  \Bigr)[\mathbf{r}]\Iso
  \Bigl(
  \coprod\limits_{p\geq 0}\capO_{A}[\mathbf{p}\boldsymbol{+}\mathbf{q}]
  \tensorcheck_{\Sigma_p}Y^{\tensorcheck p}
  \Bigr)[\mathbf{r}],\quad\quad
  q,r\geq 0,
\end{align*}
which finishes the proof of the left-hand natural isomorphisms for the case of left $\capO$-modules. The proof of the right-hand natural isomorphisms is similar.
\end{proof}

The following filtrations are motivated by \cite[13.7]{Mandell} and generalize the filtered colimit construction of the form
\begin{align*}
  B\Iso\capO_B[\mathbf{0}]\Iso 
  \colim\bigl(
  \xymatrix@1{
    \capO_A[\mathbf{0}]\ar[r]^{j_1} & A_1\ar[r]^{j_2} & A_2\ar[r]^{j_3} & \dotsb
  }
  \bigr)
\end{align*}
in Proposition \ref{prop:small_arg_pushout_modules} to a filtered colimit construction of $\capO_B[\mathbf{r}]$ for each $r\geq 0$; for other approaches to these types of filtrations compare \cite{Fresse_modules, Schwede_Shipley}.

\begin{prop}
\label{prop:filtering_OA}
Let $\capO$ be an operad in $\CC$, $A\in\AlgO$ (resp. $A\in\LtO$), and $\function{i}{X}{Y}$ in $\CC$ (resp. $\SymSeq$). Consider any pushout diagram in $\AlgO$ (resp. $\LtO$) of the form \ref{eq:small_arg_pushout_modules}. For each $r\geq 0$, $\capO_B[\mathbf{r}]$ is naturally isomorphic to a filtered colimit of the form
\begin{align}
\label{eq:filtered_colimit_modules_refined}
  \capO_B[\mathbf{r}]\Iso 
  \colim\Bigl(
  \xymatrix{
    \capO_A^0[\mathbf{r}]\ar[r]^{j_1} & 
    \capO_A^1[\mathbf{r}]\ar[r]^{j_2} & 
    \capO_A^2[\mathbf{r}]\ar[r]^{j_3} & \dotsb
  }
  \Bigr)
\end{align}
in $\CC^{\Sigma_r^\op}$ (resp. $\SymSeq^{\Sigma_r^\op}$), with $\capO_A^0[\mathbf{r}]:=\capO_A[\mathbf{r}]$ and $\capO_A^t[\mathbf{r}]$ defined inductively by pushout diagrams in $\CC^{\Sigma_r^\op}$ (resp. $\SymSeq^{\Sigma_r^\op}$) of the form
\begin{align}
\label{eq:good_filtration_modules_refined}
\xymatrix{
  \capO_A[\mathbf{t}\boldsymbol{+}\mathbf{r}]\Smash_{\Sigma_t}Q_{t-1}^t\ar[d]^{\id\wedge_{\Sigma_t}i_*}
  \ar[r]^-{f_*} & \capO_A^{t-1}[\mathbf{r}]\ar[d]^{j_t}\\
  \capO_A[\mathbf{t}\boldsymbol{+}\mathbf{r}]\Smash_{\Sigma_t}Y^{\wedge t}\ar[r]^-{\xi_t} & \capO_A^t[\mathbf{r}]
}\quad\quad
\text{resp.}\quad
\xymatrix{
  \capO_A[\mathbf{t}\boldsymbol{+}\mathbf{r}]\tensorcheck_{\Sigma_t}Q_{t-1}^t\ar[d]^{\id\tensorcheck_{\Sigma_t}i_*}
  \ar[r]^-{f_*} & \capO_A^{t-1}[\mathbf{r}]\ar[d]^{j_t}\\
  \capO_A[\mathbf{t}\boldsymbol{+}\mathbf{r}]\tensorcheck_{\Sigma_t}Y^{\tensorcheck t}\ar[r]^-{\xi_t} & \capO_A^t[\mathbf{r}]
}
\end{align}
\end{prop}

\begin{proof}
It suffices to consider the case of left $\capO$-modules. The argument is a generalization of the proof given in \cite[4.20]{Harper_Spectra} for the case $r=0$, hence it is enough to describe the constructions and arguments needed for future reference and for a reader of \cite[4.20]{Harper_Spectra} to be able to follow the proof. It is easy to verify that the pushout in \eqref{eq:small_arg_pushout_modules} may be calculated by a reflexive coequalizer in $\LtO$ of the form
\begin{align}
\label{eq:reflexive_coequalizer_for_desired_pushout}
  B
  \Iso\colim\Bigl(
  \xymatrix{
    A\amalg(\capO\circ Y)
    & A\amalg(\capO\circ X)\amalg(\capO\circ Y)
    \ar@<-0.5ex>[l]_-{\ol{i}}\ar@<0.5ex>[l]^-{\ol{f}}
  }
  \Bigr).
\end{align}
The maps $\ol{i}$ and $\ol{f}$ are induced by maps $\id\circ i_*$ and $\id\circ f_*$, which fit into the commutative diagram
\begin{align}
\label{eq:induced_maps_modules}
\xymatrix{
  A\amalg\bigl(\capO\circ(X\amalg Y)\bigr)\ar@<-0.5ex>[d]_{\ol{i}}\ar@<0.5ex>[d]^{\ol{f}} & 
  \capO\circ(A\amalg X\amalg Y)\ar[l]
  \ar@<-0.5ex>[d]_{\id\circ i_*}\ar@<0.5ex>[d]^{\id\circ f_*} &
  \capO\circ\bigl((\capO\circ A)\amalg X\amalg Y\bigr)\ar@<-0.5ex>[l]_-{d_0}\ar@<0.5ex>[l]^-{d_1}
  \ar@<-0.5ex>[d]_{\id\circ i_*}\ar@<0.5ex>[d]^{\id\circ f_*}\\
  A\amalg(\capO\circ Y) & \capO\circ(A\amalg Y)\ar[l] & 
  \capO\circ\bigl((\capO\circ A)\amalg Y\bigr)
  \ar@<-0.5ex>[l]_-{d_0}\ar@<0.5ex>[l]^-{d_1}
}
\end{align}
in $\LtO$, with rows reflexive coequalizer diagrams, and maps $i_*$ and $f_*$ in $\SymSeq$ induced by $\function{i}{X}{Y}$ and $\function{f}{X}{A}$ in $\SymSeq$. Here we have used the same notation for both $f$ and its adjoint \eqref{eq:free_forgetful_adjunction}. Applying $\capO_{(-)}[\mathbf{r}]$ to \eqref{eq:reflexive_coequalizer_for_desired_pushout} and \eqref{eq:induced_maps_modules}, it follows from Proposition \ref{prop:OA_commutes_with_certain_colimits} that $\capO_{B}[\mathbf{r}]$ may be calculated by a reflexive coequalizer 
\begin{align}
\label{eq:reflexive_coequalizer_for_OA}
  &\capO_B[\mathbf{r}]
  \Iso\colim\Bigl(
  \xymatrix{
    \capO_{A\amalg(\capO\circ Y)}[\mathbf{r}]
    & \capO_{A\amalg(\capO\circ X)\amalg(\capO\circ Y)}[\mathbf{r}]
    \ar@<-0.5ex>[l]_-{\ol{i}}\ar@<0.5ex>[l]^-{\ol{f}}
  }
  \Bigr)\\
\label{eq:induced_maps_modules_refined}
&\xymatrix{
  \capO_{A\amalg(\capO\circ(X\amalg Y))}[\mathbf{r}]\ar@<-0.5ex>[d]_{\ol{i}}\ar@<0.5ex>[d]^{\ol{f}} & 
  \capO_{\capO\circ(A\amalg X\amalg Y)}[\mathbf{r}]\ar[l]
  \ar@<-0.5ex>[d]\ar@<0.5ex>[d] &
  \capO_{\capO\circ((\capO\circ A)\amalg X\amalg Y)}[\mathbf{r}]\ar@<-0.5ex>[l]\ar@<0.5ex>[l]
  \ar@<-0.5ex>[d]\ar@<0.5ex>[d]\\
  \capO_{A\amalg(\capO\circ Y)}[\mathbf{r}] & 
  \capO_{\capO\circ(A\amalg Y)}[\mathbf{r}]\ar[l] & 
  \capO_{\capO\circ((\capO\circ A)\amalg Y)}[\mathbf{r}]
  \ar@<-0.5ex>[l]\ar@<0.5ex>[l]
}
\end{align}
in $\SymSeq^{\Sigma_r^\op}$ of the form \eqref{eq:reflexive_coequalizer_for_OA}, and that the maps $\ol{i}$ and $\ol{f}$ in \eqref{eq:reflexive_coequalizer_for_OA} fit into the commutative diagram \eqref{eq:induced_maps_modules_refined} in $\SymSeq^{\Sigma_r^\op}$, with rows reflexive coequalizer diagrams. By \eqref{eq:reflexive_coequalizer_for_OA}, $\capO_B[\mathbf{r}]$ may be calculated by the colimit of the left-hand column of \eqref{eq:induced_maps_modules_refined} in $\SymSeq^{\Sigma_r^\op}$. By \eqref{eq:induced_maps_modules_refined} and Proposition \ref{prop:analysis_of_OA_for_coproducts}, $f$ induces maps $\ol{f}_{q,p}$ that make the diagrams
\begin{align*}
\xymatrix{
  \capO_{A\amalg(\capO\circ(X\amalg Y))}[\mathbf{r}]\Iso\coprod\limits_{q\geq 0}\coprod\limits_{p\geq 0}
  \Bigl(\ \Bigr)\ar[d]^{\ol{f}} &
  \Bigl(
  \capO_A[\mathbf{p}\boldsymbol{+}\mathbf{q}\boldsymbol{+}\mathbf{r}]\tensorcheck_{\Sigma_p\times\Sigma_q}
  X^{\tensorcheck p}\tensorcheck Y^{\tensorcheck q}
  \Bigr)\ar[l]_-{\inmap_{q,p}}\ar@{.>}[d]^{\ol{f}_{q,p}}\\
  \capO_{A\amalg(\capO\circ Y)}[\mathbf{r}]\Iso\coprod\limits_{t\geq 0}\Bigl(\ \Bigr) &
  \Bigl(
  \capO_A[\mathbf{q}\boldsymbol{+}\mathbf{r}]\tensorcheck_{\Sigma_q}Y^{\tensorcheck q}
  \Bigr)\ar[l]_-{\inmap_q}
}
\end{align*}
in $\SymSeq^{\Sigma_r^\op}$ commute. Similarly, $i$ induces maps $\ol{i}_{q,p}$ that make the diagrams
\begin{align*}
\xymatrix{
  \capO_{A\amalg(\capO\circ(X\amalg Y))}[\mathbf{r}]\Iso\coprod\limits_{q\geq 0}\coprod\limits_{p\geq 0}
  \Bigl(\ \Bigr)\ar[d]^{\ol{i}} &
  \Bigl(
  \capO_A[\mathbf{p}\boldsymbol{+}\mathbf{q}\boldsymbol{+}\mathbf{r}]\tensorcheck_{\Sigma_p\times\Sigma_q}
  X^{\tensorcheck p}\tensorcheck Y^{\tensorcheck q}
  \Bigr)\ar[l]_-{\inmap_{q,p}}\ar@{.>}[d]^{\ol{i}_{q,p}}\\
  \capO_{A\amalg(\capO\circ Y)}[\mathbf{r}]\Iso\coprod\limits_{t\geq 0}\Bigl(\ \Bigr) &
  \Bigl(
  \capO_A[\mathbf{p}\boldsymbol{+}\mathbf{q}\boldsymbol{+}\mathbf{r}]
  \tensorcheck_{\Sigma_{p+q}}Y^{\tensorcheck (p+q)}
  \Bigr)\ar[l]_-{\inmap_{p+q}}
}
\end{align*}
in $\SymSeq^{\Sigma_r^\op}$ commute. 

We can now describe more explicitly what it means to give a cone in $\SymSeq^{\Sigma_r^\op}$ out of the left-hand column of \eqref{eq:induced_maps_modules_refined}. Let $\function{\varphi}{\capO_{A\amalg(\capO\circ Y)}[\mathbf{r}]}{\cdot}$ be a morphism in $\SymSeq^{\Sigma_r^\op}$ and define $\varphi_q:=\varphi\inmap_q$. Then $\varphi\ol{i}=\varphi\ol{f}$ if and only if the diagrams
\begin{align}
\label{eq:cone_data_refined}
\xymatrix{
  \capO_A[\mathbf{p}\boldsymbol{+}\mathbf{q}\boldsymbol{+}\mathbf{r}]
  \tensorcheck_{\Sigma_p\times\Sigma_q}X^{\tensorcheck p}
  \tensorcheck Y^{\tensorcheck q}\ar[d]^{\ol{i}_{q,p}}\ar[r]^-{\ol{f}_{q,p}} & 
  \capO_A[\mathbf{q}\boldsymbol{+}\mathbf{r}]\tensorcheck_{\Sigma_q}Y^{\tensorcheck q}\ar[d]^{\varphi_q}\\
  \capO_A[\mathbf{p}\boldsymbol{+}\mathbf{q}\boldsymbol{+}\mathbf{r}]
  \tensorcheck_{\Sigma_{p+q}}Y^{\tensorcheck(p+q)}
  \ar[r]^-{\varphi_{p+q}} & \cdot
}
\end{align}
commute for every $p,q\geq 0$. Since $\ol{i}_{q,0}=\id$ and $\ol{f}_{q,0}=\id$, it is sufficient to consider $q\geq 0$ and $p>0$. 

The next step is to reconstruct the colimit of the left-hand column of \eqref{eq:induced_maps_modules_refined} in $\SymSeq^{\Sigma_r^\op}$ via a suitable filtered colimit in $\SymSeq^{\Sigma_r^\op}$. The diagrams \eqref{eq:cone_data_refined} suggest how to proceed. Define $\capO_A^0[\mathbf{r}]:=\capO_A[\mathbf{r}]$ and for each $t\geq 1$ define $\capO_A^t[\mathbf{r}]$ by the pushout diagram \eqref{eq:good_filtration_modules_refined} in $\SymSeq^{\Sigma_r^\op}$. The maps $f_*$ and $i_*$ are induced by the appropriate maps $\ol{f}_{q,p}$ and $\ol{i}_{q,p}$. Arguing exactly as in \cite[proof of 4.20]{Harper_Spectra} for the case  $r=0$, it is easy to use the diagrams \eqref{eq:cone_data_refined} to verify that \eqref{eq:filtered_colimit_modules_refined} is satisfied.
\end{proof}

The following proposition is the key result used to prove Proposition \ref{prop:analysis_of_OA_monoidal_model_category}.

\begin{prop}
\label{prop:homotopical_analysis_of_OA_on_maps_monoidal}
Let $\capO$ be an operad in $\CC$. Suppose that Homotopical Assumption \ref{HomotopicalAssumption} is satisfied.
\begin{itemize}
\item[(a)] If $\function{j}{A}{B}$ is a cofibration in $\AlgO$ (resp. $\LtO$) such that $\capO_A[\mathbf{r}]$ is cofibrant in $\CC^{\Sigma_r^\op}$ (resp. $\SymSeq^{\Sigma_r^\op}$) for each $r\geq 0$, then $\capO_A[\mathbf{r}]\rarrow\capO_B[\mathbf{r}]$ is a cofibration in $\CC^{\Sigma_r^\op}$ (resp. $\SymSeq^{\Sigma_r^\op}$) for each $r\geq 0$.
\item[(b)] If $\function{j}{A}{B}$ is an acyclic cofibration in $\AlgO$ (resp. $\LtO$) such that $\capO_A[\mathbf{r}]$ is cofibrant in $\CC^{\Sigma_r^\op}$ (resp. $\SymSeq^{\Sigma_r^\op}$) for each $r\geq 0$, then $\capO_A[\mathbf{r}]\rarrow\capO_B[\mathbf{r}]$ is an acyclic cofibration in $\CC^{\Sigma_r^\op}$ (resp. $\SymSeq^{\Sigma_r^\op}$) for each $r\geq 0$.
\end{itemize}
\end{prop}

\begin{proof}
It suffices to consider the case of left $\capO$-modules. We first prove part (a). Let $\function{i}{X}{Y}$ be a generating cofibration in $\SymSeq$, and consider a pushout diagram of the form \eqref{eq:gluing_on_cells_proof_of_forgetful} in $\LtO$. Assume $\capO_{Z_0}[\mathbf{r}]$ is cofibrant in $\SymSeq^{\Sigma_r^\op}$ for each $r\geq 0$; let's verify that $\capO_{Z_0}[\mathbf{r}]\rarrow\capO_{Z_1}[\mathbf{r}]$ is a cofibration in $\SymSeq^{\Sigma_r^\op}$ for each $r\geq 0$. Define $A:=Z_0$, and let $r\geq 0$. By \ref{prop:filtering_OA} we know that $\capO_{Z_1}[\mathbf{r}]$ is naturally isomorphic to a filtered colimit of the form
$
  \capO_{Z_1}[\mathbf{r}]\Iso 
  \colim\bigl(
  \xymatrix@1{
    \capO_A^0[\mathbf{r}]\ar[r]^{j_1} & 
    \capO_A^1[\mathbf{r}]\ar[r]^{j_2} & 
    \capO_A^2[\mathbf{r}]\ar[r]^{j_3} & \dotsb
  }
  \bigr)
$
in $\SymSeq^{\Sigma_r^\op}$, hence it is enough to verify each $j_t$ is a cofibration in $\SymSeq^{\Sigma_r^\op}$. By the construction of $j_t$ in Proposition \ref{prop:filtering_OA}, we need only show that each $\id\tensorcheck_{\Sigma_t}i_*$ in \eqref{eq:good_filtration_modules_refined} is a cofibration in $\SymSeq^{\Sigma_r^\op}$. Suppose $\function{p}{C}{D}$ is an acyclic fibration in $\SymSeq^{\Sigma_r^\op}$. We need to verify that $\id\tensorcheck_{\Sigma_t}i_*$ has the left lifting property with respect to $p$. Consider any such lifting problem; we want to verify that the corresponding solid commutative diagram
\begin{align*}
\xymatrix{
  \emptyset\ar[r]\ar[d] & 
  \Map^\tensorcheck(Y^{\tensorcheck t}, C)\ar[d]^{(*)}\\
  \capO_A[\mathbf{t}\boldsymbol{+}\mathbf{r}]\ar[r]\ar@{.>}[ur] & 
  \Map^\tensorcheck(Q_{t-1}^t,C)
  \times_{\Map^\tensorcheck(Q_{t-1}^t,D)}
  \Map^\tensorcheck(Y^{\tensorcheck t},D)
}
\end{align*}
in $\SymSeq^{(\Sigma_t\times\Sigma_r)^\op}$ has a lift. By assumption, $\capO_A[\mathbf{t}\boldsymbol{+}\mathbf{r}]$ is cofibrant in $\SymSeq^{\Sigma_{t+r}^\op}$, hence $\capO_A[\mathbf{t}\boldsymbol{+}\mathbf{r}]$ is cofibrant in $\SymSeq^{(\Sigma_t\times\Sigma_r)^\op}$, and it is enough to check that $(*)$ is an acyclic fibration in $\SymSeq$. We know that $i_*$ is a cofibration in $\SymSeq$ by \cite[7.19]{Harper_Modules}, hence we know that $(*)$ has the desired property by \cite[6.1]{Harper_Modules}, which finishes the argument that $\capO_{Z_0}[\mathbf{r}]\rarrow\capO_{Z_1}[\mathbf{r}]$ is a cofibration in $\SymSeq^{\Sigma_r^\op}$ for each $r\geq 0$. Consider a sequence
$
\xymatrix@1{
  Z_0\ar[r] & Z_1\ar[r] & Z_2\ar[r] & \dotsb
}
$
of pushouts of maps as in \eqref{eq:gluing_on_cells_proof_of_forgetful}. Assume $\capO_{Z_0}[\mathbf{r}]$ is cofibrant in $\SymSeq^{\Sigma_r^\op}$ for each $r\geq 0$. Define $Z_\infty:=\colim_k Z_k$, and consider the natural map $Z_0\rarrow Z_\infty$. We know from above that 
$
\xymatrix@1{
  \capO_{Z_0}[\mathbf{r}]\ar[r] & 
  \capO_{Z_1}[\mathbf{r}]\ar[r] & 
  \capO_{Z_2}[\mathbf{r}]\ar[r] & \dotsb
}
$
is a sequence of cofibrations in $\SymSeq^{\Sigma_r^\op}$, hence $\capO_{Z_0}[\mathbf{r}]\rarrow\capO_{Z_\infty}[\mathbf{r}]$ is a cofibration in $\SymSeq^{\Sigma_r^\op}$. Since every cofibration $A\rarrow B$ in $\LtO$ is a retract of a (possibly transfinite) composition of pushouts of maps as in \eqref{eq:gluing_on_cells_proof_of_forgetful}, starting with $Z_0=A$, and $\capO_A[\mathbf{r}]$ is cofibrant in $\SymSeq^{\Sigma_r^\op}$ for each $r\geq 0$, the proof of part (a) is complete. The proof of part (b) is similar.
\end{proof}

\begin{proof}[Proof of Proposition \ref{prop:analysis_of_OA_monoidal_model_category}]
This follows from Proposition \ref{prop:homotopical_analysis_of_OA_on_maps_monoidal}(a) by taking $A=\capO\circ(\emptyset)$ (resp. $A=\capO\circ\emptyset$), together with \eqref{eq:calculating_OA_for_initial_algebras_and_modules} and the assumption that $\capO[\mathbf{r}]$ is cofibrant in $\CC^{\Sigma_r^\op}$ for each $r\geq 0$.
\end{proof}

The following proposition is the key result used to prove Proposition \ref{prop:analysis_of_OA_symmetric_spectra}.

\begin{prop}
\label{prop:homotopical_analysis_of_OA_on_maps_symmetric_spectra}
Let $\capO$ be an operad in $\capR$-modules. 
\begin{itemize}
\item[(a)] If $\function{j}{A}{B}$ is a cofibration in $\AlgO$ (resp. $\LtO$) such that $\capO_A[\mathbf{r}]$ is flat stable cofibrant in $\ModR$ (resp. $\SymSeq$) for each $r\geq 0$, then $\capO_A[\mathbf{r}]\rarrow\capO_B[\mathbf{r}]$ is a positive flat stable cofibration in $\ModR$ (resp. $\SymSeq$) for each $r\geq 0$.
\item[(b)] If $\function{j}{A}{B}$ is an acyclic cofibration in $\AlgO$ (resp. $\LtO$) such that $\capO_A[\mathbf{r}]$ is flat stable cofibrant in $\ModR$ (resp. $\SymSeq$) for each $r\geq 0$, then $\capO_A[\mathbf{r}]\rarrow\capO_B[\mathbf{r}]$ is a positive flat stable acyclic cofibration in $\ModR$ (resp. $\SymSeq$) for each $r\geq 0$.
\end{itemize}
\end{prop}

\begin{proof}
It suffices to consider the case of left $\capO$-modules. Consider part (a). Let $\function{i}{X}{Y}$ be a generating cofibration in $\SymSeq$ with the positive flat stable model structure, and consider a pushout diagram of the form \eqref{eq:gluing_on_cells_proof_of_forgetful} in $\LtO$. Assume $\capO_{Z_0}[\mathbf{r}]$ is flat stable cofibrant in $\SymSeq$ for each $r\geq 0$; let's verify that $\capO_{Z_0}[\mathbf{r}]\rarrow\capO_{Z_1}[\mathbf{r}]$ is a positive flat stable cofibration in $\SymSeq$ for each $r\geq 0$. Define $A:=Z_0$, and let $r\geq 0$. By Proposition \ref{prop:filtering_OA}, $\capO_{Z_1}[\mathbf{r}]$ is naturally isomorphic to a filtered colimit of the form
$
  \capO_{Z_1}[\mathbf{r}]\Iso 
  \colim\bigl(
  \xymatrix@1{
    \capO_A^0[\mathbf{r}]\ar[r]^{j_1} & 
    \capO_A^1[\mathbf{r}]\ar[r]^{j_2} & 
    \capO_A^2[\mathbf{r}]\ar[r]^{j_3} & \dotsb
  }
  \bigr)
$
in $\SymSeq$, hence it is enough to verify each $j_t$ is a positive flat stable cofibration in $\SymSeq$. By the construction of $j_t$ in Proposition \ref{prop:filtering_OA}, we need only check that each $\id\tensorcheck_{\Sigma_t}i_*$ in \eqref{eq:good_filtration_modules_refined} is a positive flat stable cofibration in $\SymSeq$. By Proposition \ref{prop:generating_cofibration}, $i_*$ is a cofibration between cofibrant objects in $\SymSeq^{\Sigma_t}$ with the positive flat stable model structure.  It is thus enough to verify that $\id\tensorcheck_{\Sigma_t}i_*$ is a flat stable cofibration in $\SymSeq$. 

Suppose $\function{p}{C}{D}$ is a flat stable acyclic fibration in $\SymSeq$. We want to show that $\id\tensorcheck_{\Sigma_t}i_*$ has the left lifting property with respect to $p$. By assumption $\capO_A[\mathbf{t}\boldsymbol{+}\mathbf{r}]$ is flat stable cofibrant in $\SymSeq$, hence by exactly the same argument used in the proof of Theorem \ref{thm:homotopical_analysis_of_forgetful_functors}, $\id\tensorcheck_{\Sigma_t}i_*$ has the left lifting property with respect to $p$, which finishes the argument that $\capO_{Z_0}[\mathbf{r}]\rarrow\capO_{Z_1}[\mathbf{r}]$ is a positive flat stable cofibration in $\SymSeq$ for each $r\geq 0$. Consider a sequence
$
\xymatrix@1{
  Z_0\ar[r] & Z_1\ar[r] & Z_2\ar[r] & \dotsb
}
$
of pushouts of maps as in \eqref{eq:gluing_on_cells_proof_of_forgetful}, define $Z_\infty:=\colim_k Z_k$, and consider the naturally occurring map $Z_0\rarrow Z_\infty$. Assume $\capO_{Z_0}[\mathbf{r}]$ is flat stable cofibrant in $\SymSeq$ for each $r\geq 0$.  By the argument above we know that
$
\xymatrix@1{
  \capO_{Z_0}[\mathbf{r}]\ar[r] & 
  \capO_{Z_1}[\mathbf{r}]\ar[r] & 
  \capO_{Z_2}[\mathbf{r}]\ar[r] & \dotsb
}
$
is a sequence of positive flat stable cofibrations in $\SymSeq$, hence $\capO_{Z_0}[\mathbf{r}]\rarrow\capO_{Z_\infty}[\mathbf{r}]$ is a positive flat stable cofibration in $\SymSeq$. Noting that every cofibration $A\rarrow B$ in $\LtO$ is a retract of a (possibly transfinite) composition of pushouts of maps as in \eqref{eq:gluing_on_cells_proof_of_forgetful}, starting with $Z_0=A$, together with the assumption that $\capO_A[\mathbf{r}]$ is flat stable cofibrant in $\SymSeq$ for each $r\geq 0$, finishes the proof of part (a). Consider part (b). By arguing exactly as in part (a), except using generating acyclic cofibrations instead of generating cofibrations, it follows that $\capO_A[\mathbf{r}]\rarrow\capO_B[\mathbf{r}]$ is a monomorphism and a weak equivalence in $\SymSeq$; for instance, this follows from exactly the same argument used in the proof of Proposition \ref{prop:homotopical_analysis_of_certain_pushouts}. Noting by part (a) that $\capO_A[\mathbf{r}]\rarrow\capO_B[\mathbf{r}]$ is a positive flat stable cofibration in $\SymSeq$ finishes the proof.
\end{proof}

\begin{proof}[Proof of Proposition \ref{prop:analysis_of_OA_symmetric_spectra}]
This follows from Proposition \ref{prop:homotopical_analysis_of_OA_on_maps_symmetric_spectra}(a) by taking $A=\capO\circ(\emptyset)$ (resp. $A=\capO\circ\emptyset$), together with \eqref{eq:calculating_OA_for_initial_algebras_and_modules} and the assumption that $\capO[\mathbf{r}]$ is flat stable cofibrant in $\ModR$ for each $r\geq 0$.
\end{proof}

\subsection{Homotopical analysis of $\capO_A$ for cofibrant operads}

The purpose of this subsection is to prove Theorem \ref{thm:rigidification}. We will also prove Theorems \ref{thm:cofibration_property_needed_for_homology_completion}, \ref{thm:forgetful_functor_cofibrant_operad_lifting_argument}, and \ref{thm:forgetful_functor_operad_symmetric_spectra_explore_flat} (resp. Propositions \ref{prop:analysis_of_OA_monoidal_model_category_cofibrant_operad} and \ref{prop:analysis_of_OA_monoidal_model_category_cofibrant_operad_spectra}), which are analogs of Theorem \ref{thm:homotopical_analysis_of_forgetful_functors} (resp. Proposition \ref{prop:analysis_of_OA_symmetric_spectra}). These analogous results, for operads in $\capR$-modules and operads in a general class of monoidal model categories, require strong assumptions on the (maps of) operads involved, that allow us to replace arguments involving filtrations of $\capO_A$ with lifting arguments involving maps of endomorphism operads of diagrams.

In the next results, we need to work with operads satisfying good lifting properties, as specified by the definition below.  

\begin{defn}
\label{defn:model-cat-operads} Suppose that $\CC$ satisfies Homotopical Assumption \ref{HomotopicalAssumption}(i). A morphism of operads in $\CC$ is a \emph{fibration} (resp.~\emph{weak equivalence}) of operads if the underlying morphism of symmetric sequences is a fibration (resp.~weak equivalence) in the corresponding projective model stucture on $\SymSeq$.  A \emph{cofibration} of operads in $\CC$ is a morphism of operads that satisfies the left lifting property with respect to all fibrations of operads that are weak equivalences.   An operad $\capO$ in $\CC$ is \emph{cofibrant} if the unique map from the initial operad to  $\capO$ is a cofibration of operads.
\end{defn}

While we have found it convenient to use model category terminology in the definition above, none of the results in this paper require a model structure to exist on the category of operads in $\CC$, and we will not establish one in this paper. The following proposition was used in Subsection \ref{subsec:TQ_completion}.

\begin{prop}
\label{prop:functorial_factorizations_of_maps_of_operads}
Let $\function{f}{\capO}{\capO'}$ be a map of operads in $\CC$. Suppose that $\CC$ satisfies Homotopical Assumption \ref{HomotopicalAssumption}(i). Then $f$ has a functorial factorization in the category of operads as $\capO\xrightarrow{g}J\xrightarrow{h}\capO'$, a cofibration followed by a weak equivalence which is also a fibration (Definition \ref{defn:model-cat-operads}).
\end{prop}

\begin{proof}
Consider symmetric sequences in $\CC$. Since $\CC$ satisfies Homotopical Assumption \ref{HomotopicalAssumption}(i), it is easy to verify, using the corresponding adjunctions $(G_p,\Ev_p)$ in \eqref{eq:adjunctions_stable_flat}, that the diagram category $\SymSeq$ also satisfies Homotopical Assumption \ref{HomotopicalAssumption}(i). Consider the free-forgetful adjunction
$
 \xymatrix@1{
  F\colon\SymSeq\ar@<0.5ex>[r] & 
  \Operad\colon U\ar@<0.5ex>[l]
}
$
with left adjoint on top and $U$ the forgetful functor; here, $\Operad$ denotes the category of operads. It is easy to verify that the functor $F$ can be constructed by a  filtered colimit of the form
\begin{align*}
  F(A)\Iso\colim\bigl(I\rightarrow I\amalg A\rightarrow 
  I\amalg A\circ(I\amalg A)\rightarrow
  I\amalg A\circ(I\amalg A\circ(I\amalg A))\rightarrow\dotsc\bigr)
\end{align*}
in the underlying category $\SymSeq$; this useful description appears in \cite{Rezk}. Since the forgetful functor $U$ commutes with filtered colimits, it follows from \cite[Remark 2.4]{Schwede_Shipley} that the smallness conditions required in \cite[Lemma 2.3]{Schwede_Shipley} are satisfied, and the (possibly transfinite) small object argument described in the proof of \cite[Lemma 2.3]{Schwede_Shipley} finishes the proof.
\end{proof}

The following theorem is motivated by \cite[4.1.14]{Rezk}.

\begin{thm}
\label{thm:cofibration_property_needed_for_homology_completion}
Let $\function{g}{\capO}{\capO'}$ be a cofibration of operads in $\CC$. Suppose that $\capO,\capO'$ and $\CC$ satisfy Homotopical Assumption \ref{HomotopicalAssumption}.
\begin{itemize}
\item[(a)] If $\function{i}{X}{Z}$ is a cofibration in $\Alg_{\capO'}$ (resp. $\Lt_{\capO'}$), and $X$ is cofibrant in the underlying category $\CC$ (resp. $\SymSeq$), then $i$ is a cofibration in $\AlgO$ (resp. $\LtO$).
\item[(b)] If the forgetful functor $\AlgO\rarrow\CC$ (resp. $\LtO\rarrow\SymSeq$) preserves cofibrant objects, and $Y$ is a cofibrant $\capO'$-algebra (resp. cofibrant left $\capO'$-module), then $Y$ is cofibrant in $\AlgO$ (resp. $\LtO$).
\end{itemize}
\end{thm}

\begin{proof}
It suffices to consider the case of left $\capO'$-modules. Consider part (b). Let $Y$ be a cofibrant left $\capO'$-module. The map $\emptyset\rarrow Y$ in $\LtO$ factors functorially in $\LtO$ as $\emptyset\rightarrow X\xrightarrow{p}Y$ a cofibration followed by an acyclic fibration; here, $\emptyset$ denotes an initial object in $\LtO$. We first want to show there exists a left $\capO'$-module structure on $X$ such that $p$ is a map in $\Lt_{\capO'}$. Consider the solid commutative diagram
\begin{align*}
\xymatrix{
  \capO\ar[d]_{g}\ar[r] & 
  \End(X\xrightarrow{p}Y)
  \ar[d]^{(*)}\ar[r]^-{(**)}
  & \Map^\circ(X,X)\ar[d]^{(\id,p)}\\
  \capO'\ar[r]^-{m}\ar@{.>}[ur]^-{\ol{m}} & 
  \Map^\circ(Y,Y)\ar[r]^{(p,\id)} & \Map^\circ(X,Y)
}
\end{align*}
in $\SymSeq$ such that the right-hand square is a pullback diagram. It is easy to verify that the maps $(*)$ and $(**)$ are morphisms of operads. By assumption, $X$ is cofibrant in $\SymSeq$, hence we know that $(\id,p)$ is an acyclic fibration by \cite[6.2]{Harper_Modules}, and therefore $(*)$ is an acyclic fibration in $\SymSeq$. Since $g$ is a cofibration of operads, there exists a morphism of operads $\ol{m}$ that makes the diagram commute. It follows that the composition 
$
  \capO'\xrightarrow{\ol{m}}\End(X\xrightarrow{p}Y)
  \xrightarrow{(**)}\Map^\circ(X,X)
$
of operad maps determines a left $\capO'$-module structure on $X$ such that $p$ is a morphism of left $\capO'$-modules. To finish the proof, we need to show that $Y$ is cofibrant in $\LtO$. Consider the solid commutative diagram
\begin{align*}
\xymatrix{
  \emptyset\ar[d]\ar[r]& X\ar[d]^{p}\\
  Y\ar@{.>}[ur]^-{\xi}\ar@{=}[r] & Y
}
\end{align*}
in $\Lt_{\capO'}$, where $\emptyset$ denotes an initial object in $\Lt_{\capO'}$. Since $Y$ is cofibrant in $\Lt_{\capO'}$, and $p$ is an acyclic fibration, this diagram has a lift $\xi$ in $\Lt_{\capO'}$. In particular, $Y$ is a retract of $X$ in $\Lt_{\capO'}$, and hence in $\LtO$. Noting that $X$ is cofibrant in $\LtO$ finishes the proof of part (b). Part (a) can be established exactly as in the proof of Theorem \ref{thm:forgetful_functor_cofibrant_operad_lifting_argument}(a), by replacing the map $I\rarrow\capO$ with the map $\capO\rarrow\capO'$.
\end{proof}

\begin{proof}[Proof of Theorem \ref{thm:rigidification}]
It suffices to consider the case of left $\capO$-modules. Since $X$ is cofibrant in $\LtO$ and $g_*$ is a left Quillen functor, $g_*(X)$ is cofibrant in $\Lt_{J_1}$ and hence by \ref{thm:comparing_homotopy_categories} and \ref{prop:unit_map_is_weak_equivalence} it follows that $g^*g_*(X)\wequiv\TQ(X)$. To iterate the argument, it suffices to verify that the right Quillen functor $g^*$ preserves cofibrant objects: this follows from Theorem \ref{thm:cofibration_property_needed_for_homology_completion} and Theorem \ref{thm:homotopical_analysis_of_forgetful_functors}.
\end{proof}

The following theorem is closely related to \cite[4.1.15]{Rezk}.

\begin{thm}
\label{thm:forgetful_functor_cofibrant_operad_lifting_argument}
Let $\capO$ be a cofibrant operad in $\CC$. Suppose that Homotopical Assumption \ref{HomotopicalAssumption} is satisfied. 
\begin{itemize}
\item[(a)] If $\function{i}{X}{Z}$ is a cofibration in $\AlgO$ (resp. $\LtO$), and $X$ is cofibrant in the underlying category $\CC$ (resp. $\SymSeq$), then $i$ is a cofibration in the underlying category $\CC$ (resp. $\SymSeq$).
\item[(b)] If $Y$ is a cofibrant $\capO$-algebra (resp. cofibrant left $\capO$-module), then $Y$ is cofibrant in the underlying category $\CC$ (resp. $\SymSeq$).
\item[(c)] If the unit $S$ is cofibrant in $\CC$, then $\capO[\mathbf{r}]$ is cofibrant in $\CC^{\Sigma_r^\op}$ for each $r\geq 0$.
\end{itemize}
\end{thm}

\begin{proof} The proof of this result is very similar to that of the previous theorem. It suffices to consider the case of left $\capO$-modules. Consider part (a). Let $\function{i}{X}{Z}$ be a cofibration in $\LtO$. The map $i$ factors functorially in the underlying category $\SymSeq$ as $X\xrightarrow{j}Y\xrightarrow{p}Z$, a cofibration followed by an acyclic fibration. We want first to show there exists a left $\capO$-module structure on $Y$ such that $j$ and $p$ are maps in $\LtO$. Consider the solid commutative diagram
\begin{align*}
\xymatrix{
  I\ar[d]\ar[r] & 
  \End(X\xrightarrow{j}Y\xrightarrow{p}Z)\ar[d]^{(*)}\ar[r]^{(**)} & 
  \Map^\circ(Y,Y)\ar[d]^{(j,p)}\\
  \capO\ar[r]^-{m}\ar@{.>}[ur]^(0.4){\ol{m}} & \End(X\xrightarrow{i}Z)\ar[r] &
  \Map^\circ(X,Y)\times_{\Map^\circ(X,Z)}\Map^\circ(Y,Z) 
}
\end{align*}
in $\SymSeq$ such that the right-hand square is a pullback diagram. It is easy to verify that the maps $(*)$ and $(**)$ are morphisms of operads. By assumption, $X$ is cofibrant in $\SymSeq$, hence we know that the pullback corner map $(j,p)$ is an acyclic fibration by \cite[6.2]{Harper_Modules}, and therefore $(*)$ is an acyclic fibration in $\SymSeq$. Since $\capO$ is a cofibrant operad,  the map $I\rarrow\capO$ is a cofibration of operads, and there exists a morphism of operads $\ol{m}$ that makes the diagram commute. It follows that the composition
$
  \capO\xrightarrow{\ol{m}}\End(X\xrightarrow{j}Y\xrightarrow{p}Z)
  \xrightarrow{(**)}\Map^\circ(Y,Y)
$
of operad maps determines a left $\capO$-module structure on $Y$ such that $j$ and $p$ are morphisms of left $\capO$-modules. To finish the proof, we need to show that $i$ is a cofibration in $\SymSeq$. Consider the solid commutative diagram
\begin{align*}
\xymatrix{
  X\ar[d]_{i}\ar[r]^{j} & Y\ar[d]^{p}\\
  Z\ar@{=}[r]\ar@{.>}[ur]^{\xi} & Z
}
\end{align*}
in $\LtO$. Since $i$ is a cofibration and $p$ is an acyclic fibration in $\LtO$, the diagram has a lift $\xi$ in $\LtO$. In particular, $i$ is a retract of $j$ in $\LtO$, and hence in the underlying category $\SymSeq$. Noting that $j$ is a cofibration in $\SymSeq$ finishes the proof of part (a). Part (b) follows immediately from \cite[proof of 10.2]{Harper_Bar}, which uses a similar argument; it is also a special case of Theorem \ref{thm:cofibration_property_needed_for_homology_completion}(b). Consider part (c). By assumption, the unit $S$ is cofibrant in $\CC$, hence the map $\emptyset\rarrow I$ is a cofibration in $\SymSeq$ and therefore $\capO\circ\emptyset\rarrow\capO\circ I$ is a cofibration in $\LtO$. Hence $\capO\Iso\capO\circ I$ is a cofibrant left $\capO$-module, and part (b) finishes the proof.
\end{proof}

\begin{thm}
\label{thm:forgetful_functor_operad_symmetric_spectra_explore_flat}
Let $\capO$ be a cofibrant operad in $\capR$-modules with respect to the positive flat stable model structure. 
\begin{itemize}
\item[(a)] $\capO[\mathbf{r}]$ is flat stable cofibrant in $\ModR^{\Sigma_r^\op}$ for each $r\geq 0$.
\item[(b)] If $\function{i}{X}{Z}$ is a cofibration in $\AlgO$ (resp. $\LtO$), and $X$ is flat stable cofibrant in the underlying category $\ModR$ (resp. $\SymSeq$), then $i$ is a flat stable cofibration in the underlying category $\ModR$ (resp. $\SymSeq$).
\end{itemize}
\end{thm}

\begin{proof}
Since every flat stable fibration in $\SymSeq$ is a positive flat stable fibration in $\SymSeq$, it follows that $\capO$ is also a cofibrant operad in $\capR$-modules with respect to the flat stable model structure. The proof of Theorem \ref{thm:forgetful_functor_cofibrant_operad_lifting_argument} finishes the argument.
\end{proof}

\begin{prop}
\label{prop:pushout_description_of_OA}
Let $\capO$ be an operad in $\CC$ and $A\in\AlgO$ (resp. $A\in\LtO$). Consider the pushout diagram in $\LtO$ (resp. $\Lt_{\tilde{\capO}}$) of the form
\begin{align}
\label{eq:pushout_description_of_OA}
\xymatrix{
  \capO\circ \emptyset\ar[r]\ar[d] & 
  \hat{A}\ar[d]^{j}\\
  \capO\circ I\ar[r] & \hat{A}\amalg(\capO\circ I)
}\quad\quad\text{resp.}\quad
\xymatrix{
  \tilde{\capO}\circtilde \emptyset\ar[r]\ar[d] & 
  \tilde{A}\ar[d]^{j}\\
  \tilde{\capO}\circtilde \hat{I}\ar[r] & 
  \tilde{A}\amalg(\tilde{\capO}\circtilde\hat{I})
}
\end{align} 
There are natural isomorphisms
$
  \capO_A[\mathbf{t}]\Iso
  \bigl(\hat{A}\amalg(\capO\circ I)\bigr)[\mathbf{t}]
$ (resp. 
$
  \capO_A[\mathbf{t}][\mathbf{r}]\Iso
  \bigl(\tilde{A}\amalg(\tilde{\capO}\circtilde \hat{I})
  \bigr)[\mathbf{r}][\mathbf{t}]
$)
for each $r,t\geq 0$. Here, $\hat{I}$ is the symmetric array concentrated at $0$ with value $I$.
\end{prop}

\begin{proof}
This follows from Propositions \ref{prop:coproduct_modules}, \ref{prop:relating_the_OA_constructions}, \ref{prop:tilde_commutes_with_OA_constructions}, and \ref{prop:symmetric_sequence_decompositions}.
\end{proof}

\begin{prop}
\label{prop:cofibrant_operad_symmetric_array_underlying}
Let $\capO$ be a cofibrant operad in $\CC$. Suppose that $\capO,\tilde{\capO}$ and $\CC$ satisfy Homotopical Assumption \ref{HomotopicalAssumption}. If $\function{i}{X}{Z}$ is a cofibration in $\Lt_{\tilde{\capO}}$ such that $X$ is cofibrant in the underlying category $\SymArray$, then $i$ is a cofibration in the underlying category $\SymArray$.
\end{prop}

\begin{proof}
This proof is similar to that of Theorem \ref{thm:forgetful_functor_cofibrant_operad_lifting_argument}, except for the following variation on the lifting argument. Let $\function{i}{X}{Z}$ be a cofibration in $\Lt_{\tilde{\capO}}$. The map $i$ factors functorially in the underlying category $\SymArray$ as
$X\xrightarrow{j}Y\xrightarrow{p}Z$, a cofibration followed by an acyclic fibration. We need to show there exists a left $\tilde{\capO}$-module structure on $Y$ such that $j$ and $p$ are maps in $\Lt_{\tilde{\capO}}$. Consider the solid diagram
\begin{align*}
\xymatrix{
  & \End(X\xrightarrow{j}Y\xrightarrow{p}Z)\ar[d]^{(*)}\ar[r]^{(**)} & 
  \Map^\circtilde(Y,Y)\ar[d]^{(j,p)}\\
  \tilde{\capO}\ar[r]^-{m}\ar@{.>}[ur]^(0.4){\ol{m}} & \End(X\xrightarrow{i}Z)\ar[r] &
  \Map^\circtilde(X,Y)\times_{\Map^\circtilde(X,Z)}\Map^\circtilde(Y,Z) 
}
\end{align*}
in $\SymArray$, such that the square is a pullback diagram. It is easy to verify that the maps $(*)$ and $(**)$ are morphisms of operads. Since $X$ is cofibrant in $\SymArray$,  the pullback corner map $(j,p)$ is an acyclic fibration in  $\SymArray$ by \cite[6.2]{Harper_Modules}, and therefore $(*)$ is as well. We need to show there exists a map of operads $\ol{m}$ that makes the diagram commute. By the right-hand adjunction in \eqref{eq:evaluate_at_zero_adjunction}, it is enough to show there exists a map $\ol{m}$ of operads in $\CC$ that makes the corresponding diagram
\begin{align*}
\xymatrix{
  & \Ev_0\bigl(\End(X\xrightarrow{j}Y\xrightarrow{p}Z)\bigr)
  \ar[d]^{\Ev_0(*)} \\
  \capO\ar[r]^-{m}\ar@{.>}[ur]^(0.4){\ol{m}} & 
  \Ev_0\bigl(\End(X\xrightarrow{i}Z)\bigr)
}
\end{align*}
of operads in $\CC$ commute. Since $\capO$ is a cofibrant operad in $\CC$,  the desired lift $\ol{m}$ exists. It follows that the composition $(**)\ol{m}$ of operad maps determines a left $\tilde{\capO}$-module structure on $Y$ such that $j$ and $p$ are morphisms of left $\tilde{\capO}$-modules. To finish the proof, we need to show that $i$ is a cofibration in $\SymArray$, which follows exactly as in the proof of Theorem \ref{thm:forgetful_functor_cofibrant_operad_lifting_argument}.
\end{proof}

\begin{prop}
\label{prop:analysis_of_OA_monoidal_model_category_cofibrant_operad}
Let $\capO$ be a cofibrant operad in $\CC$. Suppose that $\capO,\tilde{\capO}$ and $\CC$ satisfy Homotopical Assumption \ref{HomotopicalAssumption}. If the unit $S$ is cofibrant in $\CC$, and $A$ is an $\capO$-algebra (resp. left $\capO$-module) that is cofibrant in the underlying category $\CC$ (resp. $\SymSeq$), then $\capO_A[\mathbf{r}]$ is cofibrant in $\CC^{\Sigma_r^\op}$ (resp. $\SymSeq^{\Sigma_r^\op}$) for each $r\geq 0$.
\end{prop}

\begin{proof}
This follows from
Proposition \ref{prop:pushout_description_of_OA}, Theorem \ref{thm:forgetful_functor_cofibrant_operad_lifting_argument}, and Proposition \ref{prop:cofibrant_operad_symmetric_array_underlying}.
\end{proof}

\begin{prop}
\label{prop:analysis_of_OA_monoidal_model_category_cofibrant_operad_spectra}
Let $\capO$ be a cofibrant operad in $\capR$-modules with respect to the positive flat stable model structure. If $A$ is an $\capO$-algebra (resp. left $\capO$-module) that is flat stable cofibrant in $\ModR$ (resp. $\SymSeq$), then $\capO_A[\mathbf{r}]$ is flat stable cofibrant in $\ModR^{\Sigma_r^\op}$ (resp. $\SymSeq^{\Sigma_r^\op}$) for each $r\geq 0$.
\end{prop}

\begin{proof}
Since every flat stable fibration in $\SymSeq$ is a positive flat stable fibration in $\SymSeq$, it follows that $\capO$ is also a cofibrant operad in $\capR$-modules with respect to the flat stable model structure. The proof of Proposition \ref{prop:analysis_of_OA_monoidal_model_category_cofibrant_operad} finishes the argument.
\end{proof}

\subsection{Proofs}

The purpose of this short subsection is to prove Propositions \ref{prop:replacement_of_operads}, \ref{prop:connectivity}, \ref{prop:commuting_hocolim_of_simplicial_objects}, \ref{prop:connectivity_of_smash_powers_of_maps}, and \ref{prop:finiteness_derived_tensor_uses_bar_construction}.

\begin{proof}[Proof of Proposition \ref{prop:replacement_of_operads}]
This follows from a small object argument together with an analysis of the functor $F$ appearing in 
$
 \xymatrix@1{
  F\colon\SymSeq\ar@<0.5ex>[r] & 
  \Operad\colon U\ar@<0.5ex>[l]
}
$
the free-forgetful adjunction with left adjoint on top and $U$ the forgetful functor; here, $\Operad$ denotes the category of operads. It is easy to verify that the functor $F$ can be constructed by a  filtered colimit of the form
\begin{align*}
  F(A)\Iso\colim\bigl(I\rightarrow I\amalg A\rightarrow 
  I\amalg A\circ(I\amalg A)\rightarrow
  I\amalg A\circ(I\amalg A\circ(I\amalg A))\rightarrow\dotsc\bigr)
\end{align*}
in the underlying category $\SymSeq$; this useful description appears in \cite{Rezk}. Using this description of $F$, it is easy to verify that the unit map $I\rarrow\capO'$ of the operad $\capO'$ constructed in the small object argument satisfies the desired property in Cofibrancy Condition \ref{CofibrancyCondition}. 
\end{proof}

\begin{proof}[Proof of Proposition \ref{prop:connectivity}]
For a recent reference of part (a) in the context of symmetric spectra, see \cite{Schwede_book_project}. Consider part (b). It is enough to treat the special case where $X,Y$ are furthermore fibrant and cofibrant in the category of $\capR$-modules with the flat stable model structure. Let $\capR'\rarrow \capR$ be a cofibrant replacement in the category of monoids in $(\Spectra,\tensor_S,S)$ with the flat stable model structure \cite{Harper_Modules, Schwede_Shipley}. Since the sphere spectrum $S$ is flat stable cofibrant in $\Spectra$, we know by Theorem \ref{thm:homotopical_analysis_of_forgetful_functors_monoidal}(a)  that $\capR'$ is flat stable cofibrant in the underlying category $\Spectra$, and it follows from \cite{Harper_Spectra,Harper_Bar} by arguing as in the proof of Theorem \ref{thm:comparing_homotopy_completion_towers} that there are natural weak equivalences 
$
  X\Smash^\LL Y=
  X(\tensor_S)^\LL_\capR Y\wequiv
  X'(\tensor_S)^\LL_{\capR'} Y'\wequiv
  |\BAR^{\tensor_S}(X',\capR',Y')|=|B|
$.
Here, $X'\rarrow X$ and $Y'\rarrow Y$ are functorial flat stable cofibrant replacements in the category of right (resp. left) $\capR'$-modules. Denote by $B$ the indicated simplicial bar construction with respect to $\tensor_S$. We need to verify that $|B|$ is $(m+n+1)$-connected. We know by Theorem \ref{thm:homotopical_analysis_of_forgetful_functors_monoidal}(b) that $X',Y'$ are flat stable cofibrant in the underlying category $\Spectra$, hence it follows from part (a) that $B$ is objectwise $(m+n+1)$-connected and Proposition \ref{prop:connectivity_of_simplicial_maps_spectra} finishes the proof for part (b). Part (c) is verified exactly as in the proof of part (b) above, except using the group algebra spectrum $\capR[\Sigma_t]$ instead of $\capR$. Part (d) follows easily from part (b) together with \eqref{eq:tensor_check_calc}. Part (e) follows easily from parts (d) and (c) together with \eqref{eq:tensor_check_calc}.
\end{proof}

\begin{proof}[Proof of Proposition \ref{prop:commuting_hocolim_of_simplicial_objects}]
It suffices to consider the case of simplicial left $\tau_1\capO$-modules. Consider the map $\emptyset\rarrow B$ in $\sLt_{\tau_1\capO}$, and use functorial factorization in $\sLt_{\tau_1\capO}$ \cite[3.6]{Harper_Bar} to obtain $\emptyset\rarrow B^c\rarrow B$, a cofibration followed by an acyclic fibration. By Proposition \ref{prop:retract_property_and_derived_circle} and \cite[5.6]{Harper_Bar}, there is a retract of the form
\begin{align*}
\xymatrix{
  |i_k\capO\circ_{\tau_1\capO}B^c|\ar[d]^{(*)}\ar[r] &
  |\capO\circ_{\tau_1\capO}B^c|\ar[d]^{(**)}\ar[r] &
  |i_k\capO\circ_{\tau_1\capO}B^c|\ar[d]^{(*)}\\
  i_k\capO\circ_{\tau_1\capO}\colim\limits^{\Lt_{\tau_1\capO}}_{\Delta^\op}B^c\ar[r] &
  \capO\circ_{\tau_1\capO}\colim\limits^{\Lt_{\tau_1\capO}}_{\Delta^\op}B^c\ar[r] &
  i_k\capO\circ_{\tau_1\capO}\colim\limits^{\Lt_{\tau_1\capO}}_{\Delta^\op}B^c
}
\end{align*}
in $\SymSeq$. Since $B^c$ is cofibrant in $\sLt_{\tau_1\capO}$,  the proof of \cite[3.15]{Harper_Bar} implies that $\capO\circ_{\tau_1\capO}B^c$ is cofibrant in $\sLtO$. It follows therefore from \cite[5.24]{Harper_Bar} that $(**)$ is a weak equivalence, hence $(*)$ is also a weak equivalence. We know from \cite[3.12]{Harper_Bar} that $B^c$ is objectwise cofibrant in $\Lt_{\tau_1\capO}$, hence there are natural weak equivalences
$
  i_k\capO\circ_{\tau_1\capO}B^c\wequiv
  i_k\capO\circ^\HH_{\tau_1\capO}B^c\wequiv
  i_k\capO\circ^\HH_{\tau_1\capO}B.
$
It follows that there are natural weak equivalences
\begin{align*}
  i_k\capO\circ^\HH_{\tau_1\capO}\hocolim\limits^{\Lt_{\tau_1\capO}}_{\Delta^\op}B\wequiv
  i_k\capO\circ^\HH_{\tau_1\capO}\hocolim\limits^{\Lt_{\tau_1\capO}}_{\Delta^\op}B^c\wequiv
  i_k\capO\circ^\HH_{\tau_1\capO}\colim\limits^{\Lt_{\tau_1\capO}}_{\Delta^\op}B^c\\
  \wequiv
  i_k\capO\circ_{\tau_1\capO}\colim\limits^{\Lt_{\tau_1\capO}}_{\Delta^\op}B^c
  \wequiv
  |i_k\capO\circ_{\tau_1\capO}B^c|
  \wequiv
  \hocolim\limits^{\Lt_{\capO}}_{\Delta^\op}
  i_k\capO\circ_{\tau_1\capO}B^c \\
  \wequiv
  \hocolim\limits^{\Lt_{\capO}}_{\Delta^\op}
  i_k\capO\circ^\HH_{\tau_1\capO}B^c
  \wequiv
  \hocolim\limits^{\Lt_{\capO}}_{\Delta^\op}
  i_k\capO\circ^\HH_{\tau_1\capO}B
\end{align*}
which finishes the proof; here we have used \ref{main_hocolim_theorem}.
\end{proof}

\begin{proof}[Proof of Proposition \ref{prop:connectivity_of_smash_powers_of_maps}]
Consider part (a) and the case of $\capR$-modules. The map $f$ factors functorially in $\ModR$ with the flat stable model structure as $X\xrightarrow{g}Y\xrightarrow{h}Z$ a cofibration followed by an acyclic fibration, and hence the map $f^{\wedge t}$ factors as
$
  X^{\wedge t}\xrightarrow{g^{\wedge t}}Y^{\wedge t}
  \xrightarrow{h^{\wedge t}}Z^{\wedge t}
$. Since smashing with a flat stable cofibrant $\capR$-module preserves weak equivalences, $h^{\wedge t}$ is a weak equivalence, and hence it is enough to check that $g^{\wedge t}$ is $n$-connected. We argue by induction on $t$. Using the pushout diagrams in Definition \ref{def:filtration_setup_modules} (see, for instance, \cite[4.15]{Harper_Spectra}) together with the natural isomorphisms $Y^{\wedge t}/Q^t_{t-1}\Iso (Y/X)^{\wedge t}$, it follows that each of the maps
$
  X^{\wedge t}\rarrow Q^t_{1}\rarrow Q^t_2\rarrow\dotsb\rarrow
  Q^t_{t-1}\rarrow Y^{\wedge t}
$
is at least $n$-connected, which finishes the proof for the case of $\capR$-modules. The case of symmetric sequences is similar. Consider part (b). This follows by proceeding as in the proof of part (a), except using the positive flat stable model structure, together with part (a) and Propositions \ref{prop:cofibration}, \ref{prop:good_properties}, \ref{prop:cofibrations_to_mono}, and \ref{prop:connectivity}.
\end{proof}

Propositions \ref{prop:homotopy_spectral_sequence_chain_complexes}, \ref{prop:realzn_preserves_finiteness_properties_chain_complexes}, and \ref{prop:eilenberg_moore_chain_complexes} will be needed for the proof of Proposition \ref{prop:finiteness_derived_tensor_uses_bar_construction} below. The following homotopy spectral sequence for a simplicial unbounded chain complex is well known; for a recent reference, see \cite[5.6]{Weibel}.

\begin{prop}
\label{prop:homotopy_spectral_sequence_chain_complexes}
Let $Y$ be a simplicial unbounded chain complex over $\unit$. There is a natural homologically graded spectral sequence in the right-half plane such that
\begin{align*}
  E^2_{p,q} = H_p(H_q(Y))\Longrightarrow H_{p+q}(|Y|)
\end{align*}
Here, $H_q(Y)$ denotes the simplicial $\unit$-module obtained by applying $H_q$ levelwise to $Y$, and $\unit$ is any commutative ring.
\end{prop}

\begin{prop}
\label{prop:realzn_preserves_finiteness_properties_chain_complexes}
Let $Y$ be a simplicial unbounded chain complex over $\ZZ$. Let $m\in\ZZ$. Assume that $Y$ is levelwise $m$-connected.
\begin{itemize}
\item[(a)] If $H_k Y_n$ is finite for every $k,n$, then $H_k |Y|$ is finite for every $k$.
\item[(b)] If $H_k Y_n$ is a finitely generated abelian group for every $k,n$, then $H_k|Y|$ is a finitely generated abelian group for every $k$.
\end{itemize}
\end{prop}

\begin{proof}
This follows from Proposition \ref{prop:homotopy_spectral_sequence_chain_complexes}.
\end{proof}

Recall the following Eilenberg-Moore type spectral sequences; for a recent reference, see \cite[5.7]{Weibel}.

\begin{prop}
\label{prop:eilenberg_moore_chain_complexes}
Let $t\geq 1$. Let  $A,B$ be unbounded chain complexes over $\unit$ with a right (resp. left) $\Sigma_t$-action.  There is a natural homologically graded spectral sequence in the right-half plane such that
\begin{align*}
  E^2_{p,q} &= \Tor^{\unit[\Sigma_t]}_{p,q}(H_*A,H_*B)
  \Longrightarrow H_{p+q}(A\tensor^\LL_{\Sigma_t} B).
\end{align*}
Here, $\unit$ is any commutative ring, $(\Chaincx_\unit,\tensor,\unit)$ denotes the closed symmetric monoidal category of unbounded chain complexes over $\unit$, $\unit[\Sigma_t]$ is the group algebra, and $\tensor^\LL_{\Sigma_t}$ is the total left derived functor of $\tensor_{\Sigma_t}$.
\end{prop}

\begin{proof}[Proof of Proposition \ref{prop:finiteness_derived_tensor_uses_bar_construction}]
Consider part (a). It is enough to treat the special case where $M,N$ are furthermore cofibrant in the category of right (resp. left) $\capA$-modules. Let $\capA'\rarrow \capA$ be a cofibrant replacement in the category of monoids in $(\Chaincx_\ZZ,\tensor,\ZZ)$ with the model structure of \cite{Schwede_Shipley}. Since $\ZZ$ is cofibrant in $\Chaincx_\ZZ$, we know by Theorem \ref{thm:homotopical_analysis_of_forgetful_functors_monoidal}(a) that $\capA'$ is cofibrant in the underlying category $\Chaincx_\ZZ$, and it follows easily by arguing as in the proof of Theorem \ref{thm:comparing_homotopy_completion_towers} that there are natural weak equivalences 
$
  M\tensor^\LL_\capA N\wequiv M'\tensor^\LL_{\capA'}N'
  \wequiv|\BAR^\tensor(M',\capA',N')|=|B|.
$
Here, $M'\rarrow M$ and $N'\rarrow N$ are functorial cofibrant replacements in the category of right (resp. left) $\capA'$-modules. Denote by $B$ the indicated simplicial bar construction with respect to $\tensor$. We need  to verify that $H_k(|B|)$ is finite for every $k$. We know by Theorem \ref{thm:homotopical_analysis_of_forgetful_functors_monoidal}(b) that $M',N'$ are cofibrant in the underlying category $\Chaincx_\ZZ$, hence it follows from Proposition \ref{prop:eilenberg_moore_chain_complexes} (with $t=1$) that $H_k(B_n)$ is finite for every $k$ and $n$, and Proposition \ref{prop:realzn_preserves_finiteness_properties_chain_complexes} finishes the proof for part (a). Part (b) is similar.
\end{proof}

\section{Homotopical analysis of the simplicial bar constructions}
\label{sec:homotopical_analysis_bar_constructions}

The purpose of this section is to prove Theorem \ref{thm:reedy_cofibrant_for_bar_constructions} together with several closely related technical results on simplicial structures and the simplicial bar constructions. The results established here lie at the heart of the proofs of the main theorems in this paper.

\subsection{Simplicial structure on $\AlgO$ and $\LtO$}

The purpose of this subsection is to describe the simplicial structure on $\AlgO$ (resp. $\LtO$) and to prove several related results. The key technical results of this subsection are Proposition \ref{prop:realizations_are_isomorphic} and Theorem \ref{thm:simplicial_model_category_structure}. They are used in the proof of Theorem \ref{thm:reedy_cofibrant_for_bar_constructions} to construct skeletal filtrations in $\Alg_{\capO'}$ (resp. $\Lt_{\capO'}$) of realizations (Definition \ref{defn:realization}) of the simplicial bar constructions (Proposition \ref{prop:natural_map_is_weak_equivalence}).

Consider symmetric sequences in $\capR$-modules, and let $\capO\in\SymSeq$, $X$ in $\ModR$ (resp. $\SymSeq$), and $K\in\sSet$. Define $\nu$ to be the natural map
\begin{align*}
  \capO\circ(X)\Smash K_+\xrightarrow{\ \nu\ }\capO\circ(X\Smash K_+)
  \quad\quad
  \Bigl(\text{resp.}\quad
  (\capO\circ X)\Smash K_+\xrightarrow{\ \nu\ }\capO\circ(X\Smash K_+)
  \Bigr)
\end{align*}
in $\ModR$ (resp. $\SymSeq$) induced by the natural maps $K\rarrow K^{\times t}$ in $\sSet$ for $t\geq 0$; these are the diagonal maps for $t\geq 1$ and the constant map for $t=0$. Here, $\sSet$ denotes the category of simplicial sets. The construction of the tensor product below is motivated by \cite[VII.2.10]{EKMM}. Simplicial structures in the context of symmetric spectra have also been exploited in \cite{Hornbostel, Schwede_book_project}; see also \cite{Arone_Ching, McClure_Schwanzl_Vogt}.

\begin{defn}
\label{defn:tensor_with_simplicial_sets}
Let $\capO$ be an operad in $\capR$-modules, $X$ an $\capO$-algebra (resp. left $\capO$-module), and $K$ a simplicial set. Define the \emph{tensor product} $X\tensordot K$ in $\AlgO$ (resp. $\LtO$) by the reflexive coequalizer
\begin{align}
  \label{eq:tensor_with_simplicial_sets_algebras}
  X\tensordot K &:=\colim
  \Bigl(
  \xymatrix{
  \capO\circ(X\Smash K_+) & 
  \capO\circ\bigl(\capO\circ(X)\Smash K_+\bigr)
  \ar@<0.5ex>[l]^-{d_1}\ar@<-0.5ex>[l]_-{d_0}
  }
  \Bigr)\\
  \label{eq:tensor_with_simplicial_sets_modules}
  \Bigl(\text{resp.}\quad
  X\tensordot K &:=\colim
  \Bigl(
  \xymatrix{
  \capO\circ(X\Smash K_+) & 
  \capO\circ\bigl((\capO\circ X)\Smash K_+\bigr)
  \ar@<0.5ex>[l]^-{d_1}\ar@<-0.5ex>[l]_-{d_0}
  }
  \Bigr)
  \Bigr)
\end{align}
in $\AlgO$ (resp. $\LtO$), with $d_0$ induced by operad multiplication $\function{m}{\capO\circ\capO}{\capO}$ and the map $\nu$, while $d_1$ is induced by the left $\capO$-action map $\function{m}{\capO\circ(X)}{X}$ (resp. $\function{m}{\capO\circ X}{X}$).
\end{defn}

Let $\capO$ be an operad in $\capR$-modules, consider $X,Y$ in $\ModR$ (resp. $\SymSeq$), $K\in\sSet$, and recall the isomorphisms
\begin{align}
  \label{eq:adjunction_smash_with_simplicial_set_spectra}
  \hom_{\ModR}(X\Smash K_+,Y)&\Iso\hom_{\ModR}(X,\Map(K_+,Y))\\
  \label{eq:adjunction_smash_with_simplicial_set_symmetric_sequences}
  \Bigl(\text{resp.}\quad
  \hom_{\SymSeq}(X\Smash K_+,Y)&\Iso\hom_{\SymSeq}(X,\Map(K_+,Y))
  \Bigr)
\end{align}
natural in $X,K,Y$. Here, we are using the useful shorthand notation $\Map(K_+,-)$ to denote $\Map(\capR\tensor G_0K_+,-)$; see, just above \ref{defn:realization}. If $Y$ is an $\capO$-algebra (resp. left $\capO$-module), then $\Map(K_+,Y)$ has an $\capO$-algebra (resp. left $\capO$-module) structure induced by $\function{m}{\capO\circ(Y)}{Y}$ (resp. $\function{m}{\capO\circ Y}{Y}$). The next proposition is a formal argument left to the reader. We will use it below in several proofs.

\begin{prop}
\label{prop:useful_commutative_diagram_involving_nu_map}
Let $\capO$ be an operad in $\capR$-modules. Let $X\in\SymSeq$, $Y\in\LtO$, and $K\in\sSet$. If $\function{f}{X\Smash K_+}{Y}$ is a map in $\SymSeq$, then the diagram
\begin{align*}
\xymatrix{
  (\capO\circ X)\Smash K_+\ar[d]_{\nu}\ar[r]^-{\id\circ f\Smash\id} & 
  \bigl(\capO\circ\Map(K_+,Y)\bigr)\Smash K_+\ar[r]^-{\nu} &
  \capO\circ\bigl(\Map(K_+,Y)\Smash K_+\bigr)\ar[d]^{\id\circ\ev} \\
  \capO\circ(X\Smash K_+)\ar[rr]^-{\id\circ f} && \capO\circ Y
}
\end{align*}
in $\SymSeq$ commutes. Here, $\ev$ denotes the evaluation map, and we have used the same notation for both $f$ and its adjoint \eqref{eq:adjunction_smash_with_simplicial_set_symmetric_sequences}.
\end{prop}

The following proposition will be useful.

\begin{prop}
\label{prop:natural_isomorphisms_tensordot_and_mapping_object_algebras}
Let $\capO$ be an operad in $\capR$-modules. Let $X,Y$ be $\capO$-algebras (resp. left $\capO$-modules) and $K$ a simplicial set. There are isomorphisms
\begin{align*}
  \hom_{\AlgO}(X\tensordot K,Y)&\Iso\hom_{\AlgO}(X,\Map(K_+,Y))\\
  \Bigl(\text{resp.}\quad
  \hom_{\LtO}(X\tensordot K,Y)&\Iso\hom_{\LtO}(X,\Map(K_+,Y))
  \Bigr)
\end{align*}
natural in $X,K,Y$.
\end{prop}

\begin{proof}
It suffices to consider the case of left $\capO$-modules. We need to verify that specifying a map $X\tensordot K\rarrow Y$ in $\LtO$ is the same as specifying a map $X\rarrow\Map(K_+,Y)$ in $\LtO$, and that the resulting correspondence is natural. Suppose $\function{f}{X\tensordot K}{Y}$ is a map of left $\capO$-modules, and consider the corresponding commutative diagram
\begin{align}
\label{eq:induced_map_on_reflexive_coequalizers_tensor_with_simplicial_set}
\xymatrix{
  X\tensordot K\ar[d]_{f} &
  \capO\circ(X\Smash K_+)\ar[d]\ar[l]\ar[dl]_{f} & 
  \capO\circ\bigl((\capO\circ X)\Smash K_+\bigr)
  \ar[d]\ar@<0.5ex>[l]^-{d_1}\ar@<-0.5ex>[l]_-{d_0}\\
  Y & \capO\circ Y\ar[l]_-{m} & 
  \capO\circ\capO\circ Y\ar@<0.5ex>[l]^-{\id\circ m}
  \ar@<-0.5ex>[l]_-{m\circ\id}
}
\end{align}
in $\LtO$ with rows reflexive coequalizer diagrams. Using the same notation for both $\function{f}{\capO\circ(X\Smash K_+)}{Y}$ in $\LtO$ and its adjoints $\function{f}{X\Smash K_+}{Y}$ in $\SymSeq$ \eqref{eq:free_forgetful_adjunction} and $\function{f}{X}{\Map(K_+,Y)}$ in $\SymSeq$ \eqref{eq:adjunction_smash_with_simplicial_set_symmetric_sequences}, it follows easily from \eqref{eq:induced_map_on_reflexive_coequalizers_tensor_with_simplicial_set} and Proposition \ref{prop:useful_commutative_diagram_involving_nu_map} that the diagram
\begin{align*}
\xymatrix{
  (\capO\circ X)\Smash K_+\ar[d]_{\id\circ f\Smash\id}
  \ar[r]^-{m\Smash\id} & 
  X\Smash K_+\ar[r]^-{f\Smash\id} & 
  \Map(K_+,Y)\Smash K_+\ar[dd]^{\ev}\\
  \bigl(\capO\circ\Map(K_+,Y)\bigr)\Smash K_+\ar[d]_{\nu}\\
  \capO\circ\bigl(\Map(K_+,Y)\Smash K_+\bigr)\ar[r]^-{\id\circ\ev} & 
  \capO\circ Y\ar[r]^-{m} & Y
}
\end{align*}
in $\SymSeq$ commutes, which implies that $\function{f}{X}{\Map(K_+,Y)}$ is a map of left $\capO$-modules. Conversely, suppose $\function{f}{X}{\Map(K_+,Y)}$ is a map of left $\capO$-modules, and consider the corresponding map $\function{f}{X\Smash K_+}{Y}$ in $\SymSeq$. We need to verify that the adjoint map $\function{f}{\capO\circ (X\Smash K_+)}{Y}$ in $\LtO$ induces a map $\function{f}{X\tensordot K}{Y}$ in $\LtO$. Applying  $\capO\circ-$ to the commutative diagram in Proposition \ref{prop:useful_commutative_diagram_involving_nu_map}, it follows that $fd_0=fd_1$, which finishes the proof.
\end{proof}

\begin{defn} 
\label{defn:realization_algebras_and_modules}
Let $\capO$ be an operad in $\capR$-modules. The \emph{realization} functors $\functor{|-|_\AlgO}{\sAlgO}{\AlgO}$ and $\functor{|-|_\LtO}{\sLtO}{\LtO}$ for simplicial $\capO$-algebras and simplicial left $\capO$-modules are defined objectwise by the coends
\begin{align*}
  X\longmapsto |X|_\AlgO:=X\tensordot_{\Delta}\Delta[-]_+\ ,
  \quad\quad
  X\longmapsto |X|_\LtO:=X\tensordot_{\Delta}\Delta[-]_+\ .
\end{align*}
\end{defn}

Recall that the realization functors $|-|$ in Definition \ref{defn:realization} are the left adjoints in the adjunctions \eqref{eq:realization_mapping_object_adjunction_underlying} with right adjoints the functors $\Map(\Delta[-]_+,-)$. The following proposition is closely related to \cite[VII.3.3]{EKMM}; see also \cite[A]{Arone_Ching}.

\begin{prop}
\label{prop:realizations_are_isomorphic}
Let $\capO$ be an operad in $\capR$-modules and $X$ a simplicial $\capO$-algebra (resp. simplicial left $\capO$-module). The realization functors fit into adjunctions
\begin{align}
\label{eq:realization_adjunction_algebras_modules_nice}
&\xymatrix{
  \sAlgO
  \ar@<0.5ex>[r]^-{|-|_\AlgO} & 
  \AlgO,\ar@<0.5ex>[l]
}\quad\quad
\xymatrix{
  \sLtO
  \ar@<0.5ex>[r]^-{|-|_\LtO} & 
  \LtO,\ar@<0.5ex>[l]
}\\
\label{eq:underlying_realization_adjunction_algebras_modules_nice}
&\xymatrix{
  \sAlgO
  \ar@<0.5ex>[r]^-{|-|} & 
  \AlgO,\ar@<0.5ex>[l]
}\quad\quad
\xymatrix{
  \sLtO
  \ar@<0.5ex>[r]^-{|-|} &
  \LtO,\ar@<0.5ex>[l]
}
\end{align}
with left adjoints on top and right adjoints the functors $\Map(\Delta[-]_+,-)$. In particular, there are isomorphisms $|X|\Iso|X|_\AlgO$ in $\AlgO$ (resp. $|X|\Iso|X|_\LtO$ in $\LtO$), natural in $X$.  
\end{prop}

\begin{proof}
It suffices to consider the case of left $\capO$-modules. Let $X$ be a simplicial left $\capO$-module. Verifying  \eqref{eq:realization_adjunction_algebras_modules_nice} follows easily from \ref{prop:natural_isomorphisms_tensordot_and_mapping_object_algebras} and the universal property of coends. Consider \eqref{eq:underlying_realization_adjunction_algebras_modules_nice}. Suppose $\function{f}{|X|}{Y}$ is a map of left $\capO$-modules, and consider the corresponding left-hand commutative diagram
\begin{align*}
\xymatrix{
  \capO\circ|X|\Iso|\capO\circ X|\ar@<5.0ex>[d]_-{\id\circ f}\ar[r]^-{|m|} & 
  |X|\ar[d]^{f}\\
  \quad\quad\quad\quad\quad\capO\circ Y\ar[r]^-{m} & Y
}\quad\quad
\xymatrix{
  \capO\circ X\ar[d]_{(*)}\ar[r]^-{m} & X\ar[d]^{f}\\
  \Map(\Delta[-]_+,\capO\circ Y)\ar[r]^-{(\id,m)} & \Map(\Delta[-]_+,Y)
}
\end{align*}
in $\SymSeq$. Using the same notation for both $\function{f}{|X|}{Y}$ in $\SymSeq$ and its adjoint $\function{f}{X}{\Map(\Delta[-]_+,Y)}$ in $\sSymSeq$ \eqref{eq:realization_mapping_object_adjunction_underlying}, we know by \eqref{eq:realization_mapping_object_adjunction_underlying} that the left-hand diagram commutes if and only if its corresponding right-hand diagram in $\sSymSeq$ commutes. Since the map $(*)$ factors in $\sSymSeq$ as
$
  \capO\circ X\xrightarrow{\id\circ f}\capO\circ\Map(\Delta[-]_+,Y)
  \rarrow\Map(\Delta[-]_+,\capO\circ Y)
$,
the proof is complete.
\end{proof}

\begin{prop}
\label{prop:properties_for_establishing_simplicial_category}
Let $\capO$ be an operad in $\capR$-modules. Let $X,Y$ be $\capO$-algebras (resp. left $\capO$-modules) and $K,L$ simplicial sets. Then
\begin{itemize}
\item[(a)] the functor $\function{X\tensordot-}{\sSet}{\AlgO}$ (resp. $\function{X\tensordot-}{\sSet}{\LtO}$) commutes with all colimits and there are natural isomorphisms $X\tensordot *\Iso X$,
\item[(b)] there are isomorphisms $X\tensordot(K\times L)\Iso (X\tensordot K)\tensordot L$, natural in $X,K,L$.
\end{itemize}
\end{prop}

\begin{proof}
It suffices to consider the case of left $\capO$-modules. Part (a) follows easily from \eqref{eq:tensor_with_simplicial_sets_modules} and \eqref{eq:free_forgetful_adjunction}. Part (b) follows easily from the Yoneda lemma by verifying there are natural isomorphisms
$
  \hom_\LtO\bigl((X\tensordot K)\tensordot L,Y\bigr)
  \Iso
  \hom_\LtO\bigl(X\tensordot(K\times L),Y\bigr)
$; 
this involves several applications of Proposition \ref{prop:natural_isomorphisms_tensordot_and_mapping_object_algebras}, together with the observation that the natural isomorphism $\Map(K_+,\Map(L_+,Y))\Iso\Map(K_+\Smash L_+,Y)$ in $\SymSeq$ respects the left $\capO$-module structures.
\end{proof}

\begin{defn}
\label{defn:mapping_space_functor_simplicial_category}
Let $\capO$ be an operad in $\capR$-modules. Let $X,Y$ be $\capO$-algebras (resp. left $\capO$-modules). The \emph{mapping space} $\Hombold(X,Y)\in\sSet$  is defined objectwise by
\begin{align*}
  \Hombold(X,Y)_n := \hom_\AlgO(X\tensordot\Delta[n],Y)
  \quad
  \Bigl(\text{resp.}\ \ 
  \Hombold(X,Y)_n := \hom_\LtO(X\tensordot\Delta[n],Y)
  \Bigr).
\end{align*}
\end{defn}

\begin{prop}
\label{prop:simplicial_category_algebras_modules}
Let $\capO$ be an operad in $\capR$-modules. Then the category of $\capO$-algebras and the category of left $\capO$-modules are simplicial categories (in the sense of \cite[II.2.1]{Goerss_Jardine}), where the mapping space functor is that of Definition \ref{defn:mapping_space_functor_simplicial_category}.
\end{prop}

\begin{proof}
This follows from Propositions \ref{prop:natural_isomorphisms_tensordot_and_mapping_object_algebras} and \ref{prop:properties_for_establishing_simplicial_category}, together with \cite[II.2.4]{Goerss_Jardine}.
\end{proof}

\begin{prop}
\label{prop:simplicial_model_category_axiom}
Let $\capO$ be an operad in $\capR$-modules. Consider $\AlgO$ (resp. $\LtO$) with the model structure of Theorem \ref{thm:positive_flat_stable_AlgO} or \ref{thm:positive_stable_AlgO}.
\begin{itemize}
\item[(a)] If $\function{j}{K}{L}$ is a cofibration in $\sSet$, and $\function{p}{X}{Y}$ is a fibration in $\AlgO$ (resp. $\LtO$), then 
$
\xymatrix@1{
  \Map(L_+,X)\ar[r] & 
  \Map(K_+,X)\times_{\Map(K_+,Y)}
  \Map(L_+,Y)
}
$
is a fibration in $\AlgO$ (resp. $\LtO$) that is an acyclic fibration if either $j$ or $p$ is a weak equivalence.
\item[(b)] If $\function{j}{A}{B}$ is a cofibration in $\AlgO$ (resp. $\LtO$), and $\function{p}{X}{Y}$ is a fibration in $\AlgO$ (resp. $\LtO$), then the pullback corner map is a fibration
$
\xymatrix@1{
  \Hombold(B,X)\ar[r] & 
  \Hombold(A,X)\times_{\Hombold(A,Y)}
  \Hombold(B,Y)
}
$
in $\sSet$ that is an acyclic fibration if either $j$ or $p$ is a weak equivalence.
\end{itemize}
\end{prop}

\begin{proof}
Consider the case of left $\capO$-modules with the positive flat stable model structure. Part (a) follows from the proof of Proposition \ref{prop:mixing_flat_stable_with_positive_flat_stable_tensorcheck}, and part (b) follows from part (a) together with \cite[II.3.13]{Goerss_Jardine}. The case of $\capO$-algebras with the positive flat stable model structure is similar. Consider the case of $\capO$-algebras or left $\capO$-modules with the positive stable model structure. This follows by exactly the same argument as above together with the fact that $\capR\tensor G_0(-)_+$ applied to a cofibration in $\sSet$ gives a cofibration in $\ModR$ with the stable model structure (Section \ref{sec:model_structures} and \cite{Schwede_book_project}).
\end{proof}

The following theorem states that the simplicial structure respects the model category structure; this has also been observed in the context of symmetric spectra in \cite{Hornbostel, Schwede_book_project}; see also \cite{Arone_Ching, EKMM, McClure_Schwanzl_Vogt}.

\begin{thm}
\label{thm:simplicial_model_category_structure}
Let $\capO$ be an operad in $\capR$-modules. Consider $\AlgO$ (resp. $\LtO$) with the model structure of Theorem \ref{thm:positive_flat_stable_AlgO} or \ref{thm:positive_stable_AlgO}. Then $\AlgO$ (resp. $\LtO$) is a simplicial model category with the mapping space functor of Definition \ref{defn:mapping_space_functor_simplicial_category}.
\end{thm}

\begin{proof}
This follows from Propositions \ref{prop:simplicial_category_algebras_modules} and \ref{prop:simplicial_model_category_axiom}, together with \cite[II.3.13]{Goerss_Jardine}.
\end{proof}

\subsection{Homotopical analysis of the simplicial bar constructions}

The purpose of this subsection is to prove Theorem \ref{thm:reedy_cofibrant_for_bar_constructions}. This will require that we establish certain homotopical properties of the tensor product (Proposition \ref{prop:mixing_flat_stable_with_positive_flat_stable_tensorcheck}) and circle product (Theorem \ref{thm:mixing_flat_stable_with_positive_flat_stable} and Proposition \ref{prop:flat_stable_cofibration_properties_symmetric_sequences}) constructions arising in the description of the degenerate subobjects (Proposition \ref{prop:nice_description_of_degenerate_subobjects}).

\begin{prop}
\label{prop:flat_cofibrations_and_positive_flat_cofibrations}
Consider symmetric sequences in $\capR$-modules. Let $A,B$ be symmetric sequences.
\begin{itemize}
\item[(a)] $\function{f}{X}{Y}$ is a flat stable cofibration in $\ModR$ and $X_0\xrightarrow{\Iso}Y_0$ is an isomorphism if and only if $f$ is a positive flat stable cofibration in $\ModR$.
\item[(b)] $\function{f}{X}{Y}$ is a flat stable cofibration in $\SymSeq$ and $X[\mathbf{r}]_0\xrightarrow{\Iso}Y[\mathbf{r}]_0$ is an isomorphism for each $r\geq 0$, if and only if $f$ is a positive flat stable cofibration in $\SymSeq$.
\item[(c)] If $X,Y\in\ModR$, then there is a natural isomorphism $(X\Smash Y)_0\Iso X_0\Smash_{\capR_0} Y_0$.
\item[(d)] If $X,Y\in\ModR$ and $Y_0=*$, then $(X\Smash Y)_0=*$.
\item[(e)] If $B[\mathbf{r}]_0=*$ for each $r\geq 0$, then $(A\tensorcheck B)[\mathbf{r}]_0=*$ for each $r\geq 0$.
\item[(f)] If $A[\mathbf{0}]_0=*=B[\mathbf{r}]_0$ for each $r\geq 0$, then $(A\circ B)[\mathbf{r}]_0=*$ for each $r\geq 0$.
\item[(g)] If $A[\mathbf{r}]_0=*$ for each $r\geq 0$, then $(A\circ B)[\mathbf{r}]_0=*$ for each $r\geq 0$.
\end{itemize}
\end{prop}

\begin{proof}
Parts (a) and (b) follow from \ref{prop:cofibration_characterization}. The remaining parts are an easy exercise left to the reader.
\end{proof}

\begin{prop}
\label{prop:mixing_flat_stable_with_positive_flat_stable_tensorcheck}
Consider symmetric sequences in $\capR$-modules, and consider $\SymSeq$ with the positive flat stable model structure.
\begin{itemize}
\item[(a)] If $\function{i}{K}{L}$ is a flat stable cofibration in $\SymSeq$, and $\function{j}{A}{B}$ is a cofibration in $\SymSeq$, then 
$
\xymatrix@1{
  L\tensorcheck A\coprod_{K\tensorcheck A} K\tensorcheck B\ar[r] & 
  L\tensorcheck B
}
$
is a cofibration in $\SymSeq$ that is an acyclic cofibration if either $i$ or $j$ is a weak equivalence.
\item[(b)] If $\function{j}{A}{B}$ is a flat stable cofibration in $\SymSeq$, and $\function{p}{X}{Y}$ is a fibration in $\SymSeq$, then
$
\xymatrix@1{
  \Map^\tensorcheck(B,X)\ar[r] & 
  \Map^\tensorcheck(A,X)\times_{\Map^\tensorcheck(A,Y)}
  \Map^\tensorcheck(B,Y)
}
$
is a fibration in $\SymSeq$ that is an acyclic fibration if either $j$ or $p$ is a weak equivalence.
\item[(c)] If $\function{j}{A}{B}$ is a cofibration in $\SymSeq$, and $\function{p}{X}{Y}$ is a fibration in $\SymSeq$, then
$
\xymatrix@1{
  \Map^\tensorcheck(B,X)\ar[r] & 
  \Map^\tensorcheck(A,X)\times_{\Map^\tensorcheck(A,Y)}
  \Map^\tensorcheck(B,Y)
}
$
is a flat stable fibration in $\SymSeq$ that is a flat stable acyclic fibration if either $j$ or $p$ is a weak equivalence.
\end{itemize}
\end{prop}

\begin{proof}
Consider part (a). Suppose $\function{i}{K}{L}$ is a flat stable cofibration in $\SymSeq$ and $\function{j}{A}{B}$ is a cofibration in $\SymSeq$. The pushout corner map is a flat stable cofibration in $\SymSeq$ by \cite[6.1]{Harper_Modules}, hence by Proposition \ref{prop:flat_cofibrations_and_positive_flat_cofibrations} it suffices to verify the pushout corner map
$
\xymatrix@1{
  (L\tensorcheck A)[\mathbf{r}]_0\coprod_{(K\tensorcheck A)[\mathbf{r}]_0} (K\tensorcheck B)[\mathbf{r}]_0\ar[r] & 
  (L\tensorcheck B)[\mathbf{r}]_0
}
$
is an isomorphism for each $r\geq 0$. We can therefore conclude by \eqref{eq:tensor_check_calc} together with Proposition \ref{prop:flat_cofibrations_and_positive_flat_cofibrations}. The other cases are similar. Parts (b) and (c) follow from part (a) and the natural isomorphisms \eqref{eq:tensorcheck_mapping_sequence_adjunction}.
\end{proof}

\begin{thm}
\label{thm:mixing_flat_stable_with_positive_flat_stable}
Consider symmetric sequences in $\capR$-modules, and consider $\SymSeq$ with the positive flat stable model structure.
\begin{itemize}
\item[(a)] If $\function{i}{K}{L}$ is a map in $\SymSeq$ such that $K[\mathbf{r}]\rarrow L[\mathbf{r}]$ is a flat stable cofibration in $\ModR$ for each $r\geq 1$, and $\function{j}{A}{B}$ is a cofibration between cofibrant objects in $\SymSeq$, then
$
\xymatrix@1{
  L\circ A\coprod_{K\circ A} K\circ B\ar[r] & L\circ B
}
$
is a cofibration in $\SymSeq$ that is an acyclic cofibration if either $i$ or $j$ is a weak equivalence.
\item[(b)] If $\function{i}{K}{L}$ is a map in $\SymSeq$ such that $K[\mathbf{r}]\rarrow L[\mathbf{r}]$ is a flat stable cofibration in $\ModR$ for each $r\geq 0$, $K[\mathbf{0}]_0\xrightarrow{\Iso}L[\mathbf{0}]_0$ is an isomorphism, and $B$ is a cofibrant object in $\SymSeq$, then the map 
$
 K\circ B\rarrow L\circ B
$
is a cofibration in $\SymSeq$ that is an acyclic cofibration if $i$ is a weak equivalence.
\end{itemize}
\end{thm}

\begin{proof}
Consider part (a). Suppose $K[\mathbf{t}]\rarrow L[\mathbf{t}]$ is a flat stable cofibration in $\ModR$ for each $t\geq 1$, and $\function{j}{A}{B}$ is a cofibration between cofibrant objects in $\SymSeq$. We want to verify each 
$
\xymatrix@1{
  L[\mathbf{t}]\Smash_{\Sigma_t} A^{\tensorcheck t}
  \coprod_{K[\mathbf{t}]\Smash_{\Sigma_t} A^{\tensorcheck t}} 
  K[\mathbf{t}]\Smash_{\Sigma_t} B^{\tensorcheck t}
  \ar[r] & 
  L[\mathbf{t}]\Smash_{\Sigma_t} B^{\tensorcheck t}
}
$
is a cofibration in $\SymSeq$. If $t=0$, this map is an isomorphism. Let $t\geq 1$. Consider any acyclic fibration $\function{p}{X}{Y}$ in $\SymSeq$. We want to show that the pushout corner map has the left lifting property with respect to $p$. Consider any such lifting problem; we want to verify that the corresponding solid commutative diagram
\begin{align*}
\xymatrix{
  A^{\tensorcheck t}\ar[d]\ar[r] & \Map(L[\mathbf{t}], X)\ar[d]^{(*)}\\
  B^{\tensorcheck t}\ar[r]\ar@{.>}[ur] &
  \Map(K[\mathbf{t}], X)\times_{\Map(K[\mathbf{t}], Y)}
  \Map(L[\mathbf{t}], Y)
}
\end{align*}
in $\SymSeq^{\Sigma_t}$ has a lift. We know that the left-hand vertical map is a cofibration in $\SymSeq^{\Sigma_t}$ by Proposition \ref{prop:cofibration}, hence it suffices to verify that the map $(*)[\mathbf{r}]$ is a positive flat stable acyclic fibration in $\ModR$ for each $r\geq 0$.  By considering symmetric sequences concentrated at $0$, Proposition \ref{prop:mixing_flat_stable_with_positive_flat_stable_tensorcheck} finishes the argument for this case. The other cases are similar. Consider part (b). Suppose $K[\mathbf{t}]\rarrow L[\mathbf{t}]$ is a flat stable cofibration in $\ModR$ for each $t\geq 0$, $K[\mathbf{0}]_0\xrightarrow{\Iso}L[\mathbf{0}]_0$ is an isomorphism, and $B$ is a cofibrant object in $\SymSeq$. We need to check that each induced map
$
  K[\mathbf{t}]\Smash_{\Sigma_t}B^{\tensorcheck t}\rarrow
  L[\mathbf{t}]\Smash_{\Sigma_t}B^{\tensorcheck t}
$
is a cofibration in $\SymSeq$. The proof of part (a) implies this for $t\geq 1$, and Proposition \ref{prop:flat_cofibrations_and_positive_flat_cofibrations} implies this for $t=0$. The other case is similar.
\end{proof}

\begin{prop}
\label{prop:flat_stable_cofibration_properties_symmetric_sequences}
Let $\capO$ be an operad in $\capR$-modules such that $\capO[\mathbf{0}]=*$, and let $\function{\eta}{I}{\capO}$ be its unit map. Assume that  $I[\mathbf{r}]\rarrow\capO[\mathbf{r}]$ is a flat stable cofibration between flat stable cofibrant objects in $\ModR$ for each $r\geq 0$. 
\begin{itemize}
\item[(a)] If $\function{i}{K}{L}$ is a map in $\SymSeq$ such that $K[\mathbf{r}]\rarrow L[\mathbf{r}]$ is a flat stable cofibration in $\ModR$ for each $r\geq 1$, then the pushout corner map 
$
\xymatrix@1{
  \bigl(L\circ I\coprod_{K\circ I} K\circ \capO\bigr)[\mathbf{r}]
  \ar[r] & (L\circ\capO)[\mathbf{r}]
}
$
is a flat stable cofibration in $\ModR$ for each $r\geq 0$.
\item[(b)] If $t\geq 1$, then the induced map $(I^{\tensorcheck t})[\mathbf{r}]\rarrow(\capO^{\tensorcheck t})[\mathbf{r}]$ is a flat stable cofibration in $\ModR^{\Sigma_t}$ for each $r\geq 0$.
\end{itemize}
\end{prop}

\begin{proof}
Consider part (b). The induced map is an isomorphism for $0\leq r\leq t-1$ and the case for $r\geq t$ follows from Proposition \ref{prop:cofibration_characterization} by arguing as in the proof of Proposition \ref{prop:cofibration}. Consider part (a). We need to verify that each 
\begin{align*}
\xymatrix{
  L[\mathbf{t}]\Smash_{\Sigma_t} (I^{\tensorcheck t})[\mathbf{r}]
  \coprod_{K[\mathbf{t}]\Smash_{\Sigma_t} (I^{\tensorcheck t})[\mathbf{r}]} 
  K[\mathbf{t}]\Smash_{\Sigma_t} (\capO^{\tensorcheck t})[\mathbf{r}]
  \ar[r] & 
  L[\mathbf{t}]\Smash_{\Sigma_t} (\capO^{\tensorcheck t})[\mathbf{r}]
}
\end{align*}
is a flat stable cofibration in $\ModR$. If $t=0$, this map is an isomorphism. Let $t\geq 1$, and let $\function{p}{X}{Y}$ be a flat stable acyclic fibration in $\ModR$. We need to show that the pushout corner map has the left lifting property with respect to $p$. Consider any such lifting problem; we want to verify that the corresponding solid commutative diagram
\begin{align*}
\xymatrix{
  (I^{\tensorcheck t})[\mathbf{r}]\ar[d]\ar[r] & 
  \Map(L[\mathbf{t}], X)\ar[d]^{(*)}\\
  (\capO^{\tensorcheck t})[\mathbf{r}]\ar[r]\ar@{.>}[ur] &
  \Map(K[\mathbf{t}], X)\times_{\Map(K[\mathbf{t}], Y)}
  \Map(L[\mathbf{t}], Y)
}
\end{align*}
in $\ModR^{\Sigma_t}$ has a lift. The left-hand vertical map is a flat stable cofibration in $\ModR^{\Sigma_t}$ by part (b), hence it suffices to verify the map $(*)$ is a flat stable acyclic fibration in $\ModR$. By assumption, each $K[\mathbf{t}]\rarrow L[\mathbf{t}]$ is a flat stable cofibration in $\ModR$, which finishes the proof.
\end{proof}

\begin{defn} Let $\capO$ be an operad in $\capR$-modules, $t\geq 1$ and $n\geq 0$.
\begin{itemize}
\item $\Cube_t$ is the category with objects the vertices $(v_1,\dotsc,v_t)\in\{0,1\}^t$ of the unit $t$-cube. There is at most one morphism between any two objects, and there is a morphism
$
  (v_1,\dotsc,v_t)\rarrow  (v_1',\dotsc,v_t')
$
if and only if $v_i\leq v_i'$ for each $1\leq i\leq t$. In particular, $\Cube_t$ is the category associated to a partial order on the set $\{0,1\}^t$.
\item The \emph{punctured cube} $\pCube_t$ is the full subcategory of $\Cube_t$ with all objects except the terminal object $(1,\dotsc,1)$ of $\Cube_t$.
\item Define the functor $\function{w}{\pCube_t}{\SymSeq}$ objectwise by
\begin{align*}
  w(v_1,\dotsc,v_t):=c_1\circ\dotsb\circ c_t\quad\quad\text{with}
  \quad\quad
  c_i :=
  \left\{
    \begin{array}{rl}
    I,&\text{for $v_i=0$,}\\
    \capO,&\text{for $v_i=1$,}
  \end{array}
  \right.
\end{align*}
and with morphisms induced by the unit map $\function{\eta}{I}{\capO}$.
\item If $X$ is an object in $\sModR$ or $\sSymSeq$, denote by $DX_n\subsetof X_n$ the \emph{degenerate subobject} \cite[9.12]{Harper_Bar} of $X_n$.
\end{itemize}
\end{defn}

The following proposition gives a useful construction of degenerate subobjects.

\begin{prop}
\label{prop:nice_description_of_degenerate_subobjects}
Let $\capO$ be an operad in $\capR$-modules, $Y$ an $\capO$-algebra (resp. left $\capO$-module) and $N$ a right $\capO$-module. Let $t\geq 1$ and $n\geq 0$. Define $X:=\BAR(N,\capO,Y)$ and $Q^t:=\colim_{\pCube_t}(N\circ w)$, and consider the induced maps $\function{\eta_*}{Q^0:=*}{N}$ and $\function{\eta_*}{Q^t}{N\circ\capO^{\circ t}}$.
\begin{itemize}
\item[(a)] The inclusion map $DX_n\rarrow X_n$ is isomorphic to the map
$
  Q^n\circ (Y)\xrightarrow{\eta_*\circ(\id)}
  N\circ\capO^{\circ n}\circ (Y)
$ (resp. 
$
  Q^n\circ Y\xrightarrow{\eta_*\circ\id} 
  N\circ\capO^{\circ n}\circ Y
$).
\item[(b)] The induced map $\function{\eta_*}{Q^{n+1}}{N\circ\capO^{\circ(n+1)}}$ is isomorphic to the pushout corner map
$
  (N\circ\capO^{\circ n}\circ I)\amalg_{(Q^n\circ I)}(Q^n\circ\capO)\rarrow N\circ\capO^{\circ (n+1)}
$ induced by $\function{\eta}{I}{\capO}$ and $\function{\eta_*}{Q^n}{N\circ\capO^{\circ n}}$.
\end{itemize}
\end{prop}

\begin{proof}
It suffices to consider the case of left $\capO$-modules. Consider part (a). It follows easily from \cite[9.23]{Harper_Bar}, together with the fact that $\function{-\circ Y}{\SymSeq}{\SymSeq}$ commutes with colimits \eqref{eq:circle_mapping_sequence_adjunction}, that there are natural isomorphisms
\begin{align*}
  DX_0&\Equal*,\quad DX_1\Iso N\circ I\circ Y,\\
  DX_2&\Iso (N\circ\capO\circ I\circ Y)\amalg_{(N\circ I\circ I\circ Y)}
  (N\circ I\circ\capO\circ Y)\\
  &\Iso\bigl(
  (N\circ\capO\circ I)\amalg_{(N\circ I\circ I)}
  (N\circ I\circ\capO)
  \bigr)\circ Y,\, \dotsc\, ,\\
  DX_t&\Iso\colim_{pCube_t}(N\circ w\circ Y)\Iso
  \bigl(\colim_{pCube_t}(N\circ w)\bigr)\circ Y
\end{align*}
in $\SymSeq$. Consider part (b). Since $\function{-\circ B}{\SymSeq}{\SymSeq}$ commutes with colimits for each $B\in\SymSeq$, it follows easily that the colimit $Q^{n+1}$ may be computed inductively using pushout corner maps.
\end{proof}

\begin{thm}
\label{thm:inclusion_of_degenerate_subobjects_positive_flat_stable}
Let $\capO$ be an operad in $\capR$-modules such that $\capO[\mathbf{0}]=*$, $Y$ an $\capO$-algebra (resp. left $\capO$-module) and $N$ a right $\capO$-module, and consider the unit map $\function{\eta}{I}{\capO}$. Assume that  $I[\mathbf{r}]\rarrow\capO[\mathbf{r}]$ is a flat stable cofibration between flat stable cofibrant objects in $\ModR$ for each $r\geq 0$ and that $N[\mathbf{r}]$ is flat stable cofibrant in $\ModR$ for each $r\geq 0$. Let $X:=\BAR(N,\capO,Y)$. If $Y$ is positive flat stable cofibrant in $\ModR$ (resp. $\SymSeq$) and $N[\mathbf{0}]_0=*$, then the inclusion maps
\begin{align*}
  *\rarrow DX_n\rarrow X_n,\quad\quad
  *\rarrow |\BAR(N,\capO,Y)|,
\end{align*}
are positive flat stable cofibrations in $\ModR$ (resp. $\SymSeq$) for each $n\geq 0$. In particular, the simplicial bar construction $\BAR(N,\capO,Y)$ is Reedy cofibrant in $\sModR$ (resp. $\sSymSeq$) with respect to the positive flat stable model structure.
\end{thm}

\begin{proof}
It suffices to consider the case of left $\capO$-modules. Consider Proposition \ref{prop:nice_description_of_degenerate_subobjects}; let's verify that the left-hand induced maps
\begin{align}
\label{eq:cofibration_properties_of_degenerate_subobjects}
  *\rarrow Q^n[\mathbf{r}]\rarrow (N\circ\capO^{\circ n})[\mathbf{r}],
  \quad\quad
  Q^n[\mathbf{0}]_0=*=(N\circ\capO^{\circ n})[\mathbf{0}]_0
\end{align} 
are flat stable cofibrations in $\ModR$ for each $n,r\geq 0$ and that the right-hand relations are satisfied for each $n\geq 0$. It is easy to check this for $n=0$, and by induction on $n$, the general case follows from Propositions \ref{prop:flat_stable_cofibration_properties_symmetric_sequences} and \ref{prop:nice_description_of_degenerate_subobjects}. By assumption, $Y$ is positive flat stable cofibrant in $\SymSeq$, hence by Proposition \ref{prop:nice_description_of_degenerate_subobjects} and Theorem \ref{thm:mixing_flat_stable_with_positive_flat_stable}, the inclusion maps
$
  *\rarrow DX_n\rarrow X_n
$
are positive flat stable cofibrations in $\SymSeq$ for each $n\geq 0$. Since $DX_n$ and $X_n$ are positive flat stable cofibrant in $\SymSeq$ for each $n\geq 0$, we know by \ref{prop:cofibration_characterization} that the relations
$
  DX_n[\mathbf{r}]_0=*=X_n[\mathbf{r}]_0
$
are satisfied for each $n,r\geq 0$.  It then follows easily from the skeletal filtration of realization \cite[9.11, 9.16]{Harper_Bar}, together with Proposition \ref{prop:flat_cofibrations_and_positive_flat_cofibrations}, that $|\BAR(N,\capO,Y)|$ is positive flat stable cofibrant in $\SymSeq$. It is easy to check that the natural map $DX_n\rarrow X_n$ is isomorphic to the natural map $L_nX\rarrow X_n$ described in \cite[VII.1.8]{Goerss_Jardine}. Hence, in particular, we have verified that $X$ is Reedy cofibrant \cite[VII.2.1]{Goerss_Jardine} in $\sSymSeq$.
\end{proof}

\begin{prop}
\label{prop:cofibrant_bar_constructions_for_chain_complexes_algebras}
Let $\capO$ be an operad in $\capR$-modules, $Y$ an $\capO$-algebra (resp. left $\capO$-module) and $N$ a right $\capO$-module. Consider $\SymSeq$ with the flat stable model structure. Assume that the unit map $I\rarrow\capO$ is a cofibration between cofibrant objects in $\SymSeq$ and that $N$ is cofibrant in $\SymSeq$. If $Y$ is flat stable cofibrant in $\ModR$ (resp. $\SymSeq$), then $|\BAR(N,\capO,Y)|$ is flat stable cofibrant in $\ModR$ (resp. $\SymSeq$).
\end{prop}

\begin{proof}
Argue as in the proof of Theorem \ref{thm:inclusion_of_degenerate_subobjects_positive_flat_stable}.
\end{proof}

\begin{proof}[Proof of Theorem \ref{thm:reedy_cofibrant_for_bar_constructions}]
It suffices to consider the case of left $\capO$-modules. Consider part (a). This follows as in the proof of Theorem \ref{thm:inclusion_of_degenerate_subobjects_positive_flat_stable}, except using the skeletal filtration in \cite[VII.3.8]{Goerss_Jardine}, 
Proposition \ref{prop:nice_description_of_degenerate_subobjects} and Theorem \ref{thm:mixing_flat_stable_with_positive_flat_stable}, together with the fact that $\function{\capO'\circ-}{\SymSeq}{\Lt_{\capO'}}$ is a left Quillen functor and hence preserves both colimiting cones and cofibrations. Part (b) follows immediately from part (a) together with Proposition \ref{prop:realizations_are_isomorphic}, Theorem \ref{thm:simplicial_model_category_structure}, and \cite[VII.3.4]{Goerss_Jardine}.
\end{proof}

\section{Model structures}
\label{sec:model_structures}

The purpose of this section is to prove Theorems \ref{thm:positive_flat_stable_AlgO}, \ref{thm:positive_stable_AlgO}, and \ref{thm:comparing_homotopy_categories}, together with Theorems \ref{thm:bar_calculates_derived_circle}, \ref{main_hocolim_theorem}, and \ref{thm:fattened_replacement} which improve the main results in \cite{Harper_Spectra, Harper_Bar} from operads in symmetric spectra to the more general context of operads in $\capR$-modules. Our approach to this generalization, which is motivated by Hornbostel \cite{Hornbostel}, is to establish only the necessary minimum of technical propositions for $\capR$-modules needed for the proofs of the main results as described in \cite{Harper_Spectra, Harper_Bar} to remain valid in the more general context of $\capR$-modules.

\subsection{Smash products and $\capR$-modules}

Denote by $(\Spectra,\tensor_S,S)$ the closed symmetric monoidal category of symmetric spectra \cite{Hovey_Shipley_Smith, Schwede_book_project}. To keep this section as concise as possible, from now on we will freely use the notation from \cite[Section 2]{Harper_Spectra} which agrees (whenever possible) with \cite{Hovey_Shipley_Smith}.  

The following is proved in \cite[2.1]{Hovey_Shipley_Smith} and states that tensor product in the category $\sSet^\Sigma_*$ inherits many of the good properties of smash product in the category $\sSet_*$.

\begin{prop}
\label{prop:closed_symmetric_monoidal_structure_on_sym_sequences_pointed_ssets}
$(\sSet^\Sigma_*,\tensor,S^0)$ has the structure of a closed symmetric monoidal category. All small limits and colimits exist and are calculated objectwise. The unit $S^0\in\sSet^\Sigma_*$ is given by $S^0[\mathbf{n}]=*$ for each $n\geq 1$ and $S^0[\mathbf{0}]=S^0$.
\end{prop}

There are two naturally occurring maps $S\tensor S\rarrow S$ and $S^0\rarrow S$ in $\sSet^\Sigma_*$ that give $S$ the structure of a commutative monoid in $(\sSet^\Sigma_*,\tensor,S^0)$. Furthermore, for any symmetric spectrum $X$, there is a naturally occurring map $\function{m}{S\tensor X}{X}$ endowing $X$ with a left action of $S$ in $(\sSet^\Sigma_*,\tensor,S^0)$. The following is proved in \cite[2.2]{Hovey_Shipley_Smith} and provides a useful interpretation of symmetric spectra. 

\begin{prop} 
\label{prop:symmetric_spectra_as_S_modules}
Define the category $\Sigma':=\amalg_{n\geq 0}\Sigma_n$, a skeleton of $\Sigma$.
\begin{itemize}
\item[(a)] The sphere spectrum $S$ is a commutative monoid in $(\sSet^\Sigma_*,\tensor,S^0)$.
\item[(b)] The category of symmetric spectra is equivalent to the category of left $S$-modules in $(\sSet^\Sigma_*,\tensor,S^0)$.
\item[(c)] The category of symmetric spectra is isomorphic to the category of left $S$-modules in $(\sSet^{\Sigma'}_*,\tensor,S^0)$.
\end{itemize}
\end{prop}

In this paper we will not distinguish between these equivalent descriptions of symmetric spectra. 

\begin{defn}
\label{defn:left_R_modules}
Let $\capR$ be a commutative monoid in $(\Spectra,\tensor_S,S)$ (Basic Assumption \ref{assumption:commutative_ring_spectrum}). A \emph{left $\capR$-module} is an object in $(\Spectra,\tensor_S,S)$ with a left action of $\capR$ and a \emph{morphism of left $\capR$-modules} is a map in $\Spectra$ that respects the left $\capR$-module structure. Denote by $\ModR$ the category of left $\capR$-modules and their morphisms.
\end{defn}

The \emph{smash product} $X\Smash Y\in\ModR$ of left $\capR$-modules $X$ and $Y$ is defined by
\begin{align*}
  X\Smash Y
  :=\colim
  \Bigl(
  \xymatrix{
  X\tensor_S Y & X\tensor_S \capR\tensor_S Y
  \ar@<-0.5ex>[l]_-{m\tensor\id}\ar@<0.5ex>[l]^-{\id\tensor m}
  }
  \Bigr)
  \Iso
  \colim
  \Bigl(
  \xymatrix{
  X\tensor Y & X\tensor \capR\tensor Y
  \ar@<-0.5ex>[l]_-{m\tensor\id}\ar@<0.5ex>[l]^-{\id\tensor m}
  }
  \Bigr)
\end{align*}
the indicated colimit. Here, $m$ denotes the indicated $\capR$-action map and  since $\capR$ is a commutative monoid in $(\Spectra,\tensor_S,S)$, a left action of $\capR$ on $X$ determines a right action $\function{m}{X\tensor_S \capR}{X}$, which gives $X$ the structure of an $(\capR,\capR)$-bimodule. Hence the smash product $X\Smash Y$ of left $\capR$-modules, which is naturally isomorphic to $X\tensor_\capR Y$, has the structure of a left $\capR$-module.

\begin{rem}
\label{rem:dropping_the_adjective_left}
Since $\capR$ is commutative, we usually drop the adjective ``left'' and simply refer to the objects of $\ModR$ as $\capR$-modules.
\end{rem}

The following is an easy consequence of \cite[2.2]{Hovey_Shipley_Smith}.

\begin{prop}
$(\ModR,\Smash,\capR)$ has the structure of a closed symmetric monoidal category. All small limits and colimits exist and are calculated objectwise.
\end{prop}

\subsection{Model structures on $\capR$-modules}
\label{sec:model_structures_on_capR_modules}

The material below intentionally parallels \cite[Section 4]{Harper_Spectra}, except that we work in the more general context of $\capR$-modules instead of symmetric spectra. We need to recall just enough notation so that we can describe and work with the (positive) flat stable model structure on $\capR$-modules, and the corresponding projective model structures on the diagram categories $\SymSeq$ and $\SymSeq^G$ of $\capR$-modules, for $G$ a finite group. The functors involved in such a description are easy to understand when defined as the left adjoints of  appropriate functors, which is how they naturally arise in this context.

For each $m\geq 0$ and subgroup $H\subsetof\Sigma_m$, denote by $\function{l}{H}{\Sigma_m}$ the inclusion of groups and define the \emph{evaluation} functor $\functor{\ev_m}{\sSet^\Sigma_*}{\sSet^{\Sigma_m}_*}$ objectwise by $\ev_m(X):=X_m$. There are adjunctions
$
\xymatrix@1{
  \sSet_*\ar@<0.5ex>[r] & 
  \sSet^{H}_*\ar@<0.5ex>[l]^-{\lim_H}
  \ar@<0.5ex>[r]^-{\Sigma_m\cdot_H-} &
  \sSet^{\Sigma_m}_*\ar@<0.5ex>[l]^-{l^*}\ar@<0.5ex>[r] &
  \sSet^{\Sigma}_*\ar@<0.5ex>[l]^-{\ev_m}
}
$
with left adjoints on top. Define $\functor{G^H_m}{\sSet_*}{\sSet^\Sigma_*}$ to be the composition of the three top functors, and define $\functor{\lim_H \ev_m}{\sSet^\Sigma_*}{\sSet_*}$ to be the composition of the three bottom functors; we have dropped the restriction functor $l^*$ from the notation. It is easy to check that if $K\in\sSet_*$, then $G^H_m(K)$ is the object in $\sSet^\Sigma_*$ that is concentrated at $m$ with value $\Sigma_m\cdot_H K$. Consider the forgetful functors $\Spectra\rarrow\sSet^\Sigma_*$ and $\ModR\rarrow\Spectra$. It follows from Proposition \ref{prop:symmetric_spectra_as_S_modules} that there are adjunctions
\begin{align}
\label{eq:new_adjunctions_for_spectra_as_S_modules}
\xymatrix{
  \sSet^\Sigma_*\ar@<0.5ex>[r]^-{S\tensor-} & 
  \Spectra\ar@<0.5ex>[l]\ar@<0.5ex>[r]^-{\capR\tensor_S-} &
  \ModR\ar@<0.5ex>[l]
},
\quad\quad
\xymatrix{
  \sSet^\Sigma_*\ar@<0.5ex>[r]^-{\capR\tensor-} &
  \ModR\ar@<0.5ex>[l]
},
\end{align}
with left adjoints on top; the latter adjunction is the composition of the former adjunctions. For each $p\geq 0$, define the \emph{evaluation} functor $\functor{\Ev_p}{\SymSeq}{\ModR}$ objectwise by $\Ev_p(A):=A[\mathbf{p}]$, and for each finite group $G$, consider the forgetful functor $\SymSeq^G\rarrow\SymSeq$. There are adjunctions
$
\xymatrix@1{
  \ModR\ar@<0.5ex>[r]^-{G_p} & 
  \SymSeq\ar@<0.5ex>[l]^-{\Ev_p}\ar@<0.5ex>[r]^-{G\cdot-} & 
  \SymSeq^G\ar@<0.5ex>[l]
}
$
with left adjoints on top. It is easy to check that if $X\in\ModR$, then $G_p(X)$ is the symmetric sequence concentrated at $p$ with value $X\cdot\Sigma_p$. Putting it all together, there are adjunctions
\begin{align}
\label{eq:adjunctions_stable_flat}
\xymatrix{
  \sSet_*\ar@<0.5ex>[r]^-{G^H_m} & 
  \sSet^\Sigma_*\ar@<0.5ex>[l]^-{\ \lim_H \ev_m}
  \ar@<0.5ex>[r]^-{\capR\tensor-}&
  \ModR\ar@<0.5ex>[l]\ar@<0.5ex>[r]^-{G_p}&
  \SymSeq\ar@<0.5ex>[l]^-{\Ev_p}\ar@<0.5ex>[r]^-{G\cdot-}&
  \SymSeq^G\ar@<0.5ex>[l]
}
\end{align}
with left adjoints on top. We are now in a good position to describe several  useful model structures. It is proved in \cite{Shipley_comm_ring} that the following two model category structures exist on $\capR$-modules.

\begin{defn}\ 
\label{defn:lets_define_flat_model_structure}
\begin{itemize}
\item[(a)] The \emph{flat stable model structure} on $\ModR$ has weak equivalences the stable equivalences, cofibrations the retracts of (possibly transfinite) compositions of pushouts of maps 
\begin{align*}
  \capR\tensor G^H_m \partial\Delta[k]_+\rarrow 
  \capR\tensor G^H_m\Delta[k]_+\quad
  (m\geq 0,\ k\geq 0,\ H\subsetof \Sigma_m \ \text{subgroup}),
\end{align*} 
and fibrations the maps with the right lifting property with respect to the acyclic cofibrations.
\item[(b)] The \emph{positive flat stable model structure} on $\ModR$ has weak equivalences the stable equivalences, cofibrations the retracts of (possibly transfinite) compositions of pushouts of maps
\begin{align*} 
  \capR\tensor G^H_m \partial\Delta[k]_+\rarrow 
  \capR\tensor G^H_m\Delta[k]_+\quad
  (m\geq 1,\ k\geq 0,\ H\subsetof \Sigma_m \ \text{subgroup}),
\end{align*} 
and fibrations the maps with the right lifting property with respect to the acyclic cofibrations.
\end{itemize}
\end{defn}

\begin{rem}
\label{rem:flat_notation}
In the sets of maps above, it is important to note that $H$ varies over all subgroups of $\Sigma_m$. For ease of notation purposes, we have followed Schwede \cite{Schwede_book_project} in using the term \emph{flat} (e.g., flat stable model structure) for what is called $\capR$ (e.g., stable $\capR$-model structure) in \cite{Hovey_Shipley_Smith, Schwede, Shipley_comm_ring}. 
\end{rem}

Several useful properties of the flat stable model structure are summarized in the following two propositions, which are consequences of \cite[5.3, 5.4] {Hovey_Shipley_Smith} as indicated below; see also \cite{Schwede_book_project}. These properties are used in several sections of this paper.

\begin{prop}
\label{prop:weak_equivalences_and_monos_are_preserved}
Consider $\ModR$ with the flat stable model structure. If $Z\in\ModR$ is cofibrant, then the functor $\function{-\Smash Z}{\ModR}{\ModR}$ preserves (i) weak equivalences and (ii) monomorphisms.
\end{prop}

\begin{prop}
\label{prop:sending_cofibrations_to_monos}
If $B\in\ModR$ and $X\rarrow Y$ is a flat stable cofibration in $\ModR$, then $B\Smash X\rarrow B\Smash Y$ in $\ModR$ is a monomorphism.
\end{prop}

\begin{proof}[Proof of Proposition \ref{prop:weak_equivalences_and_monos_are_preserved}]
Part (i) is the $\capR$-module analog of \cite[5.3.10]{Hovey_Shipley_Smith}. It can also be verified as a consequence of \cite[5.3.10]{Hovey_Shipley_Smith} by arguing exactly as in the proof of \cite[4.29(b)]{Harper_Spectra}. Part (ii) follows from the $\capR$-module analog of \cite[5.3.7]{Hovey_Shipley_Smith}; see, \cite[proof of 5.4.4]{Hovey_Shipley_Smith} or \cite{Schwede_book_project}.
\end{proof}

\begin{proof}[Proof of Proposition \ref{prop:sending_cofibrations_to_monos}]
This follows from the $\capR$-module analog of \cite[5.3.7]{Hovey_Shipley_Smith}; see, \cite[proof of 5.4.4]{Hovey_Shipley_Smith} or \cite{Schwede_book_project}.
\end{proof}

The \emph{stable model structure} on $\ModR$ is defined by fixing $H$ in Definition \ref{defn:lets_define_flat_model_structure}(a) to be the trivial subgroup. This is one of several model category structures that is proved in \cite{Hovey_Shipley_Smith} to exist on $\capR$-modules. The \emph{positive stable model structure} on $\ModR$ is defined by fixing $H$ in Definition \ref{defn:lets_define_flat_model_structure}(b) to be the trivial subgroup. This model category structure is proved in \cite{Mandell_May_Schwede_Shipley} to exist on $\capR$-modules. It follows immediately that every (positive) stable cofibration is a (positive) flat stable cofibration.

These model structures on $\capR$-modules enjoy several good properties, including that smash products of $\capR$-modules mesh nicely with each of the model structures defined above. More precisely, each model structure above is cofibrantly generated, by generating cofibrations and acyclic cofibrations with small domains, and with respect to each model structure $(\ModR,\Smash,\capR)$ is a monoidal model category.

If $G$ is a finite group, it is easy to check that the diagram categories $\ModR^G$, $\SymSeq$ and $\SymSeq^G$ inherit corresponding projective model category structures, where the weak equivalences (resp. fibrations) are the maps that are underlying objectwise weak equivalences (resp. objectwise fibrations). We refer to these model structures by the names above (e.g., the \emph{positive flat stable} model structure on $\SymSeq^G$). Each of these model structures is cofibrantly generated by generating cofibrations and acyclic cofibrations with small domains. Furthermore, with respect to each model structure $(\SymSeq,\tensor,1)$ is a monoidal model category; this is proved in \cite{Harper_Modules}.

\subsection{Model structures on $\capO$-algebras and left $\capO$-modules}

The purpose of this subsection is to prove the following two theorems. These generalizations are motivated by Hornbostel \cite{Hornbostel} and improve the corresponding results in \cite[1.1, 1.3]{Harper_Spectra} from operads in symmetric spectra to the more general context involving operads in $\capR$-modules and play a key role in this paper. An important first step in establishing these theorems was provided by the characterization given by Schwede \cite{Schwede_book_project} of flat stable cofibrations in $\ModR$ in terms of objects with an $\capR_0$-action; see Proposition \ref{prop:cofibration_characterization} below for the needed generalization of this.

\begin{thm}[Positive flat stable model structure on $\AlgO$ and $\LtO$]
\label{thm:positive_flat_stable_AlgO}
Let $\capO$ be an operad in $\capR$-modules. Then the category of $\capO$-algebras (resp. left $\capO$-modules) has a model category structure with weak equivalences the stable equivalences (resp. objectwise stable equivalences) and fibrations the maps that are positive flat stable fibrations (resp. objectwise positive flat stable fibrations) in the underlying category of $\capR$-modules (Definition \ref{defn:lets_define_flat_model_structure}(b)).
\end{thm}

\begin{thm}[Positive stable model structure on $\AlgO$ and $\LtO$]
\label{thm:positive_stable_AlgO}
Let $\capO$ be an operad in $\capR$-modules. Then the category of $\capO$-algebras (resp. left $\capO$-modules) has a model category structure with weak equivalences the stable equivalences (resp. objectwise stable equivalences) and fibrations the maps that are positive stable fibrations (resp. objectwise positive stable fibrations) in the underlying category of $\capR$-modules (Definition \ref{defn:lets_define_flat_model_structure}(b) and below Proposition \ref{prop:sending_cofibrations_to_monos}).
\end{thm}

 We defer the proof of the following two propositions to Subsection \ref{sec:flat_stable_cofibrations}.

\begin{prop}
\label{prop:cofibration}
Let $B\in\ModR^{\Sigma_t^\op}$ (resp. $B\in\SymSeq^{\Sigma_t^\op}$) and $t\geq 1$. If $\function{i}{X}{Y}$ is a cofibration between cofibrant objects in $\ModR$ (resp. $\SymSeq$) with the positive flat stable model structure, then 
\begin{itemize}
\item [(a)] $X^{\wedge t}\rarrow Y^{\wedge t}$ (resp. $X^{\tensorcheck t}\rarrow Y^{\tensorcheck t}$) is a cofibration between cofibrant objects in $\ModR^{\Sigma_t}$ (resp. $\SymSeq^{\Sigma_t}$) with the positive flat stable model structure, which is a weak equivalence if $i$ is a weak equivalence,
\item [(b)] the map $B\Smash_{\Sigma_t}Q_{t-1}^t\rarrow B\Smash_{\Sigma_t}Y^{\wedge t}$ (resp. $B\tensorcheck_{\Sigma_t}Q_{t-1}^t\rarrow B\tensorcheck_{\Sigma_t}Y^{\tensorcheck t}$) is a monomorphism.
\end{itemize}
\end{prop}

\begin{prop}
\label{prop:good_properties}
Let $G$ be a finite group and consider $\ModR$, $\ModR^G$, $\ModR^{G^\op}$, $\SymSeq$, $\SymSeq^G$, and $\SymSeq^{G^\op}$, each with the flat stable model structure.
\begin{itemize}
\item [(a)] If $B\in\ModR^{G^\op}$ (resp. $B\in\SymSeq^{G^\op}$), then the functor
\begin{align*}
  \functor{B\Smash_G -}{\ModR^G}{\ModR}
  \quad\quad
  \Bigl(\text{resp.}\quad
  \functor{B\tensorcheck_G -}{\SymSeq^G}{\SymSeq}
  \Bigr)
\end{align*}
preserves weak equivalences between cofibrant objects, and hence its total left derived functor exists.
\item [(b)] If $Z\in\ModR^G$ (resp. $Z\in\SymSeq^G$) is cofibrant, then the functor
\begin{align*}
  \functor{-\Smash_G Z}{\ModR^{G^\op}}{\ModR}
  \quad\quad
  \Bigl(\text{resp.}\quad
  \functor{-\tensorcheck_G Z}{\SymSeq^{G^\op}}{\SymSeq}
  \Bigr)
\end{align*}
preserves weak equivalences.
\end{itemize}
\end{prop}

\begin{prop}
\label{prop:homotopical_analysis_of_certain_pushouts}
Let $\capO$ be an operad in $\capR$-modules, $A\in\AlgO$ (resp. $A\in\LtO$), and $\function{i}{X}{Y}$ a generating acyclic cofibration in $\ModR$ (resp. $\SymSeq$) with the positive flat stable model structure. Consider any pushout diagram in $\AlgO$ (resp. $\LtO$) of the form \eqref{eq:small_arg_pushout_modules}. Then $j$ is a monomorphism and a weak equivalence in $\ModR$ (resp. $\SymSeq$).
\end{prop}

\begin{proof}
It suffices to consider the case of left $\capO$-modules. This is verified exactly as in \cite[proof of 4.4]{Harper_Spectra}, except using $(\ModR,\Smash,\capR)$ and Propositions \ref{prop:cofibration}, \ref{prop:good_properties} instead of $(\Spectra,\tensor_S,S)$ and \cite[4.28, 4.29]{Harper_Spectra}, respectively.
\end{proof}

\begin{proof}[Proof of Theorem~\ref{thm:positive_flat_stable_AlgO}]
Consider $\SymSeq$ and $\ModR$, both with the positive flat stable model structure. We will prove that the model structure on $\LtO$ (resp. $\AlgO$) is created by the middle (resp. left-hand) free-forgetful adjunction in \eqref{eq:free_forgetful_adjunction}. 

Define a map $f$ in $\LtO$ to be a weak equivalence (resp. fibration) if $U(f)$ is a weak equivalence (resp. fibration) in $\SymSeq$. Similarly, define a map $f$ in $\AlgO$ to be a weak equivalence (resp. fibration) if $U(f)$ is a weak equivalence (resp. fibration) in $\ModR$. Define a map $f$ in $\LtO$ (resp. $\AlgO$) to be a cofibration if it has the left lifting property with respect to all acyclic fibrations in $\LtO$ (resp. $\AlgO$).

Consider the case of $\LtO$. We want to verify the model category axioms (MC1)-(MC5) in \cite{Dwyer_Spalinski}. Arguing exactly as in \cite[proof of 1.1]{Harper_Spectra}, this reduces to the verification of Proposition \ref{prop:homotopical_analysis_of_certain_pushouts}. By construction, the model category is cofibrantly generated. Argue similarly for the case of $\AlgO$ by considering left $\capO$-modules concentrated at $0$.
\end{proof}

\begin{proof}[Proof of Theorem \ref{thm:positive_stable_AlgO}]
Consider $\SymSeq$ and $\ModR$, both with the positive stable model structure. We will prove that the model structure on $\LtO$ (resp. $\AlgO$) is created by the middle (resp. left-hand) free-forgetful adjunction in \eqref{eq:free_forgetful_adjunction}. 

Define a map $f$ in $\LtO$ to be a weak equivalence (resp. fibration) if $U(f)$ is a weak equivalence (resp. fibration) in $\SymSeq$. Similarly, define a map $f$ in $\AlgO$ to be a weak equivalence (resp. fibration) if $U(f)$ is a weak equivalence (resp. fibration) in $\ModR$. Define a map $f$ in $\LtO$ (resp. $\AlgO$) to be a cofibration if it has the left lifting property with respect to all acyclic fibrations in $\LtO$ (resp. $\AlgO$).

The model category axioms are verified exactly as in the proof of Theorem \ref{thm:positive_flat_stable_AlgO}; this reduces to the verification of Proposition \ref{prop:homotopical_analysis_of_certain_pushouts}.
\end{proof}

\subsection{Relations between homotopy categories}

The purpose of this subsection is to prove the following theorem. This generalization improves the corresponding result in \cite[1.4]{Harper_Spectra} from operads in symmetric spectra to the more general context involving operads in $\capR$-modules. It plays a key role in this paper.

\begin{thm}[Comparing homotopy categories]
\label{thm:comparing_homotopy_categories}
Let $\capO$ be an operad in $\capR$-modules and let $\AlgO$ (resp. $\LtO$) be the category of $\capO$-algebras (resp. left $\capO$-modules) with the model structure of Theorem \ref{thm:positive_flat_stable_AlgO} or \ref{thm:positive_stable_AlgO}. If $\function{f}{\capO}{\capO'}$ is a map of operads, then the adjunctions
$
\xymatrix@1{
  f_*\colon\Alg_{\capO}\ar@<0.5ex>[r] & 
  \Alg_{\capO'}:f^*\ar@<0.5ex>[l]
}
$ and 
$
\xymatrix@1{
  f_*\colon\Lt_{\capO}\ar@<0.5ex>[r] & 
  \Lt_{\capO'}:f^*\ar@<0.5ex>[l]
}
$
are Quillen adjunctions with left adjoints on top and $f^*$ the forgetful functor. If furthermore, $f$ is an objectwise stable equivalence, then the adjunctions are Quillen equivalences, and hence induce equivalences on the homotopy categories.
\end{thm}

First we make the following observation.

\begin{prop}
\label{prop:preserves_weak_equivalences}
Consider $\ModR$ and $\SymSeq$ with the positive flat stable model structure. If $W\in\ModR$ (resp. $W\in\SymSeq$) is cofibrant, then the functor
$
  \functor{-\circ(W)}{\SymSeq}{\ModR}
$ (resp. 
$ \functor{-\circ W}{\SymSeq}{\SymSeq}
$) preserves weak equivalences.
\end{prop}

\begin{proof}
It suffices to consider the case of symmetric sequences. This is verified exactly as in \cite[proof of 5.3]{Harper_Spectra}, except using $(\ModR,\Smash,\capR)$ and Propositions \ref{prop:cofibration}, \ref{prop:good_properties} instead of $(\Spectra,\tensor_S,S)$ and \cite[4.28, 4.29]{Harper_Spectra}, respectively.
\end{proof}

\begin{prop}
\label{prop:unit_map_is_weak_equivalence}
Let $\function{f}{\capO}{\capO'}$ be a map of operads in $\capR$-modules and consider $\AlgO$ (resp. $\LtO$) with the positive flat stable model structure. If $Z\in\AlgO$ (resp. $Z\in\LtO$) is cofibrant and $f$ is a weak equivalence in the underlying category $\SymSeq$ with the positive flat stable model structure, then the natural map $Z\rarrow f^*f_*Z$ is a weak equivalence in $\AlgO$ (resp. $\LtO$).
\end{prop}

\begin{proof}
It suffices to consider the case of left $\capO$-modules. This is verified exactly as in \cite[proof of 5.2]{Harper_Spectra}, except using $(\ModR,\Smash,\capR)$ and Propositions \ref{prop:cofibration}, \ref{prop:good_properties}, \ref{prop:preserves_weak_equivalences} instead of $(\Spectra,\tensor_S,S)$ and \cite[4.28, 4.29, 5.3]{Harper_Spectra}, respectively.
\end{proof}

\begin{proof}[Proof of Theorem \ref{thm:comparing_homotopy_categories}]
This is verified exactly as in \cite[proof of 1.4]{Harper_Spectra}, except using $(\ModR,\Smash,\capR)$ and Proposition \ref{prop:unit_map_is_weak_equivalence} instead of $(\Spectra,\tensor_S,S)$ and \cite[5.2]{Harper_Spectra}, respectively.
\end{proof}

\subsection{Homotopy colimits and simplicial bar constructions}

The following theorems play a key role in this paper. They improve the corresponding results in \cite{Harper_Bar} from operads in symmetric spectra to the more general context involving operads in $\capR$-modules, and are verified exactly as in the proof of \cite[1.10, 1.6, 1.8]{Harper_Bar}, respectively.

\begin{thm}
\label{thm:bar_calculates_derived_circle}
Let $\function{f}{\capO}{\capO'}$ be a morphism of operads in $\capR$-modules. Let $X$ be an $\capO$-algebra (resp. left $\capO$-module) and consider $\AlgO$ (resp. $\LtO$) with the model structure of Theorem \ref{thm:positive_flat_stable_AlgO} or \ref{thm:positive_stable_AlgO}. If the simplicial bar construction $\BAR(\capO,\capO,X)$ is objectwise cofibrant in $\AlgO$ (resp. $\LtO$), then there is a zigzag of weak equivalences
$
  \LL f_*(X)\wequiv
  |\BAR(\capO',\capO,X)|
$
in the underlying category, natural in such $X$. Here, $\LL f_*$ is the total left derived functor of $f_*$.
\end{thm}

\begin{thm}
\label{main_hocolim_theorem}
Let $\capO$ be an operad in $\capR$-modules. If $X$ is a simplicial $\capO$-algebra (resp. simplicial left $\capO$-module), then there are zigzags of weak equivalences
\begin{align*}
  U\hocolim\limits^{\AlgO}_{\Delta^\op}X & \wequiv |U X| \wequiv
  \hocolim\limits_{\Delta^\op}U X \\
  \Bigl(
  \text{resp.}\quad
  U\hocolim\limits^{\LtO}_{\Delta^\op}X & \wequiv |U X| \wequiv
  \hocolim\limits_{\Delta^\op}U X
  \Bigr)
\end{align*}
natural in $X$. Here, $U$ is the forgetful functor, $\sAlgO$ (resp. $\sLtO$) is equipped with the projective model structure inherited from the model structure of Theorem \ref{thm:positive_flat_stable_AlgO} or \ref{thm:positive_stable_AlgO}.
\end{thm}

\begin{thm}
\label{thm:fattened_replacement}
Let $\capO$ be an operad in $\capR$-modules. If $X$ is an $\capO$-algebra (resp. left $\capO$-module), then there is a zigzag of weak equivalences in $\AlgO$ (resp. $\LtO$)
\begin{align*}
  X \wequiv \hocolim\limits^{\AlgO}_{\Delta^\op}\BAR(\capO,\capO,X)
  \quad\quad
  \Bigl(
  \text{resp.}\quad
  X \wequiv \hocolim\limits^{\LtO}_{\Delta^\op}\BAR(\capO,\capO,X)
  \Bigr)
\end{align*}
natural in $X$. Here, $\sAlgO$ (resp. $\sLtO$) is equipped with the projective model structure inherited from the model structure of Theorem \ref{thm:positive_flat_stable_AlgO} or \ref{thm:positive_stable_AlgO}.
\end{thm}

\subsection{Flat stable cofibrations}
\label{sec:flat_stable_cofibrations}

The purpose of this subsection is to prove Propositions \ref{prop:cofibration} and \ref{prop:good_properties}. This requires several calculations (\ref{ex:calculation} and \ref{calculation_example}) together with a characterization of flat stable cofibrations (Proposition \ref{prop:cofibration_characterization}). This characterization is motivated by the characterization given in Schwede \cite{Schwede_book_project}, in terms of left $\capR_0$--modules, of flat stable cofibrations in $\ModR$.

Since $\capR$ is a commutative monoid in $(\Spectra,\tensor_S,S)$, it follows  that $\capR_0$ is a commutative monoid in $(\sSet_*,\Smash,S^0)$.  In particular, by \cite[2.4]{Harper_Modules} we can regard $\capR_0$ as a commutative monoid in $(\sSet_*^{\Sigma_n},\Smash,S^0)$ with the trivial $\Sigma_n$-action.

\begin{defn}
Let $n\geq 0$. A \emph{left $\capR_0$-module} is an object in $(\sSet_*^{\Sigma_n},\Smash,S^0)$ with a left action of $\capR_0$ and a \emph{morphism of left $\capR_0$-modules} is a map in $\sSet_*^{\Sigma_n}$ that respects the left $\capR_0$-module structure. Denote by $\capR_0-\sSet_*^{\Sigma_n}$ the category of left $\capR_0$-modules and their morphisms.
\end{defn}

For each $n\geq 0$, there is an adjunction
$
\xymatrix@1{
  \sSet_*^{\Sigma_n}\ar@<0.5ex>[r]^-{\capR_0\Smash-} & 
  \capR_0-\sSet_*^{\Sigma_n}\ar@<0.5ex>[l]
}
$
with left adjoint on top. It is proved in \cite{Shipley_comm_ring} that the following model category structure exists on left $\Sigma_n$-objects in pointed simplicial sets. 

\begin{defn}
\label{def:mixed_model_structure}
Let $n\geq 0$. 
\begin{itemize}
\item The \emph{mixed $\Sigma_n$-equivariant model structure} on $\sSet_*^{\Sigma_n}$ has weak equivalences the underlying weak equivalences of simplicial sets, cofibrations the retracts of (possibly transfinite) compositions of pushouts of maps
\begin{align*}
  \Sigma_n/H\cdot \partial\Delta[k]_+\rarrow 
  \Sigma_n/H\cdot \Delta[k]_+ \quad
  (k\geq 0,\ H\subsetof \Sigma_n\ \text{subgroup}),
\end{align*} 
and fibrations the maps with the right lifting property with respect to the acyclic cofibrations.
\end{itemize}
\end{defn}

Furthermore, it is proved in \cite{Shipley_comm_ring} that this model structure is cofibrantly generated by generating cofibrations and acyclic cofibrations with small domains, and that the cofibrations are the monomorphisms. It is easy to prove that the category $\capR_0-\sSet_*^{\Sigma_n}$ inherits a corresponding model structure created by the free-forgetful adjunction above Definition \ref{def:mixed_model_structure}, and that furthermore the diagram category of $(\Sigma_r^\op\times G)$-shaped diagrams in $\capR_0-\sSet^{\Sigma_n}_*$ appearing in the following proposition inherits a corresponding projective model structure. This proposition, whose proof is left to the reader, will be needed for identifying flat stable cofibrations in $\SymSeq^G$.

\begin{prop}\label{prop:mixed_diagram_model_structure}
Let $G$ be a finite group and consider any $n,r\geq 0$. The diagram category $\bigl(\capR_0-\sSet_*^{\Sigma_n}\bigr)^{\Sigma_r^\op\times G}$ inherits a corresponding model structure from the mixed $\Sigma_n$-equivariant model structure on $\sSet_*^{\Sigma_n}$. The weak equivalences (resp. fibrations) are the underlying weak equivalences (resp. fibrations) in $\sSet_*^{\Sigma_n}$.
\end{prop}

\begin{defn}
\label{defn:overline_capR}
Define $\ol{\capR}\in\ModR$ such that $\ol{\capR}_n:=\capR_n$ for $n\geq 1$ and $\ol{\capR}_0:=*$. The structure maps are the naturally occurring ones such that there exists a map of $\capR$-modules $\function{i}{\ol{\capR}}{\capR}$ satisfying $i_n=\id$ for each $n\geq 1$.
\end{defn}

The following calculation, which follows easily from \cite[2.9]{Harper_Spectra}, will be needed for characterizing flat stable cofibrations in $\SymSeq^G$.

\begin{calculation}
\label{ex:calculation}
Let $G$ be a finite group. Let $m,p\geq 0$, $H\subsetof \Sigma_m$ a subgroup, and $K$ a pointed simplicial set. Recall from \eqref{eq:adjunctions_stable_flat} the functors $G_p$ and $G^H_m$. Define $X:=G\cdot G_p(\capR\tensor G^H_m K)\in\SymSeq^G$. Here, $X$ is obtained by applying the indicated functors in \eqref{eq:adjunctions_stable_flat} to $K$. Then for $r=p$ we have
\begin{align*}
(\ol{\capR}\Smash X[\mathbf{r}])_n
  &\Iso
  \left\{
    \begin{array}{rl}
    G\cdot\bigl(\Sigma_n\cdot_{\Sigma_{n-m}\times\Sigma_m}\ol{\capR}_{n-m}
    \Smash(\Sigma_m/H\cdot K)\bigr)\cdot\Sigma_p & \text{for $n>m$,}\\
    *&\text{for $n\leq m$,}
    \end{array}
  \right.\\
  X[\mathbf{r}]_n 
  &\Iso
  \left\{
    \begin{array}{rl}
    G\cdot\bigl(\Sigma_n\cdot_{\Sigma_{n-m}\times\Sigma_m}\capR_{n-m}
    \Smash(\Sigma_m/H\cdot K)\bigr)\cdot\Sigma_p&\text{for $n>m$,}\\
    G\cdot\bigl(\capR_0\Smash(\Sigma_m/H\cdot K)\bigr)\cdot\Sigma_p&\text{for $n=m$,}\\
    *&\text{for $n<m$,}
    \end{array}
  \right.
\end{align*}
and for $r\neq p$ we have $X[\mathbf{r}]=*=\ol{\capR}\Smash X[\mathbf{r}]$.
\end{calculation}

The following characterization of flat stable cofibrations in $\SymSeq^G$ is motivated by the characterization given in Schwede \cite{Schwede_book_project} of flat stable cofibrations in $\ModR$. It improves the corresponding characterization given in \cite[6.6]{Harper_Spectra} from the context of $(\Spectra,\tensor_S,S)$ to the more general context of $(\ModR,\Smash,\capR)$.

\begin{prop}
\label{prop:cofibration_characterization}
Let $G$ be a finite group. 
\begin{itemize}
\item[(a)] A map $\function{f}{X}{Y}$ in $\SymSeq^G$ with the flat stable model structure is a cofibration if and only if the induced maps
\begin{align*}
  X[\mathbf{r}]_0\rarrow Y[\mathbf{r}]_0, &\quad r\geq 0,\ n=0,\\
  (\ol{\capR}\Smash Y[\mathbf{r}])_n\amalg_{(\ol{\capR}\Smash X[\mathbf{r}])_n}X[\mathbf{r}]_n\rarrow Y[\mathbf{r}]_n, &\quad r\geq 0,\ n\geq 1,
\end{align*}
are cofibrations in $\bigl(\capR_0-\sSet^{\Sigma_n}_*\bigr)^{\Sigma_r^\op \times G}$ with the model structure in \ref{prop:mixed_diagram_model_structure}.
\item[(b)] A map $\function{f}{X}{Y}$ in $\SymSeq^G$ with the positive flat stable model structure is a cofibration if and only if the maps
$
  X[\mathbf{r}]_0\rarrow Y[\mathbf{r}]_0
$, 
$r\geq 0$, are isomorphisms, and the induced maps
\begin{align*}
  (\ol{\capR}\Smash Y[\mathbf{r}])_n\amalg_{(\ol{\capR}\Smash X[\mathbf{r}])_n}X[\mathbf{r}]_n\rarrow Y[\mathbf{r}]_n, &\quad r\geq 0,\ n\geq 1,
\end{align*}
are cofibrations in $\bigl(\capR_0-\sSet^{\Sigma_n}_*\bigr)^{\Sigma_r^\op \times G}$ with the model structure in \ref{prop:mixed_diagram_model_structure}. 
\end{itemize} 
\end{prop}

\begin{proof}
This is verified exactly as in \cite[proof of 6.6]{Harper_Spectra}, except using $(\ModR,\Smash,\capR)$, Proposition \ref{prop:mixed_diagram_model_structure} and Calculation \ref{ex:calculation} instead of $(\Spectra,\tensor_S,S)$, \cite[6.3]{Harper_Spectra} and \cite[6.5]{Harper_Spectra}, respectively.
\end{proof}

\begin{proof}[Proof of Proposition \ref{prop:good_properties}]
It suffices to consider the case of symmetric sequences. Consider part (b). This is verified exactly as in \cite[proof of 4.29(b)]{Harper_Spectra}, except using $(\ModR,\Smash,\capR)$ and the map $g_*$ obtained by applying the indicated functors in \eqref{eq:adjunctions_stable_flat}, instead of $(\Spectra,\tensor_S,S)$ and the map $g_*$ obtained by applying the indicated functors in \cite[(4.1)]{Harper_Spectra}, respectively. Consider part (a). This is verified exactly as in \cite[proof of 4.29(a)]{Harper_Spectra}, except using $(\ModR,\Smash,\capR)$ instead of $(\Spectra,\tensor_S,S)$.
\end{proof}

\begin{prop}
\label{prop:cofibrations_to_mono}
Let $G$ be a finite group. If $B\in\ModR^{G^\op}$ (resp. $B\in\SymSeq^{G^\op}$), then the functor
$
  \functor{B\Smash_G -}{\ModR^G}{\ModR}
$ 
(resp. 
$  \functor{B\tensorcheck_G -}{\SymSeq^G}{\SymSeq}
$) sends cofibrations in $\ModR^G$ (resp. $\SymSeq^G$) with the flat stable model structure to monomorphisms.
\end{prop}

\begin{proof}
It suffices to consider the case of symmetric sequences. This is verified exactly as in \cite[proof of 6.11]{Harper_Spectra}, except using $(\ModR,\Smash,\capR)$ and the map $g_*$ obtained by applying the indicated functors in \eqref{eq:adjunctions_stable_flat}, instead of $(\Spectra,\tensor_S,S)$ and the map $g_*$ obtained by applying the indicated functors in \cite[(4.1)]{Harper_Spectra}, respectively.
\end{proof}

The following calculation, which follows easily from \cite[2.9]{Harper_Spectra} and \eqref{eq:tensor_check_calc}, will be needed in the proof of Proposition \ref{prop:cofibration} below.

\begin{calculation}
\label{calculation_example}
Let $k,m,p\geq 0$, $H\subsetof \Sigma_m$ a subgroup, and $t\geq 1$. Let the map $\function{g}{\partial\Delta[k]_+}{\Delta[k]_+}$ be a generating cofibration for $\sSet_*$ and define $X\rarrow Y$ in $\SymSeq$ to be the induced map
$
\xymatrix@1{
  g_*\colon G_p(\capR\tensor G^H_m\partial\Delta[k]_+)\ar[r] & 
  G_p(\capR\tensor G^H_m\Delta[k]_+)
}
$.
Here, the map $g_*$ is obtained by applying the indicated functors in \eqref{eq:adjunctions_stable_flat} to the map $g$. For $r=tp$ we have the calculation
\begin{align*}
  \bigl((Y^{\tensorcheck t})[\mathbf{r}]\bigr)_n 
  &\Iso
  \left\{
    \begin{array}{rl}
    \Sigma_n\cdot_{\Sigma_{n-tm}\times H^{\times t}}\capR_{n-tm}
    \Smash(\Delta[k]^{\times t})_+\cdot\Sigma_{tp}&\text{for $n>tm$,}\\
    \Sigma_{tm}\cdot_{H^{\times t}}
    \capR_0\Smash(\Delta[k]^{\times t})_+\cdot\Sigma_{tp}&\text{for $n=tm$,}\\
    *&\text{for $n<tm$,}
    \end{array}
  \right.\\
  \bigl(\ol{\capR}\Smash (Y^{\tensorcheck t})[\mathbf{r}]\bigr)_n
  &\Iso
  \left\{
    \begin{array}{rl}
    \Sigma_n\cdot_{\Sigma_{n-tm}\times H^{\times t}}\ol{\capR}_{n-tm}
    \Smash(\Delta[k]^{\times t})_+\cdot\Sigma_{tp}&\text{for $n>tm$,}\\
    *&\text{for $n\leq tm$,}
    \end{array}
  \right.
  \end{align*}
  \begin{align*}
  \bigl(Q_{t-1}^t[\mathbf{r}]\bigr)_n 
  &\Iso
  \left\{
    \begin{array}{rl}
    \Sigma_n\cdot_{\Sigma_{n-tm}\times H^{\times t}}\capR_{n-tm}
    \Smash\partial(\Delta[k]^{\times t})_+\cdot\Sigma_{tp}&\text{for $n>tm$,}\\
    \Sigma_{tm}\cdot_{H^{\times t}}
    \capR_0\Smash\partial(\Delta[k]^{\times t})_+\cdot\Sigma_{tp}&\text{for $n=tm$,}\\
    *&\text{for $n<tm$,}
    \end{array}
  \right.\\
  \bigl(\ol{\capR}\Smash Q_{t-1}^t[\mathbf{r}]\bigr)_n
  &\Iso
  \left\{
    \begin{array}{rl}
    \Sigma_n\cdot_{\Sigma_{n-tm}\times H^{\times t}}\ol{\capR}_{n-tm}
    \Smash\partial(\Delta[k]^{\times t})_+\cdot\Sigma_{tp}&\text{for $n>tm$,}\\
    *&\text{for $n\leq tm$,}
    \end{array}
  \right.
  \end{align*}
and for $r\neq tp$ we have $(Y^{\tensorcheck t})[\mathbf{r}]=*=\ol{\capR}\Smash (Y^{\tensorcheck t})[\mathbf{r}]$ and $Q_{t-1}^t[\mathbf{r}]=*=\ol{\capR}\Smash Q_{t-1}^t[\mathbf{r}]$.
\end{calculation}

\begin{proof}[Proof of Proposition \ref{prop:cofibration}]
It suffices to consider the case of symmetric sequences. Consider part (a). This is verified exactly as in \cite[proof of 4.28(a)]{Harper_Spectra}, except using $(\ModR,\Smash,\capR)$, the map $g_*$ obtained by applying the indicated functors in \eqref{eq:adjunctions_stable_flat}, Proposition \ref{prop:cofibration_characterization}, and Calculation \ref{calculation_example} instead of $(\Spectra,\tensor_S,S)$ the map $g_*$ obtained by applying the indicated functors in \cite[(4.1)]{Harper_Spectra}, \cite[6.6 and 6.15]{Harper_Spectra}, respectively. The acyclic cofibration assertion follows immediately from \cite[7.19]{Harper_Modules}. Consider part (b). This is verified exactly as in \cite[proof of 4.28(b)]{Harper_Spectra}, except using $(\ModR,\Smash,\capR)$ and Proposition \ref{prop:cofibrations_to_mono} instead of $(\Spectra,\tensor_S,S)$ and \cite[6.11]{Harper_Spectra}, respectively.
\end{proof}

The following will be needed in other sections of this paper.

\begin{prop}\label{prop:generating_cofibration}
Let $t\geq 1$. If $\function{i}{X}{Y}$ is a generating cofibration in $\ModR$ (resp. $\SymSeq$) with the positive flat stable model structure, then $Q_{t-1}^t\rarrow Y^{\wedge t}$ (resp. $Q_{t-1}^t\rarrow Y^{\tensorcheck t}$) is a cofibration between cofibrant objects in $\ModR^{\Sigma_t}$ (resp. $\SymSeq^{\Sigma_t}$) with the positive flat stable model structure.
\end{prop}

\begin{proof}
It suffices to consider the case of symmetric sequences. This follows immediately from the proof of Proposition \ref{prop:cofibration}.
\end{proof}

\section{Operads in chain complexes over a commutative ring}
\label{sec:chain_complexes_over_a_commutative_ring}

The purpose of this section is to observe that the main results of this paper remain true in the context of unbounded chain complexes over a commutative ring, provided that the desired model category structures exist on algebras (resp. left modules) over operads $\capO$ and $\tau_k\capO$. Since the constructions and proofs of the theorems are essentially identical to the arguments above in the context of $\capR$-modules, modulo the obvious changes, the arguments are left to the reader.

\begin{assumption}
\label{assumption:commutative_ring_notation}
From now on in this section, we assume that $\unit$ is any commutative ring.\end{assumption}

Denote by $(\Chaincx_\unit,\tensor,\unit)$ the closed symmetric monoidal category of unbounded chain complexes over $\unit$ \cite{Hovey, MacLane_homology}.

\begin{homotopical_assumption}
\label{HomotopicalAssumption_chain_complexes}
If $\capO$ is an operad in $\Chaincx_\unit$, assume that the following model structure exists on $\Alg_{\tilde{\capO}}$ (resp. $\Lt_{\tilde{\capO}}$) for $\tilde{\capO}=\capO$ and $\tilde{\capO}=\tau_k\capO$ for each $k\geq 1$: the model structure on $\Alg_{\tilde\capO}$ (resp. $\Lt_{\tilde{\capO}}$) has weak equivalences the homology isomorphisms (resp. objectwise homology isomorphisms) and fibrations the maps that are dimensionwise surjections (resp. objectwise dimensionwise surjections).
\end{homotopical_assumption}

\begin{cofibrancy}
\label{CofibrancyCondition_chain_complexes}
If $\capO$ is an operad in $\Chaincx_\unit$, consider the unit map $\function{\eta}{I}{\capO}$ of the operad $\capO$ and assume that  $I[\mathbf{r}]\rarrow\capO[\mathbf{r}]$ is a cofibration (\cite[3.1]{Harper_Bar}) between cofibrant objects in $\Chaincx_\unit^{\Sigma_r^\op}$ for each $r\geq 0$.
\end{cofibrancy}

If $\unit$ is any field of characteristic zero, then Homotopical Assumption \ref{HomotopicalAssumption_chain_complexes} and Cofibrancy Condition \ref{CofibrancyCondition_chain_complexes} are satisfied by every operad in $\Chaincx_\unit$ (see \cite{Harper_Modules, Hinich}). In the case of algebras over operads, if $\unit$ is any commutative ring and $\capO'$ is any non-$\Sigma$ operad in $\Chaincx_\unit$, then it is proved in \cite{Harper_Modules, Hinich} that the corresponding operad $\capO=\capO'\cdot\Sigma$ satisfies Homotopical Assumption \ref{HomotopicalAssumption_chain_complexes}.

The following is a commutative rings version of Definitions \ref{defn:homotopy_completion} and \ref{defn:quillen_homology}.

\begin{defn}
\label{defn:homotopy_completion_chain_complexes}
Let $\capO$ be an operad in $\Chaincx_\unit$ such that $\capO[\mathbf{0}]=*$. Assume that $\capO$ satisfies Homotopical Assumption \ref{HomotopicalAssumption_chain_complexes}. Let $X$ be an $\capO$-algebra (resp. left $\capO$-module). The \emph{homotopy completion} $X^\hwedge$ of $X$ is the $\capO$-algebra (resp. left $\capO$-module) defined by 
$
  X^\hwedge:=\holim^\AlgO_k \bigl(\tau_k\capO\circ_\capO (X^c)\bigr)
$ (resp. 
$
  X^\hwedge:=\holim^\LtO_k \bigl(\tau_k\capO\circ_\capO X^c\bigr)
$)
the homotopy limit of the completion tower of the functorial cofibrant replacement $X^c$ of $X$ in $\AlgO$ (resp. $\LtO$). The \emph{Quillen homology complex} (or Quillen homology object) $\QQ(X)$ of $X$ is the $\capO$-algebra $\tau_1\capO\circ^\HH_\capO (X)$ (resp. left $\capO$-module $\tau_1\capO\circ^\HH_\capO X$).
\end{defn}

The following is a commutative rings version of Theorem \ref{thm:finiteness}.

\begin{thm}
\label{thm:finiteness_commutative_rings}
Let $\capO$ be an operad in $\Chaincx_\unit$ such that $\capO[\mathbf{0}]$ is trivial. Assume that $\capO$ satisfies Homotopical Assumption \ref{HomotopicalAssumption_chain_complexes} and Cofibrancy Condition \ref{CofibrancyCondition_chain_complexes}. Let $X$ be a $0$-connected $\capO$-algebra (resp. left $\capO$-module) and assume that $\capO$ is $(-1)$-connected and $H_k\capO[\mathbf{r}],U\unit$ are finitely generated abelian groups for every $k,r$.
\begin{itemize}
\item[(a)] If the Quillen homology groups $H_k\QQ(X)$ (resp. $H_k\QQ(X)[\mathbf{r}]$) are finite for every $k,r$, then the homology groups $H_k X$ (resp. $H_k X[\mathbf{r}]$) are finite for every $k,r$. 
\item[(b)] If the Quillen homology groups $H_k\QQ(X)$ (resp. $H_k\QQ(X)[\mathbf{r}]$) are finitely generated abelian groups for every $k,r$, then the homology groups $H_k X$ (resp. $H_k X[\mathbf{r}]$) are finitely generated abelian groups for every $k,r$. 
\end{itemize}
Here, $U$ denotes the forgetful functor from commutative rings to abelian groups.
\end{thm}

The following is a commutative rings version of Theorem \ref{thm:hurewicz}.

\begin{thm}
\label{thm:hurewicz_commutative_rings}
Let $\capO$ be an operad in $\Chaincx_\unit$ such that $\capO[\mathbf{0}]$ is trivial. Assume that $\capO$ satisfies Homotopical Assumption \ref{HomotopicalAssumption_chain_complexes} and Cofibrancy Condition \ref{CofibrancyCondition_chain_complexes}. Let $X$ be a $0$-connected $\capO$-algebra (resp. left $\capO$-module), $n\geq 0$, and assume that $\capO$ is $(-1)$-connected.
\begin{itemize}
\item[(a)] The Quillen homology complex $\QQ(X)$ is $n$-connected if and only if $X$ is $n$-connected.
\item[(b)] If the Quillen homology complex $\QQ(X)$ is $n$-connected, then the natural Hurewicz map $H_k X\rarrow H_k\QQ(X)$ is an isomorphism for $k\leq 2n+1$ and a surjection for $k=2n+2$.
\end{itemize}
\end{thm}

The following is a commutative rings version of Theorem \ref{thm:relative_hurewicz}.

\begin{thm}
\label{thm:relative_hurewicz_commutative_rings}
Let $\capO$ be an operad in $\Chaincx_\unit$ such that $\capO[\mathbf{0}]$ is trivial. Assume that $\capO$ satisfies Homotopical Assumption \ref{HomotopicalAssumption_chain_complexes} and Cofibrancy Condition \ref{CofibrancyCondition_chain_complexes}. Let $\function{f}{X}{Y}$ be a map of $\capO$-algebras (resp. left $\capO$-modules) and $n\geq 0$. Assume that $\capO$ is $(-1)$-connected.
\begin{itemize}
\item[(a)] If $X,Y$ are $0$-connected, then $f$ is $n$-connected if and only if $f$ induces an $n$-connected map $\QQ(X)\rarrow\QQ(Y)$ on Quillen homology complexes.
\item[(b)] If $X,Y$ are $(-1)$-connected and $f$ is $(n-1)$-connected, then $f$ induces an $(n-1)$-connected map $\QQ(X)\rarrow\QQ(Y)$ on Quillen homology complexes.
\item[(c)] If $f$ induces an $n$-connected map $\QQ(X)\rarrow\QQ(Y)$ on Quillen homology complexes between $(-1)$-connected objects, then $f$ induces an $(n-1)$-connected map $X^\hwedge\rarrow Y^\hwedge$ on homotopy completion.
\item[(d)] If the Quillen homology complex $\QQ(X)$ is $(n-1)$-connected, then homotopy completion $X^\hwedge$ is $(n-1)$-connected.
\end{itemize}
Here, $\QQ(X)\rarrow\QQ(Y)$, $X^\hwedge\rarrow Y^\hwedge$ denote the natural induced zigzags in the category of $\capO$-algebras (resp. left $\capO$-modules) with all backward facing maps weak equivalences.
\end{thm}

The following is a commutative rings version of Theorem \ref{MainTheorem}.
\begin{thm}
\label{MainTheorem_commutative_rings}
Let $\capO$ be an operad in $\Chaincx_\unit$ such that $\capO[\mathbf{0}]$ is trivial. Assume that $\capO$ satisfies Homotopical Assumption \ref{HomotopicalAssumption_chain_complexes} and Cofibrancy Condition \ref{CofibrancyCondition_chain_complexes}. Let $\function{f}{X}{Y}$ be a map of $\capO$-algebras (resp. left $\capO$-modules). 
\begin{itemize}
\item[(a)] If $X$ is $0$-connected and $\capO$ is $(-1)$-connected, then the natural  coaugmentation $X\wequiv X^\hwedge$ is a weak equivalence.
\item[(b)] If the Quillen homology complex $\QQ(X)$ is $0$-connected and $\capO$ is $(-1)$-connected, then the homotopy completion spectral sequence
\begin{align*}
  E^1_{-s,t} &= H_{t-s}\Bigl(i_{s+1}\capO\circ^{\HH}_{\tau_1\capO}\bigl(\QQ(X)\bigr)\Bigr)
  \Longrightarrow
  H_{t-s}\bigl(X^{\hwedge}\bigr)\\
  \text{resp.}\quad
  E^1_{-s,t}[\mathbf{r}] &= H_{t-s}\Bigl(\bigl(i_{s+1}\capO\circ^{\HH}_{\tau_1\capO}\QQ(X)\bigr)[\mathbf{r}]\Bigr)
  \Longrightarrow
  H_{t-s}\bigl(X^{\hwedge}[\mathbf{r}]\bigr),\quad\text{$r\geq 0$},
\end{align*}
converges strongly.
\item[(c)] If $f$ induces a weak equivalence $\QQ(X)\wequiv\QQ(Y)$ on Quillen homology complexes, then $f$ induces a weak equivalence $X^\hwedge\wequiv Y^\hwedge$ on homotopy completion. 
\end{itemize}
\end{thm}

\bibliographystyle{plain}
\bibliography{HomotopyCompletion.bib}

\begin{thebibliography}{10}

\bibitem{Andre}
M.~Andr{\'e}.
\newblock {\em Homologie des alg\`ebres commutatives}.
\newblock Springer-Verlag, Berlin, 1974.
\newblock Die Grundlehren der mathematischen Wissenschaften, Band 206.

\bibitem{Arone_Ching}
G.~Arone and M.~Ching.
\newblock Operads and chain rules for the calculus of functors.
\newblock {\em Ast\'erisque}, (338):vi+158, 2011.

\bibitem{Arone_Kankaanrinta}
G.~Arone and M.~Kankaanrinta.
\newblock A functorial model for iterated {S}naith splitting with applications
  to calculus of functors.
\newblock In {\em Stable and unstable homotopy ({T}oronto, {ON}, 1996)},
  volume~19 of {\em Fields Inst. Commun.}, pages 1--30. Amer. Math. Soc.,
  Providence, RI, 1998.

\bibitem{Baker_Gilmour_Reinhard}
A.~Baker, H.~Gilmour, and P.~Reinhard.
\newblock Topological {A}ndr\'e-{Q}uillen homology for cellular commutative
  {$S$}-algebras.
\newblock {\em Abh. Math. Semin. Univ. Hambg.}, 78(1):27--50, 2008.

\bibitem{Baker_Richter}
A.~Baker and B.~Richter, editors.
\newblock {\em Structured ring spectra}, volume 315 of {\em London Mathematical
  Society Lecture Note Series}.
\newblock Cambridge University Press, Cambridge, 2004.

\bibitem{Basterra}
M.~Basterra.
\newblock Andr\'e-{Q}uillen cohomology of commutative {$S$}-algebras.
\newblock {\em J. Pure Appl. Algebra}, 144(2):111--143, 1999.

\bibitem{Basterra_Mandell}
M.~Basterra and M.~A. Mandell.
\newblock Homology and cohomology of {$E\sb \infty$} ring spectra.
\newblock {\em Math. Z.}, 249(4):903--944, 2005.

\bibitem{Basterra_Mandell_thh}
M.~Basterra and M.~A. Mandell.
\newblock Homology of {$E_n$} ring spectra and iterated {$THH$}.
\newblock {\em Algebr. Geom. Topol.}, 11(2):939--981, 2011.

\bibitem{Bousfield_Kan}
A.~K. Bousfield and D.~M. Kan.
\newblock {\em Homotopy limits, completions and localizations}.
\newblock Lecture Notes in Mathematics, Vol. 304. Springer-Verlag, Berlin,
  1972.

\bibitem{Carlsson_equivariant}
G.~Carlsson.
\newblock Equivariant stable homotopy and {S}ullivan's conjecture.
\newblock {\em Invent. Math.}, 103(3):497--525, 1991.

\bibitem{Carlsson}
G.~Carlsson.
\newblock Derived completions in stable homotopy theory.
\newblock {\em J. Pure Appl. Algebra}, 212(3):550--577, 2008.

\bibitem{Chataur_Rodriguez_Scherer}
D.~Chataur, J.~L. Rodr{\'{\i}}guez, and J.~Scherer.
\newblock Realizing operadic plus-constructions as nullifications.
\newblock {\em $K$-Theory}, 33(1):1--21, 2004.

\bibitem{Ching_duality}
M.~Ching.
\newblock Bar-cobar duality for operads in stable homotopy theory.
\newblock {\em J. Topol.}, 5(1):39--80, 2012.

\bibitem{Dugger_Isaksen}
D.~Dugger and D.~C. Isaksen.
\newblock Topological hypercovers and {$\Bbb A\sp 1$}-realizations.
\newblock {\em Math. Z.}, 246(4):667--689, 2004.

\bibitem{Dwyer_strong_convergence}
W.~G. Dwyer.
\newblock Strong convergence of the {E}ilenberg-{M}oore spectral sequence.
\newblock {\em Topology}, 13:255--265, 1974.

\bibitem{Dwyer_Greenlees_Iyengar}
W.~G. Dwyer, J.~P.~C. Greenlees, and S.~Iyengar.
\newblock Duality in algebra and topology.
\newblock {\em Adv. Math.}, 200(2):357--402, 2006.

\bibitem{Dwyer_Spalinski}
W.~G. Dwyer and J.~Spali{\'n}ski.
\newblock Homotopy theories and model categories.
\newblock In {\em Handbook of algebraic topology}, pages 73--126.
  North-Holland, Amsterdam, 1995.

\bibitem{EKMM}
A.~D. Elmendorf, I.~Kriz, M.~A. Mandell, and J.~P. May.
\newblock {\em Rings, modules, and algebras in stable homotopy theory},
  volume~47 of {\em Mathematical Surveys and Monographs}.
\newblock American Mathematical Society, Providence, RI, 1997.
\newblock With an appendix by M. Cole.

\bibitem{Elmendorf_Mandell}
A.~D. Elmendorf and M.~A. Mandell.
\newblock Rings, modules, and algebras in infinite loop space theory.
\newblock {\em Adv. Math.}, 205(1):163--228, 2006.

\bibitem{Fresse_lie_theory}
B.~Fresse.
\newblock Lie theory of formal groups over an operad.
\newblock {\em J. Algebra}, 202(2):455--511, 1998.

\bibitem{Fresse}
B.~Fresse.
\newblock Koszul duality of operads and homology of partition posets.
\newblock In {\em Homotopy theory: relations with algebraic geometry, group
  cohomology, and algebraic $K$-theory}, volume 346 of {\em Contemp. Math.},
  pages 115--215. Amer. Math. Soc., Providence, RI, 2004.

\bibitem{Fresse_modules}
B.~Fresse.
\newblock {\em Modules over operads and functors}, volume 1967 of {\em Lecture
  Notes in Mathematics}.
\newblock Springer-Verlag, Berlin, 2009.

\bibitem{Ginzburg_Kapranov}
V.~Ginzburg and M.~Kapranov.
\newblock Koszul duality for operads.
\newblock {\em Duke Math. J.}, 76(1):203--272, 1994.

\bibitem{Goerss_f2_algebras}
P.~G. Goerss.
\newblock On the {A}ndr\'e-{Q}uillen cohomology of commutative {${\bf F}\sb
  2$}-algebras.
\newblock {\em Ast\'erisque}, (186):169, 1990.

\bibitem{Goerss_Hopkins}
P.~G. Goerss and M.~J. Hopkins.
\newblock Andr\'e-{Q}uillen (co)-homology for simplicial algebras over
  simplicial operads.
\newblock In {\em Une d\'egustation topologique [Topological morsels]: homotopy
  theory in the Swiss Alps (Arolla, 1999)}, volume 265 of {\em Contemp. Math.},
  pages 41--85. Amer. Math. Soc., Providence, RI, 2000.

\bibitem{Goerss_Hopkins_moduli_spaces}
P.~G. Goerss and M.~J. Hopkins.
\newblock Moduli spaces of commutative ring spectra.
\newblock In {\em Structured ring spectra}, volume 315 of {\em London Math.
  Soc. Lecture Note Ser.}, pages 151--200. Cambridge Univ. Press, Cambridge,
  2004.

\bibitem{Goerss_Jardine}
P.~G. Goerss and J.~F. Jardine.
\newblock {\em Simplicial homotopy theory}, volume 174 of {\em Progress in
  Mathematics}.
\newblock Birkh\"auser Verlag, Basel, 1999.

\bibitem{Goerss_Schemmerhorn}
P.~G. Goerss and K.~Schemmerhorn.
\newblock Model categories and simplicial methods.
\newblock In {\em Interactions between homotopy theory and algebra}, volume 436
  of {\em Contemp. Math.}, pages 3--49. Amer. Math. Soc., Providence, RI, 2007.

\bibitem{Goodwillie_calc2}
T.~G. Goodwillie.
\newblock Calculus. {II}. {A}nalytic functors.
\newblock {\em $K$-Theory}, 5(4):295--332, 1991/92.

\bibitem{Goodwillie_calc3}
T.~G. Goodwillie.
\newblock Calculus. {III}. {T}aylor series.
\newblock {\em Geom. Topol.}, 7:645--711 (electronic), 2003.

\bibitem{Harper_Spectra}
J.~E. Harper.
\newblock Homotopy theory of modules over operads in symmetric spectra.
\newblock {\em Algebr. Geom. Topol.}, 9(3):1637--1680, 2009.

\bibitem{Harper_Bar}
J.~E. Harper.
\newblock Bar constructions and {Q}uillen homology of modules over operads.
\newblock {\em Algebr. Geom. Topol.}, 10(1):87--136, 2010.

\bibitem{Harper_Modules}
J.~E. Harper.
\newblock Homotopy theory of modules over operads and non-{$\Sigma$} operads in
  monoidal model categories.
\newblock {\em J. Pure Appl. Algebra}, 214(8):1407--1434, 2010.

\bibitem{Hess}
K.~Hess.
\newblock {A general framework for homotopic descent and codescent}.
\newblock {\em {\\ \tt arXiv:1001.1556v3 [math.AT]}}, 2010.

\bibitem{Hinich}
V.~Hinich.
\newblock Homological algebra of homotopy algebras.
\newblock {\em Comm. Algebra}, 25(10):3291--3323, 1997.
\newblock Erratum: {\tt{arXiv:math/0309453 [math.QA]}}.

\bibitem{Hirschhorn}
P.~S. Hirschhorn.
\newblock {\em Model categories and their localizations}, volume~99 of {\em
  Mathematical Surveys and Monographs}.
\newblock American Mathematical Society, Providence, RI, 2003.

\bibitem{Hornbostel}
J.~Hornbostel.
\newblock Preorientations of the derived motivic multiplicative group.
\newblock {\em \\ {\tt arXiv:1005.4546 [math.KT]}}, 2011.

\bibitem{Hovey}
M.~Hovey.
\newblock {\em Model categories}, volume~63 of {\em Mathematical Surveys and
  Monographs}.
\newblock American Mathematical Society, Providence, RI, 1999.

\bibitem{Hovey_Shipley_Smith}
M.~Hovey, B.~Shipley, and J.~H. Smith.
\newblock Symmetric spectra.
\newblock {\em J. Amer. Math. Soc.}, 13(1):149--208, 2000.

\bibitem{Jardine_generalized_etale}
J.~F. Jardine.
\newblock {\em Generalized \'etale cohomology theories}, volume 146 of {\em
  Progress in Mathematics}.
\newblock Birkh\"auser Verlag, Basel, 1997.

\bibitem{Johnson_McCarthy}
B.~Johnson and R.~McCarthy.
\newblock Deriving calculus with cotriples.
\newblock {\em Trans. Amer. Math. Soc.}, 356(2):757--803 (electronic), 2004.

\bibitem{Kriz_May}
I.~Kriz and J.~P. May.
\newblock Operads, algebras, modules and motives.
\newblock {\em Ast\'erisque}, (233):iv+145pp, 1995.

\bibitem{Kuhn}
N.~J. Kuhn.
\newblock Localization of {A}ndr\'e-{Q}uillen-{G}oodwillie towers, and the
  periodic homology of infinite loopspaces.
\newblock {\em Adv. Math.}, 201(2):318--378, 2006.

\bibitem{Kuhn_survey}
N.~J. Kuhn.
\newblock Goodwillie towers and chromatic homotopy: an overview.
\newblock In {\em Proceedings of the {N}ishida {F}est ({K}inosaki 2003)},
  volume~10 of {\em Geom. Topol. Monogr.}, pages 245--279. Geom. Topol. Publ.,
  Coventry, 2007.

\bibitem{Lawson}
T.~Lawson.
\newblock The plus-construction, {B}ousfield localization, and derived
  completion.
\newblock {\em J. Pure Appl. Algebra}, 214(5):596--604, 2010.

\bibitem{Lazarev}
A.~Lazarev.
\newblock Cohomology theories for highly structured ring spectra.
\newblock In {\em Structured ring spectra}, volume 315 of {\em London Math.
  Soc. Lecture Note Ser.}, pages 201--231. Cambridge Univ. Press, Cambridge,
  2004.

\bibitem{Livernet_thesis}
M.~Livernet.
\newblock {\em Homotopie rationnelle des alg\`ebres sur une op\'erade}.
\newblock PhD thesis, Universit\'e Louis Pasteur (Strasbourg I), 1998.
\newblock \\ Available at {\verb=http://www.math.univ-paris13.fr/~livernet/=}.

\bibitem{Livernet}
M.~Livernet.
\newblock On a plus-construction for algebras over an operad.
\newblock {\em $K$-Theory}, 18(4):317--337, 1999.

\bibitem{MacLane_homology}
S.~Mac~Lane.
\newblock {\em Homology}.
\newblock Classics in Mathematics. Springer-Verlag, Berlin, 1995.
\newblock Reprint of the 1975 edition.

\bibitem{MacLane_categories}
S.~Mac~Lane.
\newblock {\em Categories for the working mathematician}, volume~5 of {\em
  Graduate Texts in Mathematics}.
\newblock Springer-Verlag, New York, second edition, 1998.

\bibitem{Mandell}
M.~A. Mandell.
\newblock {$E\sb \infty$} algebras and {$p$}-adic homotopy theory.
\newblock {\em Topology}, 40(1):43--94, 2001.

\bibitem{Mandell_TAQ}
M.~A. Mandell.
\newblock Topological {A}ndr\'e-{Q}uillen cohomology and {$E_\infty$}
  {A}ndr\'e-{Q}uillen cohomology.
\newblock {\em Adv. Math.}, 177(2):227--279, 2003.

\bibitem{Mandell_May_Schwede_Shipley}
M.~A. Mandell, J.~P. May, S.~Schwede, and B.~Shipley.
\newblock Model categories of diagram spectra.
\newblock {\em Proc. London Math. Soc. (3)}, 82(2):441--512, 2001.

\bibitem{McCarthy_Minasian_preprint}
R.~McCarthy and V.~Minasian.
\newblock {On triples, operads, and generalized homogeneous functors}.
\newblock {\em {\tt arXiv:math/0401346v1 [math.AT]}}, 2004.

\bibitem{McClure_Schwanzl_Vogt}
J.~E. McClure, R.~Schw{\"a}nzl, and R.~Vogt.
\newblock {$THH(R)\cong R\otimes S^1$} for {$E_\infty$} ring spectra.
\newblock {\em J. Pure Appl. Algebra}, 121(2):137--159, 1997.

\bibitem{McClure_Smith_conjecture}
J.~E. McClure and J.~H. Smith.
\newblock A solution of {D}eligne's {H}ochschild cohomology conjecture.
\newblock In {\em Recent progress in homotopy theory (Baltimore, MD, 2000)},
  volume 293 of {\em Contemp. Math.}, pages 153--193. Amer. Math. Soc.,
  Providence, RI, 2002.

\bibitem{Miller}
H.~R. Miller.
\newblock The {S}ullivan conjecture on maps from classifying spaces.
\newblock {\em Ann. of Math. (2)}, 120(1):39--87, 1984.
\newblock Correction: \emph{Ann. of Math. (2)}, 121(3):605-609, 1985.

\bibitem{Minasian}
V.~Minasian.
\newblock Andr\'e-{Q}uillen spectral sequence for {$THH$}.
\newblock {\em Topology Appl.}, 129(3):273--280, 2003.

\bibitem{Quillen}
D.~Quillen.
\newblock {\em Homotopical algebra}.
\newblock Lecture Notes in Mathematics, No. 43. Springer-Verlag, Berlin, 1967.

\bibitem{Quillen_rings}
D.~Quillen.
\newblock On the (co-) homology of commutative rings.
\newblock In {\em Applications of Categorical Algebra (Proc. Sympos. Pure
  Math., Vol. XVII, New York, 1968)}, pages 65--87. Amer. Math. Soc.,
  Providence, R.I., 1970.

\bibitem{Rezk}
C.~Rezk.
\newblock {\em Spaces of Algebra Structures and Cohomology of Operads}.
\newblock PhD thesis, MIT, 1996.
\newblock Available at {\verb=http://www.math.uiuc.edu/~rezk/=}.

\bibitem{Richter}
B.~Richter.
\newblock An {A}tiyah-{H}irzebruch spectral sequence for topological
  {A}ndr\'e-{Q}uillen homology.
\newblock {\em J. Pure Appl. Algebra}, 171(1):59--66, 2002.

\bibitem{Rognes_topological_Galois}
J.~Rognes.
\newblock Galois extensions of structured ring spectra. {S}tably dualizable
  groups.
\newblock {\em Mem. Amer. Math. Soc.}, 192(898):viii+137, 2008.

\bibitem{Rognes_logarithmic}
J.~Rognes.
\newblock Topological logarithmic structures.
\newblock In {\em New topological contexts for {G}alois theory and algebraic
  geometry ({BIRS} 2008)}, volume~16 of {\em Geom. Topol. Monogr.}, pages
  401--544. Geom. Topol. Publ., Coventry, 2009.

\bibitem{Schwede_cotangent}
S.~Schwede.
\newblock Spectra in model categories and applications to the algebraic
  cotangent complex.
\newblock {\em J. Pure Appl. Algebra}, 120(1):77--104, 1997.

\bibitem{Schwede}
S.~Schwede.
\newblock {$S$}-modules and symmetric spectra.
\newblock {\em Math. Ann.}, 319(3):517--532, 2001.

\bibitem{Schwede_algebraic}
S.~Schwede.
\newblock Stable homotopy of algebraic theories.
\newblock {\em Topology}, 40(1):1--41, 2001.

\bibitem{Schwede_book_project}
S.~Schwede.
\newblock {\em An untitled book project about symmetric spectra}.
\newblock 2007,2009.
\newblock Available at: {\verb=http://www.math.uni-bonn.de/people/schwede/=}.

\bibitem{Schwede_homotopy_groups}
S.~Schwede.
\newblock On the homotopy groups of symmetric spectra.
\newblock {\em Geom. Topol.}, 12(3):1313--1344, 2008.

\bibitem{Schwede_Shipley}
S.~Schwede and B.~Shipley.
\newblock Algebras and modules in monoidal model categories.
\newblock {\em Proc. London Math. Soc. (3)}, 80(2):491--511, 2000.

\bibitem{Shipley_comm_ring}
B.~Shipley.
\newblock A convenient model category for commutative ring spectra.
\newblock In {\em Homotopy theory: relations with algebraic geometry, group
  cohomology, and algebraic $K$-theory}, volume 346 of {\em Contemp. Math.},
  pages 473--483. Amer. Math. Soc., Providence, RI, 2004.

\bibitem{Sullivan_MIT_notes}
D.~Sullivan.
\newblock {\em Geometric topology. {P}art {I}}.
\newblock Massachusetts Institute of Technology, Cambridge, Mass., 1971.
\newblock Localization, periodicity, and Galois symmetry, Revised version.

\bibitem{Sullivan_genetics}
D.~Sullivan.
\newblock Genetics of homotopy theory and the {A}dams conjecture.
\newblock {\em Ann. of Math. (2)}, 100:1--79, 1974.

\bibitem{Turner}
J.~M. Turner.
\newblock On simplicial commutative algebras with vanishing {A}ndr\'e-{Q}uillen
  homology.
\newblock {\em Invent. Math.}, 142(3):547--558, 2000.

\bibitem{Weibel}
C.~A. Weibel.
\newblock {\em An introduction to homological algebra}, volume~38 of {\em
  Cambridge Studies in Advanced Mathematics}.
\newblock Cambridge University Press, Cambridge, 1994.

\end{thebibliography}

\end{document}